\documentclass[11pt]{amsart}
\usepackage{fullpage}

\makeatletter
\def\do@rerunfilecheck{} 
\makeatother
\usepackage[centertags]{amsmath}
\usepackage{mathtools}
\usepackage{amssymb}
\usepackage{amsthm}
\usepackage{pstricks}
\usepackage{epsfig}
\usepackage{pst-grad} 
\usepackage{pst-plot} 
\usepackage{wrapfig}
\usepackage{pgfplots}
 \usepgfplotslibrary{fillbetween}
 \pgfplotsset{compat=1.12}
\usepackage{amsfonts}
\usepackage{graphicx}
\usepackage[margin=0.6in]{geometry}
\usepackage{tikz-layers}
\usepackage[normalem]{ulem}
\usepackage[utf8]{inputenc}
\usepackage{amsmath}
\usepackage{amssymb}
\usepackage{bm}
\usepackage{float}
\usepackage{comment}
\usepackage{tikz}
\usepackage{caption}
\usepackage{subcaption}
\usetikzlibrary{positioning}
\usetikzlibrary{decorations.markings}
\usepackage{pgfplots}
\usepackage{pst-node}       
\usepackage{tikz-cd} 
\usepackage{hyperref}
\usepackage{soul}
\usepackage[capitalise,noabbrev]{cleveref}

\pgfplotsset{compat=newest}
\usepgfplotslibrary{fillbetween}

\newcommand\phantomarrow[2]{%
  \setbox0=\hbox{$\displaystyle #1\to$}%
  \hbox to \wd0{%
    $#2\mapstochar
     \cleaders\hbox{$\mkern-1mu\relbar\mkern-3mu$}\hfill
     \mkern-7mu\rightarrow$}%
  \,}

\newtheorem{theo}{Theorem}[section]
\newtheorem{coro}[theo]{Corollary}
\newtheorem{lemma}[theo]{Lemma}

\newtheorem{rem}[theo]{Remark}

\newtheorem{propo}[theo]{Proposition}
\newtheorem{defi}[theo]{Definition}

\Crefname{propo}{Proposition}{Propositions}
\Crefname{theo}{Theorem}{Theorems}
\Crefname{coro}{Corollary}{Corollaries}
\Crefname{rem}{Remark}{Remarks}

\newtheorem{theoreml}{Theorem}

\newcommand{\R}{{\mathbb R}}
\newcommand{\C}{{\mathbb C}}

\newcommand{\Z}{{\mathbb Z}}

\newcommand{\N}{{\mathbb N}}

\newcommand{\T}{{\mathbb T}}
\newcommand{\Pic}{\mathrm{Pic}}
\newcommand{\CP}{\C{\mathbb P}}
\newcommand{\mcI}{\mathcal{I}}
\newcommand{\mcC}{\mathcal{C}}

\newcommand{\al}{\alpha}
\newcommand{\be}{\beta}

\newcommand{\ga}{\gamma}

\newcommand{\la}{\lambda}

\newcommand{\de}{\delta}

\newcommand{\om}{\omega}

\newcommand{\ze}{\zeta}

\newcommand{\rX}{\mathrm{X}}
\newcommand{\rV}{\mathrm{V}}
\newcommand{\rW}{\mathrm{W}}

\newcommand{\mupq}{\mu_{p,q}}
\newcommand{\mupo}{\mu_{p,1}}
\newcommand{\suml}{\mathop{\sum}\limits}
\newcommand{\Kh}{K^\textit{h}}

\newcommand{\ri}{\mathrm{i}}
\newcommand\restr[2]{{
  \left.\kern-\nulldelimiterspace 
  #1 
  \vphantom{\big|} 
  \right|_{#2} 
  }}
\title{Complex dynamics perspective for birational maps of the plane arising from cluster algebra mutations}
\author{Andrei Grigorev}
\address{Andrei Grigorev: Department of Mathematics\\ IU Indianapolis\\ 402 North Blackford Street room LD270\\Indianapolis, IN 46202-3267, USA;\newline  National Research University Higher School of Economics\\ Department of Mathematics\\ Russian Federation, 6 Usacheva st., Moscow 119048}
\email{aagrigor@iu.edu, andrei.al.grigorev@gmail.com}

\author{Krishna Chaitanya Kalidindi}
\address{Krishna Chaitanya Kalidindi: Department of Mathematics\\ IU Indianapolis\\ 402 North Blackford Street room LD270\\Indianapolis, IN 46202-3267, USA}
\email{kkalidin@iu.edu}

\author{Andres Quintero Santander}
\address{Andres Quintero Santander: Department of Mathematics\\ IU Indianapolis\\ 402 North Blackford Street room LD270\\Indianapolis, IN 46202-3267, USA}
\email{aequinte@iu.edu}

\author{Roland Roeder}
\address{Roland Roeder: Department of Mathematics\\ IU Indianapolis\\ 402 North Blackford Street room LD270\\Indianapolis, IN 46202-3267, USA}
\email{roederr@iu.edu}

\date{\today}

\begin{document}

\begin{abstract}
Using the methods of holomorphic dynamics we investigate planar birational mappings that arise from the theory of cluster algebras and integrable systems.   Computing dynamical degrees of these mappings, many of which are greater than one, allows us to show that many of the mappings do not have a conserved quantity (nor an invariant fibration).   In most of the examples, invariant fibrations 
can also be ruled out by finding superattracting periodic points.  This answers a question posted by  Machacek and Ovenhouse \cite{machacek2024discrete} and by Chen and Li \cite{CHEN202425}.  Moreover, having found a good algebraically stable model for the mappings and having computed the dynamical degree, we can then apply results from the ergodic theory of birational maps to produce invariant measures with positive entropy and positive Lyapunov exponents.
\end{abstract}
\maketitle
\section{Introduction.}

We study the dynamics of a family of rational mappings $\mu_{p,q}:\C^2\to \C^2$ arising from the theory of cluster algebras.   Here $p$ and $q$ are non-zero integers and $\mu_{p,q}$ is defined as $\mu_{p,q}= \mu_{p,q}^{(2)}\circ \mu_{p,q}^{(1)}$, where
\begin{equation}\label{eq:mutations}
    \mu_{p,q}^{(1)}(x,y) = \left(\frac{1+y^q}{x},y\right),\;\;\mu_{p,q}^{(2)}(x,y) = \left(x,\frac{1+x^p}{y}\right).
\end{equation}
Note that both $\mu_{p,q}^{(1)}(x,y)$ and $\mu_{p,q}^{(2)}(x,y)$ are involutions and therefore $\mu_{p,q}$ is a birational mapping of the complex plane with inverse $\mu_{p,q}^{-1}$ birationally conjugate (but not equal) to $\mu_{q,p}$ (See $\cref{rem:pgeqq}$).

Cluster algebras are certain commutative algebras introduced in the work of Fomin and Zelevinsky \cite{FOMIN1}.
 One of the main notions in the theory of cluster algebras is the notion of a {\em cluster seed} and {\em mutation}. 
A seed is a pair consisting of a skew-symmetrizable $n\times n$ matrix $B = (b_{ij})_{i,j=1,\dots, n}$ with integer elements and a tuple of variables $\boldsymbol{A} = ({A}_1,\dots, {A}_n)$.
A mutation $\mu_i$ for each $i = 1,\dots,n$ is a transformation between seeds that maps the seed $(B,\boldsymbol{A})$ to a seed $(\mu_i(B), \mu_i(\boldsymbol{A}))=(\tilde{B},\tilde{\boldsymbol{A}})$ defined by
\begin{align}\label{EQN:BJK}
\tilde{b}_{jk} = \begin{cases}
-b_{jk},\; i\in\{j,k\}\\
b_{jk}+[ b_{ij}]_{+}[ b_{ki}]_{+}-[ -b_{ij}]_{+}[ -b_{ki}]_{+} 
\end{cases},\;\;
\tilde{A}_j = \begin{cases} A_j,\;\hfill j\neq i\\
A_{i}^{-1}\left(\displaystyle\prod_{j=1}^{n}A_j^{[ b_{ij}]_{+}} +\displaystyle\prod_{j=1}^{n}A_j^{[ -b_{ij}]_{+}}\right), \qquad j=i.\end{cases}
\end{align}
Here for any $b\in \Z$ we denote $[b]_{+} = \max(b,0)$.

Mutation transformations are involutions. Taking some initial seed and considering the result of application of a composition of multiple different mutations that preserve the matrix $B$ gives a non-trivial birational dynamics. The examples of such dynamics include certain periodic maps for the cluster seeds associated with root systems, see \cite{FOMIN2} and discrete Painlev\'e equations for cluster seeds associated with Newton polygons with one internal point, see \cite{bershtein2018cluster}. 

In our work we study in particular the dynamics appearing as the composition of the two mutations $\mu=\mu_2\circ \mu_1$ for $B = \begin{pmatrix}
    0 & p\\ -q & 0
\end{pmatrix}$
for $p,q\in \Z_{> 0}$.  Setting $A_1 = x, A_2 = y$ the mappings $\mu_1$ and $\mu_2$ given by \eqref{EQN:BJK} correspond to the mappings 
$\mu_{p,q}^{(1)}$ and $\mu_{p,q}^{(2)}$ from \eqref{eq:mutations}, respectively.
The dynamics of these mappings and search for invariant quantities were recently studied by Machacek and Ovenhouse \cite{machacek2024discrete} and by Chen and Li \cite{CHEN202425}.  While these two papers have somewhat different perspectives, each of them asks about whether the mappings $\mu_{p,q}$ have a conserved quantity.     

The cases $pq<4$ are related to the finite root systems $A_2, B_2,G_2$ and in this case the dynamics $\mu_{p,q}$ is periodic.  The cases $pq = 4$ are integrable, in the sense that the have conserved quantities. The conserved quantities for them were found in the work \cite{CHEN202425}. We show that the dynamics on a generic fiber is conjugated to a linear fractional transformation. (For a further discussion of quantitites conserved under cluster algebra mutations see \cite{LAMPE,KAUFMAN,HoneKouloukas} and the references therein.)  
Further we will show that in the cases $pq>4$ the dynamics $\mu_{p,q}$ is not integrable, at least in the context of holomorphic dynamics (Theorems \ref{THM:POSPQ_NO_FIBR} and~\ref{THM:NEGPQ_NO_FIBR}, below.)

An interesting feature of the mappings $\mu_{p,q}$ is that they preserve the meromorphic two-form 
\begin{align}
    \eta = \frac{dx \wedge dy}{xy}
\end{align}
under pullback.   In other words, $\mu_{p,q}^* \eta = \eta$.   A more general study of rational self-maps of the complex plane that preserve the meromorphic form $\eta$ (up to a factor) under pullback was done in the two recent papers of Diller and Roeder \cite{DR1,DR2}, building on works of Bell-Diller-Jonsson \cite{BDJ} and Diller-Lin \cite{diller_lin_2016}.   

The goal of this paper is to investigate the mappings $\mu_{p,q}$ from the perspective of holomorphic dynamics. 
Indeed, the holomorphic dynamics of birational mappings of the complex projective plane is quite well-developed (see for example \cite{DillerDynamicsMeromorphicMaps2009b,DUJARDIN_LAMINAR,diller_favre_2001,dillerInvariantMeasureLyapunov2001,MR1422105} and the references therein) and it will allow us to say some interesting things about these mappings, including the aformentioned claim about maps with $pq > 4$ not being integrable.

We will need to develop some terminology in order to state our main results.  We will typically compactify $\mathbb{C}^2$ by the complex projective plane $\mathbb{CP}^2$ and sometimes by more complicated rational complex surfaces.   Any rational self-mapping $f:\mathbb{CP}^2 \dashrightarrow \mathbb{CP}^2$ can be expressed in homogeneous coordinates $[X:Y:Z]$ as a triple 
\begin{align*}
    f([X:Y:Z]) = [f_1(X,Y,Z):f_2(X,Y,Z):f_3(X,Y,Z)],
\end{align*}
where $f_1, f_2$, and $f_3$ are homogeneous polynomials of the same degree that are chosen so that there is no common factor of positive degree.   With such a choice the common degree of these polynomials is called the {\em algebraic degree} of $f$ and denoted $d_{\rm alg}(f)$.    One can also associate to any such $f$ the {\em topological degree} $d_{\rm top}(f)$ which is the number of preimages of a sufficiently generic point in $\mathbb{CP}^2$ under $f$.

It is important to note that $d_{\rm alg}(f)$ is not invariant under birational conjugacy and also that it does not always behave well under iteration; i.e. one may have that $d_{\rm alg}(f^n) < (d_{\rm alg}(f))^n$ for some natural number $n \geq 2$.   (Here $f^n$ denotes the $n$-th iterate of $f$.)  For this reason, one introduces the {\em dynamical degree} of $f$ denoted by $\lambda_1(f)$ and defined as
\begin{align*}
    \lambda_1(f) := \lim \left( d_{\rm alg}(f^n)\right)^{1/n}.
\end{align*}
This limit always exists and it is unchanged under birational conjugacies.   The main dynamical properties of a dominant rational self-map of a K\"ahler surface (and in particular of $\mathbb{CP}^2$) are conjectured to be determined by the two invariants $\lambda_1(f)$ and $d_{\rm top}(f)$; see, e.g. \cite{GUEDJ}.  For this reason, the first thing one does when encountering a new rational self-map of $\mathbb{CP}^2$ is to ask ``what are $\lambda_1(f)$ and $d_{\rm top}(f)$?''

The mutation mappings $\mu_{p,q}$ are birational, so $d_{\rm top}(\mu_{p,q}) = 1$.   It is much more subtle to compute the dynamical degree $\lambda_1(\mu_{p,q})$ and this computation (and its corollaries) is one of the main achievements of our paper.

\subsection{Ruling out conserved quantities, invariant fibrations, and invariant foliations.}
Given a birational mapping $f:X\dashrightarrow X$ of a projective algebraic surface, an invariant fibration is a triple $(Y, \rho, g)$, where $Y$ is a projective curve, $\rho : X \dashrightarrow Y$ is a dominant rational map and $g: Y \to Y$ is a regular map such that the diagram
\[\begin{tikzcd}
	X & X \\
	Y & Y
	\arrow["f", dashed, from=1-1, to=1-2]
	\arrow["\rho"', dashed, from=1-1, to=2-1]
	\arrow["\rho", dashed, from=1-2, to=2-2]
	\arrow["g", from=2-1, to=2-2]
\end{tikzcd}\]
commutes at each point where the composition is defined. (Note that some authors require the generic fibers of $\rho$ to be irreducible, but we do not in this paper.) A particular case of an invariant fibration is a conserved quantity (in that case $\rho \in \C(X),\; Y = \CP^1$ and $g=\mathrm{Id}_{\CP^1}$).
Note also that if $f: X \dashrightarrow X$ has an invariant fibration and if $\tilde{f}: \tilde{X} \dashrightarrow \tilde{X}$ is birationally conjugate to $f$, then $\tilde{f}$ also has an invariant fibration.   Many interesting examples of mappings that preserve a fibration are presented in \cite{DGL}, along with various methods that can be used to prove that a mapping
has an invariant fibration.

We will use the following two criteria for (non-)existence of an invariant fibration.   

\begin{propo}\label{PROP:INV_FIBR_VS_DYN_DEG}Let $f:X\dashrightarrow X$ be a birational self-map of a smooth projective surface. Then $f$ has an invariant fibration if and only if $\lambda_1(f) = 1$.
\end{propo}

\begin{propo}\label{PROP:INV_FIBR_VS_SUPER_ATTR_FP}
Let $f:X\dashrightarrow X$ be a birational self-map of a smooth projective surface. If there exists a superattracting periodic point $P \in X$ of $f$, then $f$ does not admit an invariant fibration.
\end{propo}

Proposition \ref{PROP:INV_FIBR_VS_DYN_DEG} is a combination of well-known results from Diller-Favre \cite{diller_favre_2001} and Dinh-Nguyen and Dinh-Nguyen-Troung \cite{Dinh-Nguyen1,Dinh-Nguyen2}.   Computations of dynamical degrees are often challenging, so in many cases Proposition \ref{PROP:INV_FIBR_VS_SUPER_ATTR_FP} is easier to apply than Proposition \ref{PROP:INV_FIBR_VS_DYN_DEG}.   More details about both statements will be given in Section \ref{SEC:GEN_WAYS_RULE_OUT_FIB}, including a full proof of Proposition \ref{PROP:INV_FIBR_VS_SUPER_ATTR_FP}.

\subsection{Results when $p,q \geq 1$.}
This is the primary focus of our paper because then connection of $\mu_{p,q}$ with cluster algebras only holds when $p,q \geq 1$.

\begin{theoreml}\label{thm:A} For any integers $p, q \geq 1$ we have 
\begin{itemize}
\item[(i)] If $pq \leq 4$ then $\lambda_1(\mu_{p,q}) = 1$.
\item[(ii)] If $pq > 4$ then 
\begin{align*}
   \lambda_1(\mu_{p,q}) = \frac{pq-2 + \sqrt{(pq-2)^2-4}}{2} > 1.
\end{align*}
\end{itemize}
\end{theoreml}

Our proof of Theorem \ref{thm:A} is based on finding an algebraically stable dynamics on a smooth projective surface that is birationally conjugate to $\mu_{p,q}$. This dynamics naturally acts by pullback on the Picard group of the surface. Motivated by the type of the corresponding linear pullback operator, we call the corresponding pair $(p,q)$ semi-simple or non-semi-simple. This boils down to a pair $(p,q)$ to be defined as semi-simple if and only if $p,q > 1$. 
\begin{rem}
The statement of Theorem A for the cases $p=q\geq 2$ is known, see  \cite[Example 5.3]{IK}.
\end{rem}

\begin{theoreml}\label{THM:POSPQ_NO_FIBR}
If $p,q \geq 1$ and $pq > 4$ then $\mu_{p,q}$ does not preserve an invariant fibration.   In particular, $\mu_{p,q}$ has no rational conserved quantity.
\end{theoreml}

The reader may compare Theorem \ref{THM:POSPQ_NO_FIBR} with Section 2.4 of \cite{CHEN202425} in which Chen and Li prove that for such $p$ and $q$ the mapping $\mu_{p,q}$ does not have a conserved quantity that is in the form of a Laurent polynomial in $x$ and $y$.  Because of Proposition \ref{PROP:INV_FIBR_VS_DYN_DEG}, Theorem \ref{THM:POSPQ_NO_FIBR} follows directly from Theorem \ref{thm:A}.  However, one can also check that if $p,q \geq 1$ and $pq > 4$ then $\mu_{p,q}$ has a superattracting fixed point (either as a self-map of $\mathbb{CP}^2$
or as a self-map of the surface $Y$ obtained by blowing up $\mathbb{CP}^2$ twice); see Lemmas \ref{Lemma: Superattracting semi-simple} and  \ref{Lemma: superattracting non-semisimple}.  Therefore, readers who do not want to compute $\lambda_1(\mu_{p,q})$ can obtain Theorem \ref{THM:POSPQ_NO_FIBR} more directly
by using Proposition \ref{PROP:INV_FIBR_VS_SUPER_ATTR_FP}.

Note that the mappings $\mu_{p,q}$ with $pq \leq 4$ all preserve fibrations, but they are interesting in their own right. We analyze them in \cref{sec:affine_cases} below.

\subsection{Results when $p,q \leq -1$}
In this case the mappings $\mu_{p,q}$ do not arise from cluster algebras, however
these mappings are interesting in their own right, especially because
they still preserve the meromorphic two form $(dx \wedge dy)/(xy)$, thus providing valuable examples of such maps.  Some aspects of these maps are more difficult than those with $p,q \geq 1$, but they can be well-handled using the method of {\em tropicalization} that was developed in \cite{diller_lin_2016} and further discussed in \cite{DR1,DR2}. 

\begin{theoreml} \label{thm:C}
If $p,q\leq -2$  then $\lambda_1(\mu_{p,q})$ is the largest real-root of polynomial equation $$x^4+(-p q -2 )x^3+3 x^2-2 x+1 =0.$$
Moreover, $\lambda_1(\mu_{p,q}) > 1$.  
\end{theoreml}

\begin{theoreml}\label{THM:NEGPQ_NO_FIBR} Suppose $p,q \leq -1$ and $pq > 4$ or $p=q=-2$.   Then the mapping $\mu_{p,q}$ does not preserve an invariant fibration.
\end{theoreml}

When $p,q \leq -2$, Theorem \ref{THM:NEGPQ_NO_FIBR} follows directly from Theorem \ref{thm:C}, using Proposition \ref{PROP:INV_FIBR_VS_DYN_DEG}.
When $p$ or $q$ equals~$-1$ computation of the dynamical degree seems rather involved, however if $pq > 4$
one can find a blow-up of $\mathbb{CP}^2$ on which $\mu_{p,q}$  has a superattracting periodic point; see \cref{Lemma: Superattracting negative pair} and Remark \ref{rem: p=-1, q<-4}.  Therefore, in that case, Proposition~\ref{PROP:INV_FIBR_VS_SUPER_ATTR_FP} gives the result of Theorem \ref{THM:NEGPQ_NO_FIBR}.  (In fact, Theorem \ref{THM:NEGPQ_NO_FIBR} can be obtained from Proposition \ref{PROP:INV_FIBR_VS_SUPER_ATTR_FP} in all cases except for $(p,q) = (-2,-2)$).

\begin{rem} We have computed $\lambda_1(\mu_{-1,-2})$ to equal the leading root of the polynomial
$x^6-4x^5+3x^4-2x^3+3x^2-4x+1$, which is larger than one.  The computation is quite cumbersome, so we do not include it in  this paper, however, we have also done numerical experiments in an attempt to verify it.  The cases
$(p,q) = (-1,-3)$ and $(-1,-4)$ seem to be even harder to determine the dynamical degree, but the same
numerical experiments indicate that we probably have $\lambda_1(\mu_{p,q}) > 1$ in those cases as well.
Therefore, we expect that $\mu_{-1,q}$ also does not preserve a fibration when $q=-2,-3,$ or $-4$.  We describe the numerical experiment in Section \ref{subsec:NUM_APPROX}.
\end{rem}

\begin{rem}
One can also consider the case when $p$ and $q$ have opposite signs ($pq < 0$).   We do not explore this
cases systematically in our paper.   However, we expect that the methods used here can be used to prove results similar
to Theorems A, B, C, and D in the cases $pq < 0$.
\end{rem}

\subsection{Ergodic Theory and Lyaponov Exponents of the maps.}
    
Our final result concerns the ergodic theory of the mappings $\mu_{p,q}$ and it is relevant for $p,q \geq 1$ as well as for $p,q \leq -1$.

\begin{theoreml}\label{thm:Ergodic_properties}
    Suppose  $p,q\leq -2$ or $p,q>0$ s.t. $pq > 4$  and let $\lambda_1 \equiv \lambda_1(\mu_{p,q}) > 1$ denote the dynamical degree of $\mu_{p,q}$,
    as expressed in Theorem \ref{thm:A}.   The mapping $\mu_{p,q}: \mathbb{CP}^2 \dashrightarrow \mathbb{CP}^2$ has
    \begin{itemize}
        \item[(i)] Topological entropy equal to $\log(\lambda_1) > 0$.
        \item[(ii)] An invariant measure $\nu$ of maximal entropy; i.e. whose measure-theoretic entropy equals $\log(\lambda_1)$.
        \item[(iii)] The Lyapunov exponents $\chi_+(\nu,\mu_{p,q})$ and $\chi_-(\nu,\mu_{p,q})$ of $\nu$ under $\mu_{p,q}$ satisfy
        \begin{align*}
           \chi_+(\nu,\mu_{p,q}) \geq \log(\lambda_1) / 2 > 0 > - \log(\lambda_1) / 2 \geq \chi_+(\nu,\mu_{p,q}).
        \end{align*}
        \item[(iv)] Saddle-type periodic orbits for $\mu_{p,q}$ are asymptotically equidistributed with respect to $\nu$.
        \item[(v)] $\nu$ assigns zero mass to any pluripolar set, in particular $\nu(C) = 0$ for any algebraic curve $C \subset \mathbb{CP}^2$.
        
    \end{itemize}
\end{theoreml}

The meanings of the terms used in  \cref{thm:Ergodic_properties} will be explained in \cref{subsec:entropy_of_the_mappings}.   Using the vast number of powerful results
on the ergodic properties of birational mappings mentioned earlier in the introduction, \cref{thm:Ergodic_properties} will follow easily from the fact that $\lambda_1(\mu_{p,q}) > 1$ after checking a few additional technical details.   We present it here as an illustration of what the methods of holomorphic dynamics in several variables can achieve when applied to a rational mapping whose motivations came from another area.

\subsection{Further discussion on ruling out conserved quantities}  
Theorem \ref{thm:Ergodic_properties} can also be used to rule out a rational
conserved quantity $\rho: \mathbb{CP}^2 \dashrightarrow \mathbb{CP}^1$.  More specifically, for $\nu$-almost any initial condition $(x_0,y_0)$ it follows from Birkhoff's Ergodic Theorem that 
\begin{align*}
\overline{\mathcal{O}^+(x_0,y_0)} = \overline{\left\{\mu_{p,q}^n(x_0,y_0) \, : \, n \geq 0\right\}}
\end{align*}
has full $\nu$-measure.  On the other hand, as long as $(x_0,y_0)$ was not chosen as one of the finitely many indeterminate points for $\rho$, then $\overline{\mathcal{O}^+(x_0,y_0)}$ would be a subset of the algebraic curve $C = \rho^{-1}(\rho(x_0,y_0)) \subset \mathbb{CP}^2$, which has $\nu$-measure $0$ by Part (v) of Theorem  \ref{thm:Ergodic_properties}.
Note that the discussion above also rules out the possibility of an algebraic function serving as a conserved quantity since the level sets would also be unions of finitely many algebraic curves.   

Using Part (iii) of Theorem \ref{thm:Ergodic_properties}, one can actually conclude even stronger
statements about the absence of conserved quantities.  Because the Lyaponov exponents are bounded away from zero, Pesin's Theorem (see, e.g., \cite[Ch. 21]{Katok_Hasselblatt}) implies that a set of arbitrarily large $\nu$ measure will have stable and unstable manifolds of a definite size.   Any reasonable notion of conserved quantity would need to be constant on the union of all of these stable and unstable manifolds, which would place serious restrictions on the regularity of the conserved quantity.

\begin{rem}
Machacek and Ovenhouse  \cite[Problem 4.1]{machacek2024discrete} also ask if $\mu_{p,q}$ has a conserved quanitity for real values of the exponents $p$ and $q$.  When $p$ and $q$ are not integers this is beyond the scope of the methods in this paper and it remains an interesting question.
\end{rem}

\subsection{Structure of the paper}\hfill

In Section \ref{sec:preliminaries}, we fix the notations and recall the basic notions and facts regarding rational maps of smooth projective surfaces, their algebraic stability and entropy. We also recall the notions of divisor, Picard group and blowup.

In Section \ref{SEC:GEN_WAYS_RULE_OUT_FIB}, we give proofs of Propositions \ref{PROP:INV_FIBR_VS_DYN_DEG} and \ref{PROP:INV_FIBR_VS_SUPER_ATTR_FP}. Each of these propositions gives a sufficient condition on a birational map to have no invariant fibration. We use these results further in the text.

In Section \ref{sec:affine_cases} we discuss two pairs $(p,q) = (2,2), (4,1)$ for which $\mu_{p,q}$ admits a rational conserved quantity. We explicitly integrate the dynamics $\mu_{p,q}$ in these cases and realize this dynamics as a finite-time evolution along the Hamiltonian flow generated by the conserved quantity.

In \cref{sec:stable_model}, we study the critical dynamics of the map $\mu_{p,q}:\CP^2\dashrightarrow \CP^2$. Moreover, by blowing up $\CP^2$ a finite number of times  we construct a surface $X_{p,q}$ (for $p,q\in \Z_{>0}$ with $pq>4$) such that the lift map $\mu_{p,q}:X_{p,q}\dashrightarrow X_{p,q}$ is algebraically stable. We use this stable model in \cref{sec:picard}. We also show that the map $\mu_{p,q}$ (for $p,q\in \Z_{>0}$ with $pq>4$) has a superattracting fixed point on $X_{p,q}$ (see Lemma \ref{Lemma: Superattracting semi-simple} and Lemma \ref{Lemma: superattracting non-semisimple}). Together with Proposition \ref{PROP:INV_FIBR_VS_SUPER_ATTR_FP} this proves Theorem \ref{THM:POSPQ_NO_FIBR}. 

In \cref{sec:picard}, we compute the map $\mu^*_{p,q}\in \mathrm{End}_{\Z}(\Pic(X_{p,q}))$ induced by the dynamics $\mu_{p,q}$ on the Picard group $\Pic(X_{p,q})$ for the pairs $(p,q)\in \Z_{>0}$ such that $pq>4$. We also compute the spectral radius of (the complexification of) $\mu^*_{p,q}$, which yields the dynamical degree of $\mu_{p,q}$. In particular, this proves the second part of the Theorem \ref{thm:A} which together with Proposition \ref{PROP:INV_FIBR_VS_DYN_DEG} implies Theorem \ref{THM:POSPQ_NO_FIBR}.

In \cref{sec:tropicalization}, we compute the tropicalizations for the map $\mu_{p,q}$ to better understand the stable model for the mapping obtained in \cref{sec:stable_model}.

In \cref{sec:negative_pairs}, we study the mapping $\mu_{p,q}$ when $p,q<0$ using the results proved in the previous sections. We prove Theorems \ref{thm:C} and \ref{THM:NEGPQ_NO_FIBR} using approaches similar to the ones used in Sections \ref{sec:stable_model} and \ref{sec:picard}.

In \cref{sec:entropy_and_futher_dynamical_consequences}, using the results from \cite{dillerInvariantMeasureLyapunov2001} and \cite{DillerDynamicsMeromorphicMaps2009b}, we give the proof of \cref{thm:Ergodic_properties}.

In \cref{sec:computer_experiment}, we give a computer simulation of the orbits of $\mu_{-1,-5}$, showing a phenomenon which is resembling KAM phenomenon.

Section \ref{sec:preliminaries} may be better consulted upon need than read completely.
Sections \ref{sec:affine_cases}, \ref{sec:tropicalization}, \ref{sec:entropy_and_futher_dynamical_consequences} and \ref{sec:computer_experiment} are rather additional to the Sections \ref{SEC:GEN_WAYS_RULE_OUT_FIB}, \ref{sec:stable_model}, \ref{sec:picard}, \ref{sec:negative_pairs} which constitute the kernel of the paper.

\subsection*{Acknowledgments} This work originated as a result of Michael Shapiro telling the fourth author about the paper of Machacek and Ovenhouse \cite{machacek2024discrete}. We are very grateful to him for this. 
We are grateful to Richard Birkett and Jeffrey Diller for many valuable conversations.  We are also grateful for Tien-Cuong Dinh for clarifying issues about maps
that preserve a fibration.
The first author thanks Mikhail Beshtein and Anton Dzhamay for useful discussions. The first author thanks SISSA and BIMSA for hospitality.  The work of Kalidindi, Quintero, and Roeder was partially supported by NSF grant DMS-2154414.

\section{Preliminaries}\label{sec:preliminaries}

In this section we recall standard definitions such as a rational map, Picard group, blowup of a surface and algebraic stability of a birational transformation of a surface. The reader can refer to \cite{shafarevich1994basic}, \cite{diller_favre_2001} for the details. We also recall some relevant statements for the process of computing dynamical degrees.  The last subsection contains details about the ergodic theory of rational maps.

Further, if not clarified, we assume that topological terms refer to the Zariski topology. Everywhere in this paper, the varieties/manifolds are defined over $\C$/are complex-analytic. 
Let $X$ be a quasiprojective variety. For a set of regular functions (or homogeneous polynomials in homogeneous coordinates on projective space) $\{\al_1,\dots, \al_N\}\subset \C[X]$ we denote the common vanishing locus of $\al_1,\dots, \al_N$ by $V(\al_1,\dots, \al_N) = \{x\in X |\ \al_1(x) = \dots = \al_N(x) = 0\}$. For any subset $Y \subset X$ we denote by $\mathbb{I}(Y)$ the ideal of regular functions (or forms in case of projective $X$) vanishing on $Y$.

\subsection{Rational maps.}

Let $X, Y$ be two irreducible quasiprojective varieties. A rational map $f :X \dashrightarrow Y$ is a class of regular maps $f:U\to Y$ (here $U\subset X$ is a non-empty open subset) by the relation given by
$$f \sim \tilde{f} \Leftrightarrow f \textit{ and } \tilde{f} \textit{ coincide on the intersection of their 
domains}.$$ There exists a representative $f:U_f \to Y$ such that its domain contains the domains of all the other representatives in the given class. The set $U_f$ is called domain of $f$ and the complement $\mcI(f) = X\backslash U_f$ is called indeterminacy locus of $f$.

Assume that $X\subset \CP^n,\; Y\subset \CP^m$. For a rational function $f:X \dashrightarrow Y$ any representative can be locally written as
\begin{equation}\label{eqn:rat_rep}
    f([x_0:\ldots :x_n]) = [f_{0}(x_0,\dots, x_n):\ldots :f_m(x_0,\dots, x_n)].    
\end{equation}
Here the $x_i$'s are homogeneous coordinates on $\CP^n$ and $f_j$'s are homogeneous polynomials of the same degree. A map $g: [x_0:\dots:x_n] \mapsto [g_{0}(x_0,\dots, x_n):\dots :g_m(x_0,\dots, x_n)]$ is another local representative of $f$ if and only if $f_j g_i = f_i g_j$ on $X$ for all $i,j$. 
In the case $X = \CP^n$ there is a unique representative of the form \eqref{eqn:rat_rep} such that $f_1,\dots, f_m$ are coprime. In that case we have $\mcI(f)=V(f_0,\dots, f_m)$. The degree of $f_j$ (for any $j$) is called the algebraic degree of $f$ we denote it by $\deg(f)$.
The image of a rational map is defined as the closure of $f(U_f)$ in $Y$ and we denote it by $f(X)$. A rational map $f:X\dashrightarrow Y$ is called dominant if $f(X) = Y$. If a rational map $f:X\dashrightarrow Y$ is dominant, then for any rational map $g:Y \dashrightarrow Z$ the composition $g\circ f$ is well-defined.
A dominant rational map $f:X\dashrightarrow Y$ is called birational if there is a dominant rational map $g:Y \dashrightarrow X$ which is a two-sided inverse of $f$.
The critical set of a rational map is given by $\mcC(f) = \overline{\{x\in U_f\;|\; d_{x}f \textit{ is not surjective}\}}$. 
 For an irreducible curve $C\subset X$, the map $f|_{C}:C \dashrightarrow X$ is rational. If the image of this map is a point, $f|_{C}(C) = f(C\backslash \mcI(f)) = \{P\}$, the curve $C$ is said to be collapsed by $f$ to the point $P$.
 
    The following statement is very well-known (see \cite[Proposition 3.3]{MR1422105} for dimension $2$ or more recently \cite[Proposition 3.2]{alonso2023}). 
\begin{lemma}\label{lemma:jac_crit}
    If a map $f:X\dashrightarrow Y$ is birational and $X$ and $Y$ are smooth projective surfaces, then the critical set of $f$ coincides with the union of all irreducible curves in $X$ collapsed by $f$ to points in $Y$.
\end{lemma}

\subsection{Superattracting periodic points.}\label{SEC:SUPERATTR}

\begin{defi}\label{def:aff_superattrfp}
    Let $V$ be an open subset of $\mathbb{C}^n$ containing $0$ and let $F:V \to \mathbb{C}^n$ be a holomorphic map with $F(0)=0$. We say that the point $0$ fixed by $F$ is superattracting if 
   \begin{equation*}
       DF_0=0 \in \mathrm{Mat}_{n\times n}(\C).
   \end{equation*}
\end{defi}

\begin{defi}
    Let $f: X \dashrightarrow X$ be a rational map over a smooth projective variety and let $P \in X\setminus\mathcal{I}(f)$ be a fixed point of $f$. We say that $P$ is a superattracting fixed point of $f$ if there exists an analytic chart $(U,\varphi)$ centered at $P$ such that the induced holomorphic map
    \[
    F = \varphi \circ f \circ \varphi^{-1} : \phi(U) \rightarrow \mathbb{C}^n,
    \]
    has $0$ as a superattracting fixed point in the sense of \cref{def:aff_superattrfp}. We say that $P$ is a superattracting periodic point if for some $n\in \mathbb{Z}_{>0}$, $P$ is a superattracting fixed point for $f^n$.
\end{defi}

\subsection{Picard group.}\label{sec:pic_gp_def}

Let $X$ be a smooth quasiprojective surface. A divisor $D$ on $X$ is a formal linear combination of irreducible curves 
$C_i \subset X$,
$$D = \sum_{i=1}^{k}m_iC_i.$$
Here $m_i \in \Z$ for all $i$. If all $m_i\neq 0$, the union $\bigcup\limits_{i} C_i$ is called support of $D$ and is denoted $\mathrm{Supp}(D)$. A divisor $D$ is called prime if $D = C$ for some irreducible curve $C$.

The divisor $\left<\varphi\right>$ of a rational function $\varphi\in \C(X)^{\times}$ is defined as follows. Given an irreducible curve $C\subset X$, there is an open affine subset $U\subset X$ such that $\mathbb{I}(C\cap U) = (\al)$ for some $\al \in \C[U]^{\times}$. Then writing $\displaystyle\varphi|_{U} = \frac{g}{h},\;\;g,h\in \C[U]^{\times}$ there exist unique $m_g, m_h\in\Z_{\geq 0}$ such that $g\in (\al^{m_{g}})\backslash(\al^{m_{g}+1})$ and $h\in (\al^{m_{h}})\backslash(\al^{m_{h}+1})$. Then the coefficient of $C$ in the divisor of $\varphi$ is defined to be $m_g-m_h$ and is denoted by $\nu_C(\varphi)$.

For a cover $\{U_i\}$ of $X$ by affine open sets and a set of non-zero rational functions $\{\varphi_i\},\;\; \varphi_i\in \C(U_i)$ such that $\displaystyle\frac{\varphi_i}{\varphi_j}$ is regular and does not vanish on $U_i\cap U_j$ one constructs a divisor as follows,
$$
D = \suml_{C}\nu_C(\varphi_i)C,
$$
where $C$ runs over all irreducible curves and for each $C$ the index $i$ is chosen in a way so that $C \cap U_i\neq \varnothing$. Any divisor can be obtained in this way.

Denote the group of divisors on $X$ by $\mathrm{Div}(X)$. Divisors of non-zero rational functions on $X$ form a subgroup $\mathrm{Div}_{p}(X)$ of the group $\mathrm{Div}(X)$.  Two divisors differing by a divisor of a rational function are called linearly equivalent. The quotient group
\begin{equation}
    \Pic(X) = \mathrm{Div}(X)/\mathrm{Div}_{p}(X),
\end{equation}
is called the Picard group.

Let $f:X \to Y$ be a regular dominant map between two smooth quasiprojective surfaces. A map $f^*:\Pic(Y) \to \Pic(X)$ is defined as follows.
Let $D$ be a divisor defined by a cover $\{U_i\}_{i\in I}$ and a system of functions $\{\varphi_i\}_{i\in I}$,  $\varphi_i \in \C(U_i)^{\times}$ for all $i$. Then the pullback $f^*(D)$ is the divisor defined by the cover $\{f^{-1}(U_i)\}_{i\in I}$ and the system of functions $\{\varphi_i\circ f\}_{i\in I}$. 

Now assume that $f:X \dashrightarrow Y$ is a dominant rational map between two smooth projective surfaces. The set $\mathcal{I}(f)$ of irregular points of $f$ has codimension $2$ and therefore is finite. Then one can define the pullback of $f$ as composition of the pullback of the regular map $f|_{X\backslash \mathcal{I}(f)}$ and the isomorphism $\Pic(X\backslash \mathcal{I}(f)) \cong \Pic(X)$.
The following lemma from \cite{kaschner2017complex} provides a more precise interpretation of the operator $f^*:Pic(Y)\to Pic(X)$ in the rational setting.
\begin{lemma}\cite[Lemma 2.1]{kaschner2017complex}
    Let $f:X \dashrightarrow Y$ be a dominant rational map between two smooth projective surfaces. Suppose $C\subset Y$ is an irreducible algebraic curve. Then,
    \begin{equation}\label{eqn:curve_pullback}
 f^{*}(C) = \suml_{D}m_DD,   
\end{equation}
where $D$ runs over irreducible components of $f^{-1}(C)=\overline{\left(f|_{X\backslash \mathcal{I}(f)}\right)^{-1}C}$ and the multiplicity $m_D$ is the order of vanishing of $\psi\circ f$ at any smooth point $P\in D\setminus\mathcal{I}(f)$ with $\psi$ being a local defining equation for $C$ at $f(P)$ (chosen to vanish to order 1 at smooth points of $C$).
\end{lemma}

\subsubsection{Blow-up}
Next, we recall a classical construction in algebraic geometry, known as the “blow-up,” which we will use frequently throughout this work. For further details, the reader may consult \cite{shafarevich1994basic} or the description of the blow-up construction in \cite{Hu1997}.

\begin{defi}\label{def:blowup}
     Let $X$ be a smooth projective variety and $Y\subset X$ be a closed subset. The blow-up of $X$ along $Y$ is a pair $(X_{Y}, \pi)$, where $X_Y$ is a projective variety and $\pi:X_Y\to X$ is a regular map such that $\pi^{-1}(Y)$ is a divisor (called the exceptional divisor) 
     and $(X_Y, \pi)$ has the following universal property. For any regular map $g:Z\to X$ such that $g^{-1}(Y)$ is a divisor in $Z$ there exists a unique regular map $\tilde{g}:Z\to X_Y$ such that the following diagram commutes.

     \[\begin{tikzcd}
	&{X_Y} \\
	Z & X
	\arrow["\pi", from=1-2, to=2-2]
	\arrow["{\tilde{g}}", dashed, from=2-1, to=1-2]
	\arrow["g"', from=2-1, to=2-2]
\end{tikzcd}\]
The map $\pi$ is called the \textbf{canonical projection}.
 \end{defi}
In most of this paper we only blow up at single points, i.e., when $Y=\{P\}$ for a point $P\in X$. Nevertheless, the construction naturally extends to arbitrary closed subvarieties $Y\subset X$. This more general viewpoint allows us to perform simultaneous blow-ups at multiple points, as in \cite{Hu1997}, while the universal property of the blow-up ensures that the final surface is independent of the order in which the blowups are performed. The reader may also go to sections \ref{subsec:stable_model_ss} and \ref{subsec:stable_model_nss} for detailed, step-by-step examples of the construction of blowup spaces in the context we are interested.
\subsubsection{Picard group of the Blow-up}
Let $y\in X$ and let $\pi :X_{\{y\}} \to X$ be the blowup of $X$ at the point $y$.  Then,
\begin{equation}\label{eqn:pic_blowup}
    \Pic(X_{\{y\}}) \cong \Pic(X) \oplus \Z E_y.
\end{equation}
The inclusion of $\Pic(X) \hookrightarrow \Pic(X_{\{y\}})$ is given by the pullback $\pi^*$ and $E_y$ is the class of the preimage $\pi^{-1}(\{y\})$. We will often refer the preimage $\pi^{-1}(\{y\})$ as the {\em exceptional line} corresponding to blowup $\pi$.

For any irreducible curves $C_1,C_2$ on $X$ we have
\begin{equation}\label{eqn:int_blowup}
\begin{aligned}
\pi^{*}([C_i]) &= [\pi^{\prime}(C_i)]  + k_i E_y.
\end{aligned}
\end{equation}
Here $k_i$ is the multiplicity of the point $y$ on $C_i$ and $\pi^{\prime}(C_i) = \overline{\pi^{-1}(C_i\backslash\{y\} )}$ is the proper transform of $C_i$.

We conclude with recalling that $\Pic(\CP^2) \cong \Z$ and a generator $[C]$ is class of a line.

\subsection{Algebraic Stability}\label{subsec:alg_stab}

Let $X$ be a surface and let $f:X\dashrightarrow X$ be a birational map. For a point $x\in U_f$ the orbit of $x$ is the set $\{p_j\}_{j=0}^{n}$, where $x=p_0\overset{f}{\mapsto}p_1\overset{f}{\mapsto}\dots\overset{f}{\mapsto}p_n$, where we assume that for all $j$ we have $p_j\in U_f$ and $n \in \Z_{\geq 0} \cup \{\infty\}$ is the largest possible.

\begin{defi}
 An orbit $p_0\overset{f}{\mapsto}p_1\overset{f}{\mapsto}\dots\overset{f}{\mapsto}p_n$ is called \textit{destabilizing} if $p_0\in \mathcal{I}(f^{-1})$ and $p_n\in \mathcal{I}(f)$. This orbit is called \textit{minimal} if it does not contain any shorter destabilizing orbits.

\end{defi}

It is well known that a point $p$ is in the indeterminacy locus $\mcI(f^{-1})$ if and only if there exists a component of $\mcC(f)$ collapsed to $p$ by the map $f$.

\begin{defi}
    A birational map $f:X\dashrightarrow X$ is algebraically stable if there is no destabilizing orbit of $f$.
\end{defi}
A pair $(X, \pi)$ of a smooth projective surface $X$ and a map $\pi : X \to \CP^2$ is called proper modification of $\CP^2$ if $\pi$ is a composition of a finite number of blowups.
The next result follows from work \cite{diller_favre_2001}.
\begin{theo}\label{theo: Diller-Stabilization algorithm}
Let $f:\mathbb{P}^2\dashrightarrow \mathbb{P}^2$ be birational, then there exists a proper modification
$\pi : X \to\mathbb{P}^2$, such that $f$ is lifted to an algebraically stable map $\hat{f}: X \dashrightarrow X$.
\end{theo}
An algorithm to construct a projective space $X$ and the algebraically stable map $\hat{f}$ is the following: Start by blowing up $\mathbb{P}^2$ at all the points in a minimal destabilizing orbit of the map $f$. Such orbit exists since the map $f$ is no algebraically stable. 
This gives a lift $f_1:X_1\dashrightarrow X_1$. If the map $f_1$ is not algebraically stable, then blow up $X_1$ at all the points in a minimal destabilizing orbit of the map $f_1$. This gives a lift $f_2:X_2\dashrightarrow X_2$. One repeats such procedure until the map $f_i$ is algebraically stable. 
A proof that this process terminates after a finite number of steps can be found in  \cite{BIRKETT}.

Algebraic stability is relevant as it implies the functoriality of the pull-back $f^*:\Pic(X)\to \Pic(X)$. The relation of this to the dynamics induced by $f$ is explained below.

\begin{defi}
    Let $X$ be a smooth projective surface. Let $f:X\dashrightarrow X$ be a rational map. The first dynamical degree of $f$ is the limit
\begin{equation}\label{eq:dd_def}
    \lambda_1(f) = \mathop{\lim}\limits_{n\to\infty}(\lVert(f^n)^*\rVert^{\frac{1}{n}}).
\end{equation}
\end{defi}
\noindent Here $\lVert.\rVert$ is any norm  on $\mathrm{End}\big(\C \mathop{\otimes}\limits_{\Z}\Pic(X)\big)$. It is known that in general $(f^n)^* \neq (f^*)^{n}$. However it was proved in \cite[1.14]{diller_favre_2001} that if the map $f$ is algebraically stable then we have $(f^n)^* = (f^*)^{n}$.   (Readers who prefer a more algebraic approach proof of this fact can refer to \cite[Prop. 1.4]{Roeder}.) This implies

\begin{lemma}\label{lemma:spec_rad}
    Let $X$ be a smooth projective surface. Let $f:X\dashrightarrow X$ be an algebraically stable birational map.
Then the dynamical degree $\lambda_1(f)$ equals to the spectral radius of the operator $f^*:\Pic(X) \to \Pic(X)$.
\end{lemma}

We use \cref{lemma:spec_rad} to compute the dynamical degree of $\mu_{p,q}$ in \cref{sec:picard} (for $p, q > 0)$ and in Section \ref{Subsc: 8.3 Dynamical degree negative} (for $p,q \leq -2$).
We conclude this subsection with a number of technical statements which help to construct a proper modification of a birational dynamics.

\begin{lemma}\label{lemma:no_collapse_lift}
    Let $X$ be a surface and let $\pi:X_{y} \to X$ be the blowup of $X$ at $y$. Let $f:X\dashrightarrow X$ be a rational map. Let $C \subset X$ be an irreducible curve. 
    If the map $f$ does not collapse $C$ to a point, then the lift $\hat{f}:X_{y} \dashrightarrow X_y$ does not collapse the proper transform of $C$ to a point.
\end{lemma}

In particular, the set of critical curves for the lift $\hat{f}: X_{y}\dashrightarrow X_{y}$ is contained in the set of the proper transforms of the components of the critical set $\mcC(f)$ with the exceptional line $\pi^{-1}(\{y\})$ added.

For the rest of this subsection fix a smooth projective surface $X$ and a rational map $f:X\dashrightarrow X$. Below in this subsection we assume that $\pi : \hat{X} \to X$ is composition of a finite sequence of single-point blowups of $X$ at points that are indeterminacies of the corresponding lifts of the map $f$. Denote the lift of $f$ to $\hat{X}$ by $\hat{f}$.

\begin{lemma}\label{lemma:nondestcurvecrit}
Let $C\subset X$ be an irreducible curve.
     If $\pi^{-1}\circ f$ does not collapse $C$ to a point in a destabilizing orbit of $\hat{f}$ then $\hat{f}$ does not collapse the proper transform of $C$ to a point in a destabilizing orbit of $\hat{f}$. 
\end{lemma} 
\begin{proof}
    Let $C^{\prime}$ be the proper transform of the curve $C$ with respect to $\pi$. If the map $\hat{f}$ does not collapse $C^{\prime}$ there is nothing to prove. 

    Assume that $C^{\prime}$ is collapsed to a point by the map $\hat{f}$. There is an open subset $U\subset C^{\prime}$ and an open subset $V \subset C$ such that $\pi|_U:U\to V$ is biregular. On an open non-empty subset $W\subset V$ the map $\pi^{-1}\circ f$ is regular. Therefore restricted to an open non-empty subset $\pi|_{U}^{-1}(W)$ the composition $\hat{f} = (\pi^{-1}\circ f)\circ \pi$ is just the composition of regular maps. Therefore $\pi^{-1}\circ f(V)$ is a point. The map $\pi$ is regular, therefore the map $\pi^{-1}$ does not collapse any curve, hence $f(V)$ is a point, $f(V) = \{P^{\prime}\}$. By the assumption the orbit of $P^{\prime}$ is not a destabilizing orbit. Since by assumption all blowups are performed at indeterminacies of lifts of $f$, the maps $\pi^{-1}$ and $f$ are regular at the points in the orbit of $P^{\prime}$ which implies that the map $\hat{f}$ is regular on the lift of the orbit of $P^{\prime}$.
\end{proof}

\begin{coro}\label{cor:alg_stab_1}
Assume that any component of the critical locus $\mcC(f)$ is not mapped by $\pi^{-1}\circ f$ to a point in destabilizing orbit and that any component of the exceptional curve of $\pi$ is not mapped by $\hat{f}$ to a point in a destabilizing orbit.  Then $\hat{f}$ is algebraically stable.
\end{coro}
\begin{proof}
      Let $C$ be an irreducible curve on $\hat{X}$ collapsed by $\hat{f}$. We have two cases.
    \begin{itemize}

        \item If $\pi$ does not collapse $C$ then $C$ is the proper transform of the curve $\pi(C)$. If $f$ does not collapse $\pi(C)$, then $\hat{f}$ does not collapse $C$ by \cref{lemma:no_collapse_lift}. 
        Otherwise $\pi(C)$ is a component of the critical set $\mcC(f)$. Then by the assumption $\pi^{-1}\circ f$ does not collapse $\pi(C)$ to a point in a destabilizing orbit. Then $\hat{f}$ does not collapse $C$ to a point in a destabilizing orbit of $\hat{f}$ by \cref{lemma:nondestcurvecrit}. 

         \item If $\pi$ collapses $C$, then $C$ is an irreducible component of the exceptional curve of $\pi$. Then $C$ is not collapsed by $\hat{f}$ to a point in a destabilizing orbit of $\hat{f}$ by the assumption.

  \end{itemize}

        Therefore any curve collapsed by $\hat{f}$ is not mapped to a point in a destabilizing orbit.
\end{proof}
\begin{coro}\label{cor:alg_stab_2}
    Under assumptions of  \cref{lemma:nondestcurvecrit}  let the maps $\pi_1,\pi_2$ correspond to finite sequences of blowups at points corresponding to the map $\pi = \pi_{1}\circ \pi_{2}$ (see the diagram below). Let $\tilde{f}, \hat{f}$ be the lifts of a birational map $f$. Let $C\subset X$ be an irreducible curve.
    \[\begin{tikzcd}
    {\hat{X}} & {\hat{X}} \\
	{\tilde{X}} & {\tilde{X}} \\
	{X} & {X}
	\arrow["{\hat{f}}"', from=1-1, to=1-2, dashed]
	\arrow["\pi_2"',from=1-1, to=2-1]
	\arrow["\pi_2"',from=1-2, to=2-2]
	\arrow["\tilde{f}"', from=2-1, to=2-2, dashed]
    \arrow["\pi_1"',from=2-2, to=3-2]
	\arrow["\pi_1"',from=2-1, to=3-1]
	\arrow["f"', from=3-1, to=3-2, dashed]
\end{tikzcd}\]
If the map $\pi_1^{-1}\circ f$ does not collapse the curve $C$ to a point in a destabilizing orbit of $\tilde{f}$, then the lift $\hat{f}:\hat{X} \to \hat{X}$ does not collapse the proper transform  of the curve $C\subset X$ under $\pi_1 \circ \pi_2$ to a point in a destabilizing orbit of $\hat{f}$.
\end{coro}

\begin{proof}
By  \cref{lemma:nondestcurvecrit} the map $\tilde{f}$ does not collapse the proper transform $\tilde{C}$ of $C$ with respect to the map $\pi_1$ to a point in a destabilizing orbit of $\tilde{f}$.

There are two cases
\begin{itemize}
    \item
    Assume that the map $\tilde{f}$ does not collapse $\tilde{C}$ to a point. Then by \cref{lemma:no_collapse_lift} the proper transform $\hat{C}$ of $\tilde{C}$ with respect to $\pi_2$ is not collapsed by the map $\hat{f}$. It remains to notice that $\hat{C}$ coincides with the proper transform of $C$ with respect to $\pi$.
    \item Assume that the proper transform $\tilde{C}$ is collapsed to a point $P$ in a non-destabilizing orbit of $\tilde{f}$. Since the blowup $\pi_2$ is done at indeterminacy points of lifts of $\tilde{f}$, the map $\pi_{2}^{-1}$ is regular at the points of orbit of $P$. Then the curve $\hat{C}$ is collapsed by $\hat{f}$ to the point $\pi_2^{-1}(P)$. Moreover, the orbit of $\pi_{2}^{-1}(P)$ is given by $\hat{f}^{n}(\pi_2^{-1}(P)) = \pi_2^{-1}(\tilde{f}^{n}(P))$ for $n\geq 0$ and is not destabilizing. 
\end{itemize}
\end{proof}
\subsection{Entropy and Lyapunov exponents of rational mappings}\label{subsec:entropy_of_the_mappings}
In this section, we provide definitions and some background material to understand the terminology in the statement of \cref{thm:Ergodic_properties}.

Let $X$ be a smooth projective surface and $f: X\dashrightarrow X$ be a rational map.  Recall the notation $\mcI(f)$ for the set of indeterminate points of $f$. Let $d$ be the metric induced by the Fubini-Study metric on $X$.
\subsubsection{Topological Entropy}
Let
  \begin{equation*}
   \Omega(f)=\{ x\in X\,\vert\,  f^i(x)\notin \mcI(f), \forall i=0,1,\dots\}.
\end{equation*}
Let $\mathcal{F}\subset \Omega(f)$, $n\in\N$ and $\epsilon>0$. The set $\mathcal{F}$ is said to be $(n,\epsilon)$-separated, if
\begin{equation*}
    \max_{0\leq i\leq n-1}d(f^i(x), f^i(y)) \geq \epsilon
\end{equation*}
for all distinct $x,y\in\mathcal{F}$.

\begin{defi}[See \cite{FRIEDLAND},\cite{dinhUpperBoundTopological2006}]
   Let $f:X\dashrightarrow X$ be a dominant rational self-map. The topological entropy $h_{top}(f)$ is defined by 
\[
    h_{top}(f):=\sup_{\epsilon>0}\left(\limsup_{n\to\infty}\frac{1}{n}\log\max\{\#\mathcal{F}\,\vert\,\mathcal{F}\, is\;(n,\epsilon)-separated \text{ in }\, \Omega(f)\}\right)
\]
\end{defi}

Dinh and Sibony in \cite{dinhUpperBoundTopological2006} showed the following result for the more general case when $f$ is a meromorphic map on a compact K\"ahler manifold of arbitrary dimension $n$. We state the simplified version needed here.
\begin{theo}\cite{dinhUpperBoundTopological2006}
    Let $f:X\dashrightarrow X$ be a birational self-map. Let $\lambda_1(f)$ denote the dynamical degree of $f$. Then,
    \[
        h_{top}(f) \leq \log{\lambda_1(f)}.
    \]
\end{theo}

\subsubsection{Invariant measure} Let $\nu$ be a probability measure on $X$. On $X\setminus \mcI(f)$, the map $f$ is measurable. For any $A$ in the Borel $\sigma$-algebra, 
\begin{equation}
    (f_*\nu)(A):=\nu(\{x\in X\setminus \mcI(f)\; |\:
    f(x)\in A\}) 
\end{equation}
If $\nu$ does not give mass to the indeterminacy set $\mcI(f)$, then 
\begin{equation}
    \int_X\phi\circ f\,d\nu = \int_X \phi\,d(f_*\nu),
\end{equation}
for all $\phi \in L^1(f_*\nu)$. 
\begin{defi}
    A measure $\nu$ is invariant under $f$, if $\nu$ gives no mass to $\mcI(f)$ and $f_*\nu = \nu$
\end{defi}
An equivalent statement of $f$-invariant measure $\nu$ is given by the following lemma.
\begin{lemma}
    Let $\nu$ be a measure that gives no mass to $\mcI(f)$. Then the following statements are equivalent
    \begin{enumerate}
        \item $\nu$ is $f$-invariant.
        \item For any continuous function $\phi$, we have
        \[
        \int_{X\setminus \mcI(f)} \phi\circ f\,d\nu =\int_X\phi\,d\nu .
        \]
    \end{enumerate}
\end{lemma}

\subsubsection{Measure-theoretic Entropy} Let $\nu$ be $f-$invariant probability measure. Let $\mathcal{A} = \{A_1,\dots,A_k\}$ be a partition of $X$, and let $\displaystyle \bigvee_{i=0}^{n-1}f^{-i}(\mathcal{A})$ be the partition generated by $\displaystyle \mathcal{A}, f^{-1}(\mathcal{A}),\dots,f^{-n+1}(\mathcal{A})$. Set 
\begin{align*}
    H(\mathcal{A}) = -\sum_{i=1}^{k}\nu(A_i)\log{\nu(A_i)}, \text{ and }\,
    h(\mathcal{A},f)=\lim_{n\to \infty}\frac{1} {n}H\left(\bigvee_{i=0}^{n-1}f^{-i}(\mathcal{A})\right).
\end{align*}
\begin{defi}
    Measure-theoretic entropy is defined as 
    \begin{equation}
        h_{\nu}(f):=\sup_{\mathcal{A}}h(\mathcal{A},f)
    \end{equation}
\end{defi}
\subsubsection{Lyapunov Exponents} let $p\in X$, and $v\in T_pX$, then
\begin{equation*}
    \chi(v,p):=\lim_{n\to \infty}\frac{1}{n}\log{|Df_p^n(v)|}
\end{equation*}
when the limit exists, it is called the  Lyapunov exponent of $(x,v)$.

Let $f:X\dashrightarrow X$ be a birational map and let $\nu$ be a probability measure that is ergodic with respect to $f$ i.e., if for any measurable $B\subset X$, $B= f^{-1}(B)$ then either $\nu(B)=0$ or $\nu(X\setminus B)=0$. If $\log{\|Df\|}, \log{\|Df^{-1}\|}\in L^1(\nu)$ then by Oseledec's Theorem, there exist two real numbers $\chi^{+}(f,\nu), \chi^-(f,\nu)$ such that for $\nu$ a.e. $p\in X$ and a generic vector $v\in T_pX$, 
$$\chi^{+}(f,\nu)=\lim_{n\to \infty}\frac{1}{n}\log{|Df_p^n(v)|}$$
and similarly for $f^{-1}$ and $\chi^{-}(f,\nu)$. (See \cite{dillerInvariantMeasureLyapunov2001} Section $1$)

\section{Ways to prove a birational map has no invariant fibration.}\label{SEC:GEN_WAYS_RULE_OUT_FIB}

We will now prove Propositions \ref{PROP:INV_FIBR_VS_DYN_DEG} and \ref{PROP:INV_FIBR_VS_SUPER_ATTR_FP}.

\begin{proof}[Proof of \cref{PROP:INV_FIBR_VS_DYN_DEG}]
First suppose that $f: X \dashrightarrow X$ is a birational map that has an invariant fibration.  It follows from powerful results of Dinh-Nguyen \cite{Dinh-Nguyen1} and Dinh-Nguyen-Truong \cite{Dinh-Nguyen2} that 
 a dominant rational self-map $f$ of a projective surface admitting an invariant fibration must have that the dynamical degree divides the topological degree of the map $f$.   (See \cite[Lemma 2.4]{kaschner2017complex} for further clarification.)  The topological degree of a birational map is $1$, therefore if it admits an invariant fibration the dynamical degree also has to be~$1$.

Converselly, if $f: X \dashrightarrow X$ is a birational map with $\lambda_1(f) = 1$ then Diller-Favre \cite[Thm. 0.2]{diller_favre_2001} proved that
$f$ has an invariant fibration.
(See also \cite{DESERTI_BLANC}.)
\end{proof}

\begin{lemma}\label{lemman: non-constant over exc. line}
   Let $\rho:X\dashrightarrow Y$ be a rational map between smooth projective surfaces. Assume that $P\in\mathcal{I}(\rho)$. Then there is a finite number of blowups $\pi:\hat{X} \to X$ over the point $P$ resolving the indetermination at $P$ such that the lift $\hat{\rho} = \rho \circ \pi:\hat{X}\to Y$ is regular in a neighborhood of $\pi^{-1}(P)$ and $\hat{\rho}$ is non-constant on some irreducible component $E$ of the exceptional divisor $\pi^{-1}(P)$.
\end{lemma}

\begin{proof}
Existence of the desired sequence of blowups is given by Theorem 4.8 in \cite{shafarevich1994basic}. Notice that if $\pi$ is the blow-up that regularizes the indetermination $P$, we have $E=\pi^{-1}(P)$ is connected. If  we suppose that $\hat{\rho}|_{E_0}$ is constant for every exceptional curve $E_0$ over $P$ then $\hat{\rho}(E)= Q$ for some fixed $Q\in Y$ as $E$ is connected. So, for a sufficiently small neighborhood $U$ of $E$, we have that $\hat{\rho}|_U$ is well-defined and bounded. Therefore, $P$ is a removable singularity for $\rho$, so it's not a genuine indetermination of $\rho$, a contradiction.
\end{proof}

\begin{proof}[Proof of \cref{PROP:INV_FIBR_VS_SUPER_ATTR_FP}]

     Suppose $f$ admits an invariant fibration:

    \[\begin{tikzcd}
	{X} & {X} \\
	Y & Y
	\arrow["{f}", dashed, from=1-1, to=1-2]
	\arrow["{\rho}"', dashed, from=1-1, to=2-1]
	\arrow["{\rho}", dashed, from=1-2, to=2-2]
	\arrow["g"', from=2-1, to=2-2]
\end{tikzcd}\]
Notice that if $f$ preserves an invariant fibration, then any iterate $f^n$ preserves the same fibration. So, without loss of generality, up to going to a suitable iterate, we can assume $P$ is a superattracting fixed point. Consider the Taylor series of $f$ by going to local affine coordinates $(z_1,z_2)$ centered at the superattracting fixed point $P$:
\begin{equation}\label{39}
    \displaystyle f(z_1,z_2)=\left(\displaystyle\sum_{i=k}^\infty H_i(z_1,z_2),\displaystyle\sum_{i=k}^\infty G_i(z_1,z_2)\right)
\end{equation}
where $k\geq2$, and each $H_j,G_j$ are homogeneous of degree $j$.

\textbf{Case 1: $P\notin\mathcal{I}(\rho)$.} 
 Fix some holomorphic coordinate $w$ on a small neighborhood of $\rho(P)$ centered at this point. As $P$ is a $\rho$-regular point, consider also a Taylor series with respect to the coordinate $w$. 
We have $\rho(z_1,z_2)=T_d(z_1,z_2)+T_{> d}(z_1,z_2)$ where $T_d(z_1,z_2)$ is the homogeneous term of minimal degree $d\geq 1$ on this Taylor series and $T_{> d}$ corresponds to greater degree terms. Given that $\rho\circ f=g\circ\rho$ we obtain 
\begin{align*}
    T_d\left(\displaystyle\sum_{i=k}^\infty H_i(z_1,z_2),\displaystyle\sum_{i=k}^\infty G_i(z_1,z_2)\right)+T_{>d}\left(\displaystyle\sum_{i=k}^\infty H_i(z_1,z_2),\displaystyle\sum_{i=k}^\infty G_i(z_1,z_2)\right)&=g(T_d(z_1,z_2)+T_{>d}(z_1,z_2))\\
    &= g'(0)T_d(z_1,z_2) + \tilde{T}_{>d}(z_1,z_2).
\end{align*}
Because $f$ is birational then $g$ is a biholomorphism of Riemann surfaces, therefore the differential of $g$ does not vanish. Since $d\geq 1$, there is a mismatch for the term of the smallest degree of both expressions. This is a contradiction.

     \textbf{Case 2: $P\in\mathcal{I}(\rho)$.}  By \cref{lemman: non-constant over exc. line}, there is a finite number of blow-ups $\pi=\pi_k\circ\dots\circ\pi_1$ over the point $P$, such that the composition $\hat{\rho}=\rho\circ\pi$ is regular on the preimage $\pi^{-1}(P)$. Moreover, there is a component $E_0$ of $E=\pi^{-1}(P)$ on which $\hat{\rho}$ is non-constant. We have a commuting diagram

\[\begin{tikzcd}
	{\hat{X}} & {\hat{X}} \\
	Y & Y
	\arrow["{\hat{f}}", dashed, from=1-1, to=1-2]
	\arrow["{\hat{\rho}}"',  from=1-1, to=2-1]
	\arrow["{\hat{\rho}}",  from=1-2, to=2-2]
	\arrow["g"', from=2-1, to=2-2]
\end{tikzcd}\]

 Notice that $\hat{f}$ maps $E$ to itself as $P$ was originally fixed by $f$. Moreover, the iterations of $\hat{f}$ do not collapse $E_0$. Indeed, otherwise due to $g\circ\hat{\rho}=\hat{\rho}\circ\hat{f}$ the set $g^{N}(\hat{\rho}(E_0))$ is a point for some $N\in \Z_{>0}$. But this is not possible as $g$ is a biholomorphic map and $\hat{\rho}|_{E_0}$ is non-constant. A similar reasoning shows that the map $\hat{\rho}$ is non-constant on the images $\hat{f}^{N}(E_0)$.

Without loss of generality, up to replacing $\hat{f}$ by some iterate and $E_0$ by an image with respect to some iterate of $\hat{f}$, we can assume that $\hat{f}\left(E_0\right)\subset E_0$. Indeed, as $E$ is invariant under $\hat{f}$ and the iterations of $\hat{f}$ do not collapse $E_0$, since $\{\hat{f}^{N}(E_0)\}$ is a subset of the set of components of $E$ there exist $0\leq N_1 < N_2$ such that $\hat{f}^{N_1}(E_0) = \hat{f}^{N_2}(E_0) = \hat{f}^{N_2-N_1}(\hat{f}^{N_1}(E_0))$.

Consider local affine coordinates $(w_1,w_2)$ on $\hat{X}$ near the exceptional line $E_0=\{w_1=0\}$ and local affine coordinates $(z_1,z_2)$ centered at the super attracting fixed point $P\in X$. We can write locally 
\begin{equation}\label{Eq: exceptional divisor 40}
    \pi(w_1,w_2)=(z_1,z_2)=\left(w_1^aP_1(w_1,w_2),w_1^bP_2(w_1,w_2)\right),
\end{equation}where the multiplicities (cf. \cref{sec:pic_gp_def}) $a,b$ are positive integers. (In particular $P_1$, $P_2$ are rational functions with $P_i(0,w_2)\not \equiv 0$ with $w_1=0$ not a pole for $i=1,2$.) 
Given $n\in \Z_{>0}$ by the considerations above we have $\hat{f}^n(E_0)\subset E_0$ and $\hat{f}^n$ does not collapse $E_0$. That implies 
\begin{equation}\label{Eq: local affine lift fibration}
    (\tilde{w}_1,\tilde{w}_2):=\hat{f}^{n}(w_1,w_2)=(w_1^{d_n}Q_{1,n}(w_1,w_2), Q_{2,n}(w_1,w_2)),
\end{equation}
where the multiplicity $d_n$ is a positive integer and $Q_{i,n}$ are rational functions such that $Q_{1,n}(0,w_2)\not\equiv 0$ ($w_1=0$ is not a pole of $Q_{i,n}$ for $i = 1,2$) and $Q_{2,n}(0,w_2)$ is a non-constant rational function. For any $n\in\N$, the following diagram commutes

\[\begin{tikzcd}
	{\hat{X}} & {\hat{X}} \\
	X & X
	\arrow["{\hat{f}^{n}}", dashed, from=1-1, to=1-2]
	\arrow["\pi"', from=1-1, to=2-1]
	\arrow["\pi", from=1-2, to=2-2]
	\arrow["{f^{n}}"', dashed, from=2-1, to=2-2].
\end{tikzcd}\]
Therefore, substituting \eqref{Eq: local affine lift fibration} into\eqref{Eq: exceptional divisor 40} we obtain
\begin{equation}\label{eq:superthmpf1}
    \pi\circ \hat{f}^{ n}(w_1,w_2) =\left(w_1^{ad_n}(Q_{1,n}(w_1,w_2))^aP_1(\tilde{w_1},\tilde{w_2}),\,w_1^{bd_n}(Q_{1,n}(w_1,w_2))^bP_2(\tilde{w_1},\tilde{w_2})\right) 
\end{equation}
Notice here that the greatest power of $w_1$ is factored already in both coordinates as $\tilde{w_1}|_{w_1=0}\equiv0$ and $\tilde{w_2}|_{w_1=0}$ is non-constant. Also, \eqref{eq:superthmpf1} equals to 

\begin{equation}\label{eq:superthmpf2}
    f^{n}\circ \pi(w_1,w_2)= F^{(n)}\left(w_1^aP_1(w_1,w_2),w_1^bP_2(w_1,w_2)\right) + O\left(\|\left(w_1^aP_1(w_1,w_2),w_1^bP_2(w_1,w_2)\right)\|^{k_n + 1}\right).
\end{equation}
Here $F^{(n)}=(H_{k_n},G_{k_n})$ is the lowest degree term (of degree $k_n$) in the Taylor expansion of $f^{n}$ at $P=(0,0)$ (cf. \eqref{39}). Note that the $k_n$'s grow exponentially because $P$ is a superattracting fixed point for $f$. The equality of \eqref{eq:superthmpf1} and \eqref{eq:superthmpf2} implies that the $d_n$'s grow exponentially too. Indeed, the lowest power of $w_1$ appearing in the right-hand side of \eqref{eq:superthmpf2} is bounded from below by $\min(a,b)k_n \geq k_n$.

Choose $n\in\Z_{>0}$ big enough so that $d_n\geq 2$. In the local coordinates $(w_1,w_2)$, one can use $\eqref{Eq: local affine lift fibration}$ to compute  
\begin{equation}\label{eq:pf_sap_det_crit}
    \det\left(D{\hat{f}^n}\right)\Big\vert_{w_1=0}=0.
\end{equation}
Note that $\hat{f}^n$ is a birational map, therefore \eqref{eq:pf_sap_det_crit} implies that $E_0=\{w_1=0\}$ is contracted to a point by $\hat{f}^n$ by \cref{{lemma:jac_crit}}. But the set $E_0$ is not contracted by iterations of $\hat{f}$. This is a contradiction that finishes the proof.
\end{proof}

\section{Affine pairs.}\label{sec:affine_cases}

In this section, we study the pairs $(p,q)\in\N^2$ such that $pq = 4$, which separate the pairs that give simple periodic dynamics of $\mu_{p,q}$ ($pq<4$) from the pairs which give non-integrable dynamics of $\mu_{p,q}$  $(pq > 4)$.
These pairs are given by $\mathcal{A} = \{(2,2), (1,4),(4,1)\}$ and are called affine. For the maps $\mu_{p,q}$ with $(p,q)\in \mathcal{A}$, there are explicit conserved quantities which were found in work \cite{CHEN202425}. Below we explicitly integrate the dynamics on generic level sets of these conserved quantities. In particular, we show that the closure of a generic level set is a rational curve and the rational parametrization conjugates the dynamics given by the restriction of $\mu_{p,q}$ on the level set to a linear fractional transformation. We provide explicit formulas for rational parametrizations and for the corresponding linear fractional transformations. Also, we show that for the pairs $(p,q)\in \mathcal{A}$ the dynamics of $\mu_{p,q}$ can be naturally included into a hamiltonian flow induced by the conserved quantity with respect to the logarithmically constant Poisson bracket. In particular, we partially address the question from \cite{machacek2024discrete} for the case of $(p,q)\in \mathcal{A}$.

For the affine pairs without loss of generality we have either $p = 4, \;q = 1$ or $p = q = 2$ (see  \cref{rem:pgeqq}). 

We extend the dynamics $\mu_{2,2}:\C^2 \to \C^2$ and $\mu_{4,1}:\C^2 \to \C^2$ to the dynamics on $\CP^2$ as follows.

\begin{align}
    \mu_{2,2}([x_0:x_1:x_2]) &= [(x_2^2+x_1^2)x_0x_1: x_2^2x_0^2 + (x_2^2+x_1^2)^2 :x_0^2x_1x_2],\\
    \mu_{4,1}([x_0:x_1:x_2]) &= [x_0^3 x_1 \left(x_1+x_2\right):x_2
   \left(x_0^4+\left(x_1+x_2\right){}^4\right):x_0^4 x_1].
\end{align}

We will also need a meromorphic two-form which in local affine coordinates $(x,y) = \left(\frac{x_0}{x_2},\frac{x_1}{x_2}\right)$ can be written as
\begin{equation}\label{eqn:symplectic_form}
   \eta = \frac{dx\wedge dy}{xy}.
\end{equation}
The form $\eta$ defines a symplectic structure on an open subset of $\CP^2$. The corresponding Poisson bracket writes
\begin{equation}\label{eqn:PB}
    \{f,g\} = -(\partial_x(f)\partial_y(g)-\partial_x(g)\partial_{y}(f))xy.
\end{equation}
This bracket is standard in the framework of cluster algebras. The map $\mu_{p,q}$ preserves $\eta$, see Equation \eqref{eqn:2form_pres}. In particular, it implies that if $\mu_{p,q}$ admits a conserved quantity $H(x,y)$ then the hamiltonian flow induced by $H$ on its level sets commutes with $\mu_{p,q}$ acting on them.

\subsection{The case $p = q =2$.}

The conserved quantity can be obtained from the following construction. Consider the $6$-dimensional space of homogeneous polynomials of degree $2$, $\rV_2 = \left<\{x_0^2, x_0x_1, x_0x_2, x_1^2, x_1x_2, x_2^2\}\right>$. Consider the subspace $\rW_2 \subset \rV_2$ of polynomials vanishing at $\{[\pm\ri:0:1], [0:\pm\ri:1]\}$. This subspace is given by $\left<\{x_0x_1, x_0^2 + x_1^2 + x_2^2\}\right>$. 
\begin{propo}\label{prop:22int}\cite{CHEN202425}
There is a mutation invariant Laurent polynomial $I_{2,2}(x_0,x_1,x_2) = \frac{x_0^2 + x_1^2 + x_2^2}{x_0x_1}$. In particular,
    $$\mu_{2,2}^*\left(I_{2,2}\right)(x_0,x_1,x_2) = I_{2,2}(x_0,x_1,x_2).$$
\end{propo}

We introduce the coordinates as follows. A generic level set of the conserved quantity $I_{2,2}$ function is the conic $C_{\be} = \{x_0^2 + x_1^2 + x_2^2 - \be x_0x_1 = 0\}$. Note that for generic $\beta\in \C$ the curve $C_{\be}$ is irreducible and passes through the point $[\ri:0:1]$. Using the standard parametrization of conic with respect to this point we get a birational map $\psi_{2,2}: \CP^2 \dashrightarrow \C^2$.

\begin{equation}
    \psi_{2,2}:
       [x_0:x_1:x_2] \mapsto  \left(\frac{x_0 -\ri x_2}{x_1},I_{2,2}(x_0,x_1,x_2)\right).\\
\end{equation}
The inverse map is given by
\begin{equation}
    \psi_{2,2}^{-1}: (\al,\be) \mapsto [(\al^2-1):(2\al-\be):\ri(1+\al^2-\al\be)]. 
\end{equation}

Using functions $(\al,\be)$ as local coordinates we directly integrate the dynamics given by $\mu_{2,2}$.
\begin{propo}\label{prop:aff_22}
    The map $\psi_{2,2}$ conjugates the dynamics induced by $\mu_{2,2}$ to the one given by
    \begin{equation}
        L_{2,2}:(\al, \be) \mapsto \left(\frac{1}{\be-\al}, \be \right).
    \end{equation}
\end{propo}
\begin{proof}
    By a direct computation 
    we have
    \begin{equation}
         \psi_{2,2} \circ \mu_{2,2} \circ\psi_{2,2}^{-1}((\al, \be)) = \left(\frac{1}{\be-\al},\be\right).
    \end{equation}

\end{proof}

Below we include the dynamics $L_{2,2}$ into a flow induced by $H_{2,2}(x,y) = I_{2,2}(x,y,1)$.
\begin{propo}\label{prop:hamdyn22}
    The dynamics $\mu_{2,2}$ on a generic level set $\{H_{2,2}(x,y) =\beta\}$ corresponds to moving along an integral curve of the hamiltonian vector field $v_{H_{2,2}}$ with the time $t = \frac{\de}{\sinh(\de)}$, where $\de$ is any solution of the equation $2\cosh(\de) = \be$.
\end{propo}
\begin{proof}
    We use the coordinate $\al = \frac{x-\ri}{y}$. One computes
    \begin{equation}\label{eqn:v22}
        v_{H_{2,2}}(\al) = \{H_{2,2}, \al\} = \al^2 - \al\be +1.
    \end{equation}
    
    This is a vector field corresponding to the one-parametric group of linear fractional transformations. The corresponding element in the Lie algebra $\mathfrak{sl}_2(\C)$ is $\xi = \begin{pmatrix}
    -\frac{\be}{2} & 1\\
     -1& \frac{\be}{2}
    \end{pmatrix}.
    $
    Direct computation gives 
    $$
    \exp(t\xi) = \frac{1}{\sinh(\de)}\begin{pmatrix}
         \sinh (\delta -t \sinh (\delta )) & \sinh (t \sinh (\delta )) \\
 -\sinh (t \sinh (\delta )) & \sinh (\delta +t \sinh (\delta )) \\
    \end{pmatrix}.
    $$

    The corresponding LFT equals $L_{2,2}$ if and only if $t = \frac{\de + \pi \ri n}{\sinh(\de)}$ for $n\in \Z$. Alternatively, the equation \eqref{eqn:v22} can be solved directly as a differential equation on a function $\al(t)$.
\end{proof}
\subsection{The case $p = 4, q = 1$.}

\begin{propo}\label{prop:22int(2)}\cite{CHEN202425}
There is a mutation invariant Laurent polynomial $I_{4,1}(x_0,x_1,x_2) = \frac{x_0^4+x_2^2 \left(x_1+x_2\right){}^2}{x_0^2 x_1  x_2}.$ In particular,
    $$\mu_{4,1}^*\left(\frac{x_0^4+x_2^2 \left(x_1+x_2\right){}^2}{x_0^2x_1x_2}\right) = \frac{x_0^4+x_2^2 \left(x_1+x_2\right){}^2}{x_0^2 x_1  x_2}.$$
\end{propo}

In order to integrate the dynamics induced by $\mu_{4,1}$ we introduce the coordinates as follows. A generic level set of the conserved quantity $I_{4,1}$ function is an irreducible quartic $$C_{\be} = \{x_0^4+x_2^2 \left(x_1+x_2\right){}^2- \be x_0^2 x_1  x_2 = 0\}.$$ For generic values of $\be \in \C$ the curve $C_{\be}$ has two singular points $[0:-1:1]$ and $[0:1:0]$, both have multiplicity $2$. The point $[0:-1:1]$ is a double point. The point $[0:1:0]$ is a cusp with the tangent line given by $x_2 = 0$. Note also that for any $\om$ such that $\om^4+1=0$ and for any $\be\in \C$ the curve $C_{\be}$ passes through the point $[1:0:\om] = [\om^{-1}:0:1]$. In this section we fix $\om = e^{\frac{\pi\ri}{4}}$.

Consider the pencil of conics $\{R_{\al}\}_{\al\in \CP^1}$ given by the following conditions
\begin{itemize}
    \item The conic $R_{\al}$ passes through points $[0:-1:1], [1:0:\om], [0:1:0]$.
    \item The tangent line to $R_{\al}$ at the point $[0:1:0]$ is the line $\{x_2 = 0\}$.
\end{itemize}

 Explicitly the pencil is given by $R_{\al} = \{x_2(x_1+x_2) - \ri x_0^2 - \al x_0(x_2 - \om x_0) =0 \}$.
 
\begin{lemma}\label{lemma:aff_41}
    For generic values of $\al, \be$ the curves $C_\be, R_{\al}$ intersect at points $[0:-1:1], [1:0:\om], [0:1:0]$ with multiplicities $2, 1$ and $4$ respectively.
\end{lemma}

\cref{lemma:aff_41} implies that the parameter $\al$ of the pencil $\{R_{\al}\}_{\al\in\C}$ gives a rational parametrization of $C_{\be}$ for generic $\be$. Below we give an explicit formulas for this parametrization.

Define a birational map $\psi_{4,1}: \CP^2 \dashrightarrow \C^2$ by
\begin{equation}
    \psi_{4,1}: [x_0:x_1:x_2] \mapsto  \left(\frac{x_2(x_1+x_2) -\ri x_0^2}{x_0(x_2-\om x_0)} , \frac{x_0^4+x_2^2 \left(x_1+x_2\right){}^2}{x_0^2x_1x_2}\right).\\
\end{equation}
We use the functions $(\al,\be) = \left(\frac{x_2(x_1+x_2) -\ri x_0^2}{x_0(x_2-\om x_0)} , \frac{x_0^4+x_2^2 \left(x_1+x_2\right){}^2}{x_0^2x_1x_2}\right)$
as local coordinates.

\begin{propo}\label{prop:41int}
    The map $\psi_{4,1}$ conjugates the dynamics induced by $\mu_{4,1}$ to the one given by
    \begin{equation}
        L_{4,1}:(\al, \be) \mapsto \left(\frac{\ri(\sqrt{2} + \be \om) \al +\be}{\al - \sqrt{2}}, \be\right).
    \end{equation}
\end{propo}
\begin{proof}
    By a tedious but straightforward computation.
\end{proof}
Below we include the dynamics $L_{4,1}$ into a flow induced by $H_{4,1}(x,y) = I_{4,1}(x,y,1)$.

\begin{propo}
    The dynamics $\mu_{4,1}$ on a generic level set $\{H_{4,1}(x,y) =\beta\}$ corresponds to moving along an integral curve of the hamiltonian vector field $v_{H_{4,1}}$ with the time $t = \frac{\de}{2\sinh(\de)}$, where $\de$ is any solution of the equation $2\cosh(\de) = \be$.
\end{propo}
\begin{proof}
    We use the coordinate $\al = \frac{-i x^2+y+1}{x (1-x \omega )}$ on a generic level set $\{ H_{4,1}(x,y) = \be\}$ for $\be\in \C $ generic. One computes
    \begin{equation}\label{eqn:v41}
        v_{H_{4,1}}(\al) = \{H_{4,1}, \al\} = \omega \al^2 + (\be - 2\ri) \al - \om \be.
    \end{equation}
    The rest of the proof is similar to the one of  \cref{prop:hamdyn22}.
\end{proof}

\cref{prop:aff_22,prop:41int} imply that for the affine pairs $(p,q)\in \mathcal{A}$ the dynamics of $\mu_{p,q}$ is birationally conjugate to the dynamics on $\CP^1 \times \C$ induced by a linear fractional transformation acting on the first factor.
\begin{coro}
    Let $(p,q) \in \mathcal{A}$, then the dynamical degree of $\mu_{p,q}$ equals to one.
\end{coro}

\section{Stable Model}\label{sec:stable_model}
In \cref{sec:affine_cases} we studied the maps  $\mu_{p,q}$ for $(p,q)\in \mathcal{A}$. Now, we are interested in the pairs $(p,q)\in\N^2$ defined by the condition $pq>4$. These pairs are called non-affine in \cite{CHEN202425}. From the perspective of our work it is natural to further divide them into two families.

\begin{itemize}
    \item[(1)] The pairs $(p,q)$ such that $pq > p +q$. We call these pairs semi-simple. 
    \item[(2)] The pairs $(p,q)$ such that $p+q \geq pq > 4$, or equivalently $\min(p,q) = 1$ and $\max(p,q) > 4$. we call these pairs non-semi-simple.
\end{itemize}

\begin{rem}\label{rem:pgeqq}
    Notice that the maps $\mu_{p,q}$ and $\mu_{q,p}$ are birationally conjugated. More precisely we have $$\mu_{q,p}=\Psi  \circ \mu_{p,q}\circ \Psi^{-1},$$ where $\Psi(x,y):= \mu_{p,q}^{(2)}(y,x)$. Therefore we can always assume without losing generality that $p \geq q$.
\end{rem}

The goal of this section is to construct a proper modification $\pi:X_{p,q}\to \mathbb{P}^2$ such that the lift of $\mu_{p,q}:\mathbb{P}^2\dashrightarrow \mathbb{P}^2$ to $X_{p,q}$ will be algebraically stable. In this framework where we will set our results, in particular the computation of the dynamical degree for $\mu_{p,q}$.

\subsection{Compactification}\label{subsec:compactification}

Recall the formula 
\begin{equation*}\label{eq: c^2 formula for mutation}
    \mu_{p,q}: (x,y) \mapsto \left(\frac{1+y^q}{x},\frac{x^p+(1+y^q)^p}{x^py}\right).
\end{equation*}
We extend the birational dynamics induced by $\mu_{p,q}$ to $\CP^2$. Let $(p,q)$ be semi-simple, then

\begin{equation}\label{eqn: 5}
    \mu_{p,q}([x_0\colon x_1\colon x_2])=\left[x_0^{p - 1}x_1x_2^{p(q - 1) - q}\left(x_1^q + x_2^q\right)\colon \left(x_1^q + x_2^q\right)^p + x_0^px_2^{p (q - 1)}\colon 
  x_0^px_1x_2^{p (q - 1) - 1}\right].
\end{equation}
For the case of a non-semi-simple pair $(p,1)$ the map $\mu_{p,1}$ extends to $\CP^2$ as follows
 \begin{equation}\label{eq:mu_cp2}
\mu_{p,1}([x_0\colon x_1 \colon x_2])=\left[x_0^{p - 1}x_1(x_1 + x_2)\colon x_2\left(x_0^p+ (x_1 + x_2)^p\right)\colon 
  x_0^px_1\right].
\end{equation}

\begin{rem}
  Notice that the dynamical behavior of the point $[0\colon1\colon 0]$ is not the same for all cases. In particular we will see that for the semi-simple pairs $[0:1:0]$ is a superattracting fixed point for $\mu_{p,q}$. For non-semi-simple pairs $(p,q)$ this point is part of a destabilizing orbit. This makes the construction of $X_{p,q}$ dependent on the parameters $p$ and $q$ (see \cref{subsec:alg_stab}).
\end{rem}

\subsection{Indeterminacy and Critical Sets.}\label{subsec:indeterminacy}
In this section, we compute the indeterminacy set and critical set for the map $\mu_{p,q}$ for $pq>4$. We use these sets to identify and eliminate the destabilizing orbits in order to reach algebraic stability; cf.  \cref{subsec:alg_stab}.

We use the following notation: For any $n\in \N$ denote $R_{n} = \{z\in \C\; |\;  z^n+1=0\}$, denote $S_{p} = R_p \times \{0\} \subset \C^2 \subset \CP^2$ and $T_q = \{0\}\times R_q \subset \C^2 \subset \CP^2$, where $\C^2$ is embedded by $(x,y)\mapsto [x:y:1]$. For any $n$ we will denote by $\iota$ the involution on $R_n$ which sends $\xi \mapsto \xi^{-1}$ and the corresponding involutions on the sets $S_p, T_q$.

\begin{lemma}\label{Lemma: Indeterminate set} Let $pq>4$ then
  \begin{itemize}
\item If a pair $(p,q)$ is semi-simple then the indeterminacy locus is given by

\begin{equation*}
    \mathcal{I}(\mu_{p,q}) 
    = \{[1:0:0]\} \cup S_p \cup T_q.
\end{equation*}

\item If a pair $(p,1)$ is non-semi-simple then the indeterminacy locus is given by

\begin{equation*}
    \mathcal{I}(\mu_{p,1})
    = \{[1:0:0], [0:1:0]\} \cup S_p \cup T_{1}.
\end{equation*}
\end{itemize}
  
\end{lemma}
\begin{proof}
    By a straightforward computation.
\end{proof}

\begin{lemma}\label{Lemma1: Critical set}
    Let $pq>4$. Then the critical set is given by 
    \begin{equation*}
         \displaystyle\mathcal{C}(\mu_{p,q})=\mu_{p,q}^{-1}([0\colon 1\colon 0])\cup\bigcup_{P\in S_p\cup T_q}\mu_{p,q}^{-1}(P).
    \end{equation*}
 Moreover, 
 \begin{align*}
     \mu_{p,q}^{-1}([0\colon 1\colon 0])&=\begin{cases}\displaystyle
     \bigcup_{i=0}^2\{x_i=0\}, \; q\neq1,
         \\\displaystyle
     \bigcup_{i=0}^1\{x_i=0\}, q=1.
     \end{cases}, \\
     \mu_{p,q}^{-1}([\omega:0:1])&=K_{\omega}^h, \qquad \mbox{and} \qquad \mu_{p,q}^{-1}([0:\zeta:1])=K_{\zeta}^v.
 \end{align*}
 where $K^\textit{h}_\omega:=\{x_0x_2^{q-1}-\omega^{-1}(x_1^q + x_2^q)=0\},\;\;  K^\textit{v}_\zeta:=\{x_2-\zeta x_1=0\}.$
    \end{lemma}
    \begin{proof}   
        The solution of $\mbox{det}(D \hat{\mu}_{p,q})=0$ gives us the critical locus $\mathcal{C}(\mu_{p,q})$. It remains to restrict the map $\mu_{p,q}$ to this set to obtain the corresponding critical values. 
    \end{proof}

We illustrate with an example the critical behavior of $\mu_{p,q}$ for a semi-simple pair on  \cref{fig:crit_mu23}.

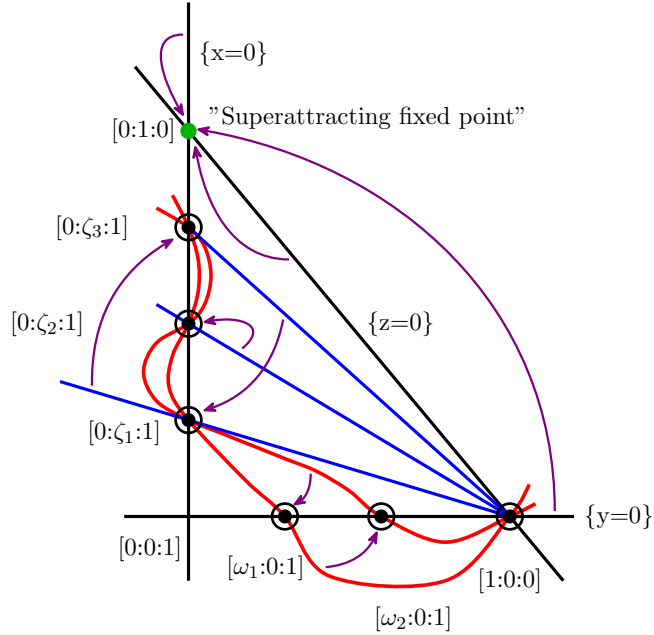
\begin{figure}[h]
    \centering
    \begin{tikzpicture}[scale=0.85]

 \begin{axis}[
 hide axis,
 xmax=6, xmin=-1,
 ymax=7, ymin=-1,
 x=1cm,
 y=1cm,
 >=stealth,
 clip=false,
 ]

\addplot[red,smooth,
    tension=0.6,
    ultra thick,
    name path=plot]
coordinates{
(5.3,0.5)(5,0) (4,-1) (2.2,-0.93) (1.5,0)  (0.9, 0.55) (0, 1.5) (-0.3,2)  (-0.25,2.5) (0,3) (0.15,3.5) (0.15,4) (0,4.5) (-0.25, 5)};

    \addplot[red,smooth,
    tension=0.6,
    ultra thick,
    name path=plot]
coordinates{
(5.4,0.2)(5,0) (4,-0.4) (3,0) (2.5,0.4) (2, 0.7) (1.5, 0.9) (1,1.1) (0.5, 1.3) (0, 1.5) (-0.5, 1.8)  (-0.7,2.25)  (-0.5,2.7) (0,3) (0.25,3.3) (0.35, 3.75)  (0.25, 4.2) (0,4.5) (-0.5, 4.8)};

\draw[line width=1.4pt] (5.833,-1) -- (-0.833,7);
\draw[line width=1.4pt] (-1,0) -- (6,0);
\draw[line width=1.4pt] (0,-1) -- (0,8);
\draw[blue,line width=1.4pt] (5,0) -- (-2,2.1);
\draw[blue,line width=1.4pt] (5,0) -- (-0.5,3.3);
\draw[blue,line width=1.4pt] (5,0) -- (0,4.5);

\node[fill=black,circle,scale=0.6] at (0,1.5){};
\node[fill=black,circle,scale=0.6] at (0,3){};
\node[fill=black,circle,scale=0.6] at (0,4.5){};
\node[fill=black,circle,scale=0.6] at (1.5,0){};
\node[fill=black,circle,scale=0.6] at (3,0){};
\node[fill=black!30!green,circle,scale=0.7] at (0,6){};
\node[fill=black,circle,scale=0.6] at (5,0){};
\draw[black, very thick] (5,0) circle (0.2cm);
\draw[black, very thick] (0,1.5) circle (0.2cm);
\draw[black, very thick] (0,3) circle (0.2cm);
\draw[black, very thick] (0,4.5) circle (0.2cm);
\draw[black, very thick] (1.5,0) circle (0.2cm);
\draw[black, very thick] (3,0) circle (0.2cm);

\node (1) at (-1.5,1.9) {};
\node (2) at (0.9-1,5.4-1) {};
\node (3) at (2.5-1,4.2-1) {};
\node (4) at (1.1-1,2.6-1) {};
\node (5) at (1.7-1,3.5-1) {};
\node (6) at (1.1-1,4.0-1) {};
\node (7) at (3-1,0.2-1) {};
\node (8) at (4-1,0.9-1) {};
\node (9) at (2.9-1,1.8-1) {};
\node (10) at (2.5-1,1.1-1) {};
\node (11) at (2.7-1,5-1) {};
\node (12) at (1.08-1,6.9-1) {};
\node (13) at (1.05-1,8.5-1) {};
\node (14) at (1-1,7.1-1) {};
\node (15) at (6.7-1,0.95-1) {};
\node (16) at (1-1,7.05-1) {};

\draw[violet,->,-{Stealth[round]},line width=1pt] (1) to[bend left] (2);
\draw[violet,->,-{Stealth[round]},line width=1pt] (3) to[bend left] (4);
\draw[violet,->,-{Stealth[round]},line width=1pt] (5) to[in=10, out=40,looseness=2] (6);
\draw[violet,->,-{Stealth[round]},line width=1pt] (7) to[bend right] (8);
\draw[violet,->,-{Stealth[round]},line width=1pt] (9) to[bend left] (10);
\draw[violet,->,-{Stealth[round]},line width=1pt] (11) to[in=-80, out=-180,looseness=1] (12);

\draw[violet,->,-{Stealth[round]},line width=1pt] (13) to[in=120, out=180,looseness=1.2] (14);

\draw[violet,->,-{Stealth[round]},line width=1pt] (15) to[in=-10, out=90,looseness=1] (16);
 \node (A) at (-0.6,-0.5) {[0{:}0{:}1]};
 \node (B) at (5,-1) {[1{:}0{:}0]};
 \node (C) at (-0.7,6) {[0{:}1{:}0]};
 \node (C2) at (3-0.2,6.2) {"Superattracting fixed point"};
 \node (D) at (6.7,0) {\{y=0\}};
 \node (E) at (0.7,7.2) {\{x=0\}};
\node (F) at (3.3,3) {\{z=0\}};
\node (G) at (-2.2,3) {$[0{:}\zeta_2{:}1]$};
\node (H) at (-1,1.3) {$[0{:}\zeta_1{:}1]$};
\node (I) at (-1.5,4.5) {$[0{:}\zeta_3{:}1]$};
\node (J) at (1.2,-0.8) {$[\omega_1{:}0{:}1]$};
\node (K) at (3.5,-1.6) {$[\omega_2{:}0{:}1]$};
\end{axis}

\end{tikzpicture}
\caption{Critical behavior of $\mu_{2,3}$.}
    \label{fig:crit_mu23}
\end{figure}

The components $K^{h}_{\omega},\; \omega\in R_2$ and $K^{v}_{\zeta},\; \zeta\in R_{3}$ are drawn in red and blue respectively. 
The violet arrows show to where the components of the critical locus of $\mu_{2,3}$ are collapsed.

\begin{rem}\label{rem: fixed point}
    Due to  \cref{Lemma: Indeterminate set,Lemma1: Critical set} for a non-semi-simple pair every critical value induces a destabilizing orbit. For a semi-simple pair, the critical value $[0:1:0]$ is fixed. This is the only critical value which does not induce a destabilizing orbit.
\end{rem}
\begin{lemma}\label{Lemma: Superattracting semi-simple}
    Let $(p,q)$ be a semi-simple pair. Then the map $\mu_{p,q}:\CP^2\dashrightarrow\CP^2$ has a superattracting fixed point at $[0:1:0]$.
\end{lemma}
\begin{proof}
   For a semi-simple pair $(p,q)$ the point $\displaystyle [0\colon1\colon0]$ is a fixed point under $\mu_{p,q}$. We claim that it is superattracting. To see this, on affine coordinates $\displaystyle a=\frac{x_0}{x_1}$, $\displaystyle b=\frac{x_2}{x_1}$ Equation \eqref{eqn: 5} becomes
 \begin{equation*}
    \mu_{p,q}(a,b)=(\tilde{a},\tilde{b})=\left(\frac{a^{p - 1}b^{p(q - 1) - q}\left(1 + b^q\right)}{\left(1 + b^q\right)^p + a^pb^{p (q - 1)}},\frac{a^pb^{p (q - 1) - 1}}{\left(1 + b^q\right)^p + a^pb^{p (q - 1)}}\right).
\end{equation*}
Therefore, $D(\mu_{p,q})_{(0,0)}=\begin{bmatrix}
        0&0\\
        0&0
    \end{bmatrix}$, and hence $[0\colon1\colon0]$ is superattracting. 
\end{proof}
\begin{rem}\label{remark:superattracting_fixed_point}
    Notice that for $(p,q)$ semi-simple then $[0:1:0]$ is not part of any destabilizing orbit  so after we construct the stable model $X_{p,q}$  in \cref{subsec:stable_model_ss} we will have that the map $\mu_{p,q}:X_{p,q}\dashrightarrow X_{p,q}$ will have a superattracting fixed point at the lift of $[0:1:0]$. 
\end{rem}

Notice that all the minimal destabilizing orbits have the form $C \to P$, where $C$ is a component of the critical set $\mathcal{C}(\mu_{p,q})$ and $P$ is a point in the indeterminacy locus $\mathcal{I}(\mu_{p,q})$. For a semi-simple pair $(p,q)$, there are $p+q$ different critical values and $p+2$ for a non-semi-simple pair $(p,1)$.

In the two following subsections, we construct the stable model space $X_{p,q}$ resulting from blowing up over the critical values inducing destabilizing orbits. For big values of $p$ and $q$ defining charts for blow-ups is a notational challenge. To avoid this problem, we imitate the formal construction of a blow-up that appears on \cite{Hu1997}, which for the maps $\mu_{p,q}$ has the advantage of resulting in at most 4 blowups for any choice of $(p,q)$.

\subsection{Stable Model, semi-simple pair $(p,q)$.}\label{subsec:stable_model_ss}
In this case, the space $X_{p,q}$ is constructed by blowing up at $S_p\cup T_q$ only. Below we introduce the notations for the local coordinates and show that the exceptional lines at blowups are not collapsed by the lift of $\mu_{p,q}$ to $X_{p,q}$. We show that this is enough to conclude algebraic stability using  \cref{lemma:nondestcurvecrit,cor:alg_stab_1,cor:alg_stab_2}.
\subsubsection{Blow-up at points in $T_q$}\label{subsubsec:comp_Tq} For this we will use the chart at $U_2=\{x_2\neq 0\}$. The points of the form $[0\colon \zeta\colon 1]$ where $\zeta^q+1=0$ are clearly the locus $Y_1=f^{-1}(0)$ of the map $f:U_2\to \C^2$ where $f\left([x\colon y\colon 1]\right):=(x,y^q+1)$, so by \cite{Hu1997} the blow-up is an open subset of

\begin{equation*}
    (U_2)'_{Y_1}=\left\{\left([x:y:1],[A_0, B_0]\right)\ |\  (x,y^q+1)\in [A_0\colon B_0]\right\}.
\end{equation*}
Set local coordinates $(v_0,u_0) =(\frac{A_0}{B_0},y)$.
We denote the projection $\beta_q: (U_2)'_{Y_1} \to U_2.$
We highlight that this procedure is equivalent to visiting each critical value in $T_q$ and performing a simple blow-up at each point separately due to the Universal Principle for blow ups.
The lift of the map $\mu_{p,q}$ in the local coordinates $(u_0,v_0)$ is given by
\begin{equation}
    \mu_{p,q}(u_0,v_0)=(\tilde{u}_0,\tilde{v}_0)=\left(\frac{1 + v_0^p}{v_0^pu_0},\frac{v_0^{pq-1}u_0^q}{v_0^{pq}u_0^q+(1+v_0^p)^q}\right).
\end{equation}

\begin{rem}\label{rem:lift_v}
    Notice that the component  $\{u_0=\zeta\}$ of the exceptional divisor is mapped by the lift of $\mu_{p,q}$ birationally to the proper transform of a component $\{x_2(x_0^p+x_2^p)=\zeta x_1x_2^p\}$ of $\mathcal{C}(\mu^{-1}_{p,q}).$
    The image of $K_{\zeta}^{\textit{v}}$ under $\beta_q^{-1}\circ \mu_{p,q}$ is a dense subset of the component $\{u_0 = \zeta\}$ of the exceptional line. 
    Therefore, the proper transforms of exceptional lines over the points in $T_q$ and the proper transforms of curves $K_{\zeta}^{\textit{v}}$ in $X_{p,q}$ are not collapsed by the lift of $\mu_{p,q}$.
 \end{rem}   
    
\subsubsection{Blow-up at points in $S_p$}\label{subsubsec:comp_Sp} In a similar fashion here we can work over $U_2$ and blow-up along the locus $Y_2$ of the map $[x\colon y \colon1]\mapsto(x^p+1,y)$. Denote

\begin{equation*}
    (U_2)'_{Y_2}=\left\{\left([x:y:1],[C_0, D_0]\right)| (x^p+1,y)\in [C_0\colon D_0]\right\}.
\end{equation*}
Set local coordinates $(r_0,s_0) =\left(x,\frac{D_0}{C_0}\right)$. We denote the projection $\varphi: (U_2)'_{Y_2} \to U_2.$
The lift of the map $\mu_{p,q}$ in the local coordinates $(r_0,s_0)$ is given by
  \begin{equation}
      \mu_{p,q}(r_0,s_0)=(\tilde{r}_0,\tilde{s}_0)=\left(\frac{s_0^q(r_0^p+1)^q + 1}{r_0},\frac{1+\displaystyle\sum_{k=1}^p\binom{p}{k}s_0^{kq}(r_0^p+1)^{qk-1}}{r_0^ps_0(r_0^p+1)}\right).
  \end{equation}

  \begin{rem}\label{rem: exceptional divisor}
      Notice that the component $\{r_0=\omega\}$ of the exceptional divisor is mapped by the lift of $\mu_{p,q}$ birationally to the proper transform of a component $\{x_0=\omega^{-1}x_2\}$ of $\mathcal{C}(\mu^{-1}_{p,q}).$ 
      It's also a similar check that the proper transforms of curves $K^h_{\omega}$ are not collapsed by the lifted map.
     
  \end{rem}

\subsection{Stable Model, non-semi-simple pair $(p,1)$.}\label{subsec:stable_model_nss}
In this case the space $X_{p,1}$ is constructed as a simple blowup at points in $S_p\cup T_1$ and two blowups at $[0:1:0]$. Below we introduce the notations for the local coordinates and explain that the exceptional lines at blowups either are not collapsed by the lift of $\mu_{p,1}$ or are collapsed to fixed points. This implies algebraic stability.
\subsubsection{Blow-ups at $P=[0\colon 1\colon 0]$}\label{ssec: coordinates for [0:1:0]}
For this we will use the chart at $U_1=\{x_1\neq 0\}$. The point $P=[0\colon 1\colon 0]$ is the locus of the Identity map on $U_1\cong \C^2$. The blow-up is an open subset of
\begin{equation*}
    (U_1)'_{P}=\left\{\left([x\colon1\colon z],[A_1\colon B_1]\right)\ |\ (x,z)\in [A_1\colon B_1]\right\}.
\end{equation*}
Set local coordinates $(u_1,v_1) = \left(x,\frac{B_1}{A_1}\right)$. We denote the projection $\pi_1: (U_1)'_{P} \to U_1.$
The lift of the map $\mu_{p,q}$ in the local coordinates $(u_1,v_1)$ is given by
\begin{equation}
    \mu_{p,1}(u_1,v_1)=(\tilde{u}_1, \tilde{v}_1)=\left(\frac{u_1^{p-2}(1+u_1v_1)}{v_1(u_1^p+(1+u_1v_1)^p)},\frac{u_1}{1+u_1v_1}\right).
\end{equation}
Notice that the exceptional divisor $\{u_1=0\}$ is mapped to the point $(0,0)$ in affine coordinates which is indeterminate, so we need to do a second blow-up there. 
We need to blow up the point $(u_1,v_1) = (0,0)$ as its orbit is clearly a minimal destabilizing orbit. We blow-up this point and use the local coordinates $(u_2,v_2)=(v_1, \frac{u_1}{v_1})$. The lift of the map to the blow-up space is

\begin{equation}\label{Equation Other super attracting}
    \mu_{p,1}(u_2,v_2)=(\tilde{u}_2, \tilde{v}_2)=\left(\frac{u_2v_2}{1+u_2^2v_2},\frac{u_2^{p-4}v_2^{p-3}(1+u_2^2v_2)^2}{u_2^pv_2^p+(1+u_2^2v_2)^p}\right).
\end{equation}
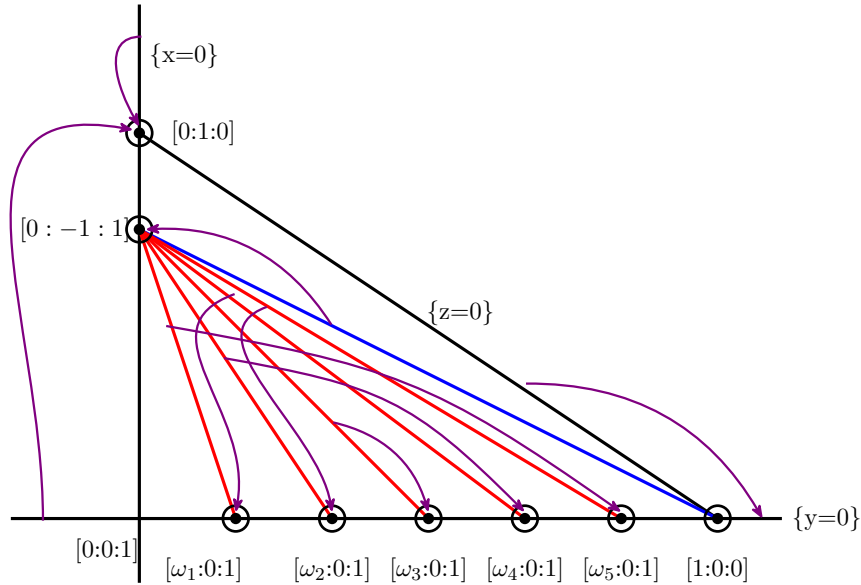
\begin{figure}[H]
    \centering
    \begin{tikzpicture}[every node/.style={font=\normalsize,
    inner sep=0.2pt,
    outer sep=0pt}, scale=0.85]

 \begin{axis}[
 hide axis,
 xmax=6, xmin=-1,
 ymax=7, ymin=-1,
 x=1cm,
 y=1cm,
 >=stealth,
 clip=false,
 ]

\draw[line width=1.4pt] (9,0) -- (0,6);
\draw[line width=1.4pt] (-2,0) -- (10,0);
\draw[line width=1.4pt] (0,-1) -- (0,8);
\draw[blue,line width=1.4pt] (9,0) -- (0,4.5);
\draw[red,line width=1.4pt] (1.5,0) -- (0,4.5);
\draw[red,line width=1.4pt] (3,0) -- (0,4.5);
\draw[red,line width=1.4pt] (4.5,0) -- (0,4.5);
\draw[red,line width=1.4pt] (6,0) -- (0,4.5);
\draw[red,line width=1.4pt] (7.5,0) -- (0,4.5);

\node[fill=black,circle,scale=5] at (0,4.5){};
\node[fill=black,circle,scale=5] at (1.5,0){};
\node[fill=black,circle,scale=5] at (3,0){};
\node[fill=black,circle,scale=5] at (4.5,0){};
\node[fill=black,circle,scale=5] at (6,0){};
\node[fill=black,circle,scale=5] at (7.5,0){};
\node[fill=black,circle,scale=5] at (0,6){};
\node[fill=black,circle,scale=5] at (9,0){};
\draw[black, very thick] (0,4.5) circle (0.2cm);
\draw[black, very thick] (1.5,0) circle (0.2cm);
\draw[black, very thick] (3,0) circle (0.2cm);
\draw[black, very thick] (4.5,0) circle (0.2cm);
\draw[black, very thick] (6,0) circle (0.2cm);
\draw[black, very thick] (7.5,0) circle (0.2cm);
\draw[black, very thick] (9,0) circle (0.2cm);
\draw[black, very thick] (0,6) circle (0.2cm);

\node (1) at (6,2.1) {};
\node (2) at (9.7,0) {};
\node (3) at (3,3) {};
\node (4) at (0.1,4.5) {};
\node (5) at (0.4,3) {};
\node (6) at (7.5,0.1) {};
\node (7) at (1.3,2.5) {};
\node (8) at (6,0.1) {};
\node (9) at (3,1.5) {};
\node (10) at (4.5,0.1) {};
\node (11) at (2,3.3) {};
\node (12) at (3,0.1) {};
\node (13) at (1.05-1,8.5-1) {};
\node (14) at (1-1,7.1-1) {};
\node (15) at (-1.5,0.95-1) {};
\node (16) at (-0.1,7.05-1) {};
\node (17) at (1.5,3.5) {};
\node (18) at (1.5,0.1) {};

\draw[violet,->,-{Stealth[round]},line width=1pt] (1) to[bend left] (2);
\draw[violet,->,-{Stealth[round]},line width=1pt] (3) to[bend right] (4);
\draw[violet,->,-{Stealth[round]},line width=1pt] (5) to[in=145, out=-10,looseness=1.2] (6);
\draw[violet,->,-{Stealth[round]},line width=1pt] (7) to[in=135, out=-10,looseness=1.2] (8);
\draw[violet,->,-{Stealth[round]},line width=1pt] (9) to[bend left] (10);
\draw[violet,->,-{Stealth[round]},line width=1pt] (11) to[in=100, out=-160,looseness=1] (12);

\draw[violet,->,-{Stealth[round]},line width=1pt] (13) to[in=120, out=180,looseness=1.2] (14);

\draw[violet,->,-{Stealth[round]},line width=1pt] (15) to[in=170, out=90,looseness=1.3] (16);
\draw[violet,->,-{Stealth[round]},line width=1pt] (17) to[in=80, out=-160,looseness=1.2] (18);
 \node (A) at (-0.5,-0.5) {[0{:}0{:}1]};
 \node (B) at (9,-0.8) {[1{:}0{:}0]};
 \node (C) at (1,6) {[0{:}1{:}0]};
 
 \node (D) at (10.7,0) {\{y=0\}};
 \node (E) at (0.7,7.2) {\{x=0\}};
\node (F) at (5,3.2) {\{z=0\}};
\node (I) at (-1,4.5) {$[0:-1:1]$};
\node (J) at (1,-0.8) {$[\omega_1{:}0{:}1]$};
\node (K) at (3,-0.8) {$[\omega_2{:}0{:}1]$};
\node (L) at (4.5,-0.8) {$[\omega_3{:}0{:}1]$};
\node (M) at (6,-0.8) {$[\omega_4{:}0{:}1]$};
\node (N) at (7.5,-0.8) {$[\omega_5{:}0{:}1]$};
\end{axis}

\end{tikzpicture}
    \caption{Critical behavior of $\mu_{5,1}$ before blowing up at $[0:1:0]$}
    \label{fig:placeholder}
\end{figure}

\begin{figure}[H]
    \centering
    \begin{tikzpicture}[scale=0.85]

 \begin{axis}[
 hide axis,
 xmax=6, xmin=-1,
 ymax=7, ymin=-1,
 x=1cm,
 y=1cm,
 >=stealth,
 clip=false,
 ]

\draw[line width=1.4pt] (9,0) -- (3.5,7);
\draw[line width=1.4pt] (-2,0) -- (10,0);
\draw[line width=1.4pt] (0,-1) -- (0,8);
\draw[dashed, very thick] (-1,5) -- (3,9);
\draw[dashed, very thick] (1,8.5) -- (3.5,7);
\draw[blue,line width=1.4pt] (9,0) -- (0,4.5);
\draw[red,line width=1.4pt] (1.5,0) -- (0,4.5);
\draw[red,line width=1.4pt] (3,0) -- (0,4.5);
\draw[red,line width=1.4pt] (4.5,0) -- (0,4.5);
\draw[red,line width=1.4pt] (6,0) -- (0,4.5);
\draw[red,line width=1.4pt] (7.5,0) -- (0,4.5);

\node[fill=black,circle,scale=0.6] at (0,4.5){};
\node[fill=black,circle,scale=0.6] at (1.5,0){};
\node[fill=black,circle,scale=0.6] at (3,0){};
\node[fill=black,circle,scale=0.6] at (4.5,0){};
\node[fill=black,circle,scale=0.6] at (6,0){};
\node[fill=black,circle,scale=0.6] at (7.5,0){};
\node[fill=black!30!green,circle,scale=0.6] at (1.9,7.9){};
\node[fill=black,circle,scale=0.6] at (9,0){};
\node[fill=black,circle,scale=0.6] at (3.5,7){};
\draw[black, very thick] (0,4.5) circle (0.2cm);
\draw[black, very thick] (1.5,0) circle (0.2cm);
\draw[black, very thick] (3,0) circle (0.2cm);
\draw[black, very thick] (4.5,0) circle (0.2cm);
\draw[black, very thick] (6,0) circle (0.2cm);
\draw[black, very thick] (7.5,0) circle (0.2cm);
\draw[black, very thick] (9,0) circle (0.2cm);
\draw[black, very thick] (3.5,7) circle (0.2cm);

\node (1) at (7.3,2.1) {};
\node (2) at (9.7,0) {};
\node (3) at (3,3) {};
\node (4) at (0.1,4.5) {};
\node (5) at (0.4,3) {};
\node (6) at (7.5,0.1) {};
\node (7) at (1.3,2.5) {};
\node (8) at (6,0.1) {};
\node (9) at (3,1.5) {};
\node (10) at (4.5,0.1) {};
\node (11) at (1.8,3.3) {};
\node (12) at (3,0.1) {};
\node (13) at (1.05-1,8.5-1) {};
\node (14) at (1.9,7.9) {};
\node (15) at (-1.5,0.95-1) {};
\node (16) at (-0.5,5.5) {};
\node (17) at (1.7,3.6) {};
\node (18) at (1.5,0.1) {};
\node (19) at (1,7) {};
\node (21) at (3,7.3) {};

\draw[violet,->,-{Stealth[round]},line width=1pt] (1) to[bend left] (2);
\draw[violet,->,-{Stealth[round]},line width=1pt] (3) to[bend right] (4);
\draw[violet,->,-{Stealth[round]},line width=1pt] (5) to[in=145, out=-10,looseness=1.2] (6);
\draw[violet,->,-{Stealth[round]},line width=1pt] (7) to[in=135, out=-10,looseness=1.2] (8);
\draw[violet,->,-{Stealth[round]},line width=1pt] (9) to[bend left] (10);
\draw[violet,->,-{Stealth[round]},line width=1pt] (11) to[in=90, out=-130,looseness=1] (12);
\draw[violet,->,-{Stealth[round]},line width=1pt] (17) to[in=80, out=-160,looseness=1.2] (18);
\draw[violet,->,-{Stealth[round]},line width=1pt] (13) to[in=180, out=80,looseness=1.2] (14);

\draw[violet,->,-{Stealth[round]},line width=1pt] (15) to[in=170, out=90,looseness=1] (16);
\draw[violet,->,-{Stealth[round]},line width=1pt] (19) to[in=-110, out=0,looseness=1] (14);
\draw[violet,->,-{Stealth[round]},line width=1pt] (21) to[in=-70, out=-110,looseness=1] (14);

\node (A) at (-0.7,-0.5) {$[0{:}0{:}1]$};
 \node (B) at (9,-0.8) {[1:0:0]};
 \node (D) at (10.7,0) {\{y=0\}};
 \node (E) at (-1,7.2) {\{x=0\}};
\node (F) at (7,3.5) {\{z=0\}};
\node (g) at (4.9,8) {"Superattracting fixed point"};
\node (I) at (-1,4.5) {$[0{:}{-}1{:}1]$};
\node (J) at (1.5,-0.8) {$[\omega_1{:}0{:}1]$};
\node (K) at (3,-0.8) {$[\omega_2{:}0{:}1]$};
\node (L) at (4.5,-0.8) {$[\omega_3{:}0{:}1]$};
\node (M) at (6,-0.8) {$[\omega_4{:}0{:}1]$};
\node (N) at (7.5,-0.8) {$[\omega_5{:}0{:}1]$};
\end{axis}

\end{tikzpicture}
\caption{Critical behavior of $\mu_{5,1}$ after blowing up.}
    \label{
}
\end{figure}
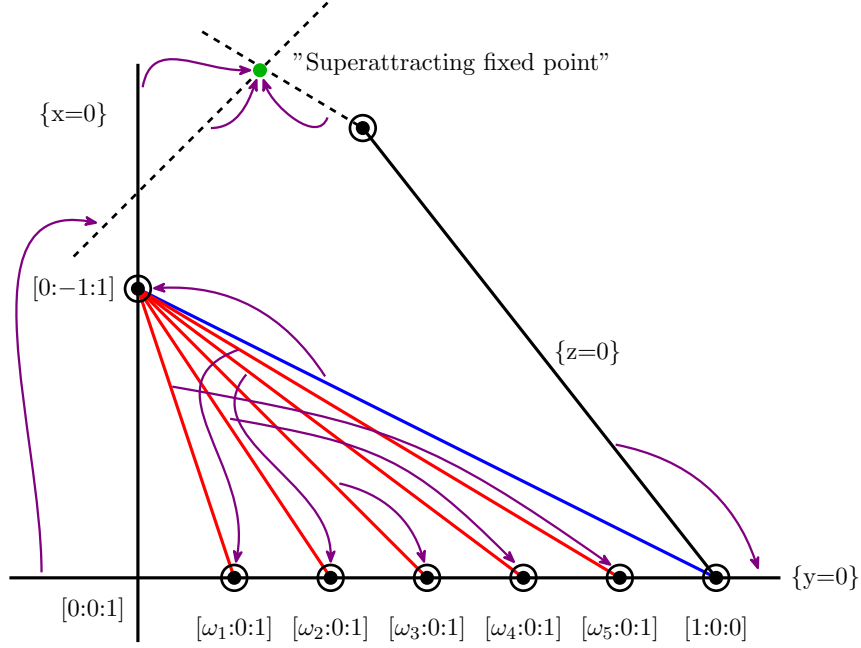
\begin{rem}\label{rem:nonss_res}
 Notice that here the exceptional divisor $\{u_2=0\}$ and the proper transform $\{v_2 = 0\}$ of the divisor $\{u_1=0\}$ are mapped to the fixed point $(0,0)$. As well as in previous examples, the proper transforms of the collapsing curves $K_{\omega}^{h}, K_{\zeta}^{v}$ are not collapsed. One can check explicitly that proper transforms of coordinate axes do not collapse to indeterminacies (cf. \cref{sec:tropicalization}).
So, we don't need to do further blow-ups over the point $[0:1:0]$, in particular, by the \cref{cor:alg_stab_1,cor:alg_stab_2}
we have eliminated the destabilizing orbits related to the point $[0:1:0]$, now we just need to blow-up the points in $S_p\cup \{[0\colon-1\colon 1]\}$ in an analogous way as we did for semi-simple pairs.
\end{rem}
If we let $Y$ be the surface obtained after these two blowups over $[0:1:0]$, then the point $(u_2,v_2)=(0,0)$ happens to be a superattracting fixed point. The reader can think of this Lemma as complementary to \cref{Lemma: Superattracting semi-simple} .
\begin{lemma}\label{Lemma: superattracting non-semisimple}
 Let $(p,1)$ be a non-semi-simple pair. Then the map $\mu_{p,1}:Y\dashrightarrow Y$ has a superattracting fixed point at $(u_2,v_2)=(0,0)$.
\end{lemma}
\begin{proof}
 For a non-semi-simple pair $(p,1)$ recall local affine coordinates $(u_2, v_2)$ on $Y.$ 
 The point $(u_2,v_2) = (0,0)$ is fixed under $\mu_{p,1}$ and superattracting by a direct computation of the differential using Equation \eqref{Equation Other super attracting}.   
\end{proof}
We highlight that Lemma \ref{Lemma: superattracting non-semisimple} is true also for the lift to the stable model $\mu_{p,1}:X_{p,1}\dashrightarrow X_{p,1}$ that we will produce below as we only need to blow-up $Y$ at points far from $(u_2,v_2)=(0,0)$ to reach algebraic stability. This follows from the universal property of blow-ups.

\subsubsection{Blow-up at $[0:-1:1]$}\label{Section: u,v coordinates}

This can be solved in the same fashion as before, for this we will use chart $U_2$. We use the local blow-up coordinates $(u,v) = (x, \frac{y+1}{x})$.

The lift of the map $\mu_{p,1}$ in the local coordinates $(u,v)$ is given by
 \begin{equation}
     \mu_{p,1}(u,v)=(\tilde{u}, \tilde{v})=\left(v,\frac{u+v^{p-1}}{uv-1}\right)
 \end{equation}
   Analogously to the case of a semi-simple pair $(p,q)$ this blowup eliminates the destabilizing orbit at the point $[0:-1:1]$. 

\subsubsection{Blow-up at points in $S_p$} 
We use chart $U_2$. We use the local blow-up coordinates $(s_0,r_0) = (\frac{y}{x^p+1},x)$. 
The lift of the map $\mu_{p,1}$ in the local coordinates $(r_0,s_0)$ is given by
\begin{equation}
    \mu_{p,1}(r_0,s_0)=(\tilde{r}_0,\tilde{s}_0)=\left(\frac{s_0(r_0^p+1) + 1}{r_0},\frac{\sum_{k=1}^p\binom{p}{k}s_0^k(r_0^p+1)^{k-1}+1}{s_0(s_0(r_0^p+1) + 1)^p+s_0r_0^p}\right)
\end{equation}
Analogously to the case of a semi-simple pair $(p,q)$ this blowup eliminates destabilizing orbits at points in $S_p$.

From the  \cref{subsec:stable_model_ss,subsec:stable_model_nss} we obtain the following.

  \begin{propo}\label{propo:stable_model}
    Let $pq>4$. Then the lift of the map $\mu_{p,q}$ to $X_{p,q}$ is algebraically stable. 
  \end{propo}
  \begin{proof}
    It it straightforward to check that the proper transforms of curves in $\CP^2$ which are collapsed by $\mu_{p,q}$ to the indeterminacy locus $\mathcal{I}(\mupq)$ are not collapsed by the lift of $\mu_{p,q}$ (and are mapped to the union of exceptional lines). Moreover one can check (cf. \cref{rem:lift_v,rem: exceptional divisor,rem:nonss_res}) that the proper transforms of the curves in critical locus of $\mu_{p,q}:\CP^2 \dashrightarrow \CP^2$ and the proper transforms of all intermediate exceptional lines are not collapsed by $\mu_{p,q}:X_{p,q} \dashrightarrow X_{p,q}$ to indeterminate points. Hence by \cref{cor:alg_stab_1,cor:alg_stab_2} the algorithm  has finished and the model is stable.  
\end{proof}

\begin{rem}
    We highlight that for any parameters $(p,q) \in \mathbb{N}^2$ with $pq>4$ although the map $\mu_{p,q}:X_{p,q}\dashrightarrow X_{p,q}$ is algebraically stable, it happens to be non-biregular, as for both semi-simple and non-semi-simple parameters there exist curves that are contracted to fixed points by this map. When $pq > 4$ we have that $\mu_{p,q}$ cannot be made regular by a finite sequence of blowups (cf. \cref{RMK:NO_REG_MODEL}).
    
\end{rem}

\section{Picard Group Method for computing $\lambda_1(\mu_{p,q})$ when $p, q > 0.$}\label{sec:picard}
In this section we compute the map $\mu_{p,q}^*$ for non-affine pairs $(p,q)$\footnote{Namely those for which $pq>4$.} on $\Pic(X_{p,q})$. By \cref{lemma:spec_rad} it is sufficient to compute the spectrum of $\mu_{p,q}^*$ to obtain the dynamical degree $\lambda_1(\mu_{p,q})$.

\subsection{Semi-simple pairs.}
The space $X_{p,q}$ is obtained by the blowup at points $S_p\cup T_q$. This gives a natural projection $\pi_{p,q}:X_{p,q} \to \CP^2$. Then due to Equation \eqref{eqn:pic_blowup} there is a basis in $\Pic(X_{p,q})$ given by $\{C\}\cup\{E_{P}\}_{P\in S_p\cup T_q }$. Here $C$ is the class of the proper transform of a generic line in $\CP^2$ with respect to $\pi_{p,q}$ and for each $P \in S_p\cup T_q$, an exceptional divisor $E_{P}$ is the class of the line $\pi_{p,q}^{-1}(P)$. 

 The map $\mupq^*:\Pic(X_{p,q}) \to \Pic(X_{p,q})$ is given by the following.
\begin{propo}\label{prop:coh_action_ss}
    Assume that the pair $(p,q)$ is semi-simple. Let $P = [\omega:0:1] \in S_p,\;\; Q = [0:\zeta:1]\in T_q$. Recall the points $\iota(P) = [\omega^{-1}:0:1],\iota(Q) = [0:\zeta^{-1}:1]$. Then
    \begin{align}    
            \mupq^*(E_{Q}) &= C - E_{\iota(Q)},\label{eqn:picmuexcz}\\
            \mupq^*(E_{P}) &= q C - \sum_{Q_1\in T_{q}}E_{Q_1} - E_{\iota(P)}.\label{eqn:picmuexcw}\\
            \mupq^*(C) &= p\big(q C - \sum_{Q_1\in T_{q}}E_{Q_1}\big) - \sum_{P_1\in S_{p}}E_{P_1},\label{eqn:picmuline}
    \end{align}
\end{propo}
\begin{proof}
    We've checked in  \cref{subsec:stable_model_ss} that the images of exceptional lines under $\mu_{p,q}$ intersect exceptional lines in a finite number of points. Therefore the set-theoretic preimage of an exceptional line appearing as the blowup of a point $P$ coincides with the proper transform of the curve collapsing to $P$. As well the subset $\{x_0x_1x_2 = 0\}$ is collapsed by $\mu_{p,q}$ to the fixed point $[0:1:0]$, hence the preimage of an exceptional divisor does not contain components of this subset. 

Then in order to compute the pullback of a prime divisor $D$ in the Picard group due to Equation \eqref{eqn:curve_pullback} it is sufficient to
\begin{enumerate}
    \item Compute the preimage $\mu_{p,q}^{-1}(D_0)$ of a representative $D_0$ of the divisor $D$ in an open subset $U = \{x_0x_1x_2 \neq 0\}$ of $X_{p,q}$.
    \item Compute the multiplicity of the local equation of $\mu_{p,q}^{-1}(D_0)$ in the pullback of a local equation of $D_0$ at a generic point of $\mu_{p,q}^{-1}(D_0)$.
    \item Compute the class of the closure $\overline{\mu_{p,q}^{-1}(D_0)}$ in $X_{p,q}$.
\end{enumerate}

We put the detailed computations of these three steps for the basis of $\Pic(X_{p,q})$ to \cref{sec:appx_pullback}.
\end{proof}

In particular, we compute the norm of this operator acting on the complexification $\rV_{p,q} = \C \otimes_\Z \Pic(X_{p,q})$.

\begin{lemma}\label{lemma:spectrum_ss}
    In semi-simple case the operator $\mupq^* \in \mathrm{End}_{\C}(\rV_{p,q})$ is diagonalizable with spectrum 
    \begin{align*}(\lambda_{+},\lambda_{-}, 0, \underbrace{-1,\dots, -1}_{\left \lceil{\frac{p}{2}}\right \rceil + \left \lceil{\frac{q}{2}}\right \rceil -2},\underbrace{1,\dots, 1}_{ \left\lfloor{\frac{p}{2}}\right\rfloor  + \left \lfloor{\frac{q}{2}}\right \rfloor}), \qquad \mbox{where} \qquad
        \lambda_{\pm} = \frac{pq-2\pm\sqrt{(pq-2)^2-4}}{2}.
        \end{align*}
\end{lemma}
\begin{proof}
Is given in \cref{sec:appx_pullback}.
\end{proof}

\subsection{Non-semi-simple pairs.}
The space $X_{p,1}$ is obtained by blowups at points $S_p\cup T_1$ and two consistent blowups at $[0:1:0]$. Due to Equation \eqref{eqn:pic_blowup} there is a basis in $\Pic(X_{p,1})$ given by $\{C\}\cup \{E_{P}\}_{P \in S_p}\cup \{E_Q\}_{Q\in T_1}\cup \{E_{1,\infty}, E_{2,\infty}\}$. Here we denote by $C$ the class of the lift of a generic line in $\CP^2$, for each $P \in S_p\cup T_1$ we denote by $E_{P}$ the class of exceptional line $\pi_{p,1}^{-1}(P)$. We denote by $E_{1,\infty}$ the class of the proper transform of the exceptional line corresponding to the first blowup at point $[0:1:0]$ and by $E_{2,\infty}$ the exceptional line corresponding to the second blowup at $[0:1:0]$.

 The map $\mu_{p,1}^*\in  \mathrm{End}_{\Z}(\Pic(X_{p,1}))$ is given by the following.
\begin{propo}\label{prop:coh_action_nonss}
    Let $p > 4$. Let $P = [\omega:0:1] \in S_p,\;\; Q = [0:-1:1]\in T_1$. Then
    \begin{align}
    \mupo^*(C) &= (p+1) C - pE_{Q} - \sum_{P_1\in S_{p}}E_{P_1} - E_{1,\infty} - 2E_{2,\infty},\label{eqn:picmulinenonss}\\
            \mupo^*(E_{P}) &= C - E_{Q} - E_{\iota(P)},\label{eqn:picmuomnss}\\
           \mupo^*(E_{Q}) &= C - E_{Q},\label{eqn:picmuzenss}\\
           \mupo^*(E_{1,\infty}) &= (p-1) C +(2-p) E_{Q} - \sum_{P_1\in S_{p}}E_{P_1} - E_{1,\infty} - 2E_{2,\infty},\label{eqn:picmuinf1}\\
           \mupo^*(E_{2,\infty}) &= C - E_{Q}\label{eqn:picmuinf2}.
    \end{align}
\end{propo}
\begin{proof}
    Derivation of equations \eqref{eqn:picmuomnss},\eqref{eqn:picmuzenss} repeats the computations for semi-simple pairs (see the proof of \cref{prop:coh_action_ss}). In order to compute the pullback for the classes of exceptional lines corresponding to blowups at $[0:1:0]$ we compute pullbacks of classes of proper transforms of lines with different intersection conditions at $[0:1:0]$. The details are given in \cref{sec:appx_pullback}.
\end{proof}
    The following statement is an analogue of  \cref{lemma:spectrum_ss}. In particular, it explains our choice of naming for semi-simple and non-semi-simple pairs.

    \begin{lemma}\label{lemma:spectrum_nonss}
    In a non-semi-simple case the operator $\mupo^* \in \mathrm{End}_{\C}(\rV_{p,1})$ the spectrum of operator $\mupo^*$ is given by 
    \begin{equation}\label{eqn:specnonss}
    (\lambda_{+},\lambda_{-}, 0,0,0, \underbrace{-1,\dots, -1}_{\left \lceil{\frac{p}{2}}\right \rceil -1},\underbrace{1,\dots, 1}_{ \left\lfloor{\frac{p}{2}}\right\rfloor }).    
    \end{equation}

    In the Jordan normal form of the operator $\mupo^*$ all blocks except for one have size $1$ and there is one block corresponding to the eigenvalue $0$ of size $2$.  
    
    Here
    \begin{equation*}
        \lambda_{\pm} = \frac{p-2\pm\sqrt{(p-2)^2-4}}{2}.
    \end{equation*}
\end{lemma}

\begin{proof}
    Is given in \cref{sec:appx_pullback}.
\end{proof}

 \cref{prop:coh_action_ss,prop:coh_action_nonss} together with \cref{lemma:spec_rad,lemma:spectrum_ss,lemma:spectrum_nonss} imply the following statement.
\begin{theo}\label{Theorem Dynamical degree}
    Let $pq > 4$. Then the dynamical degree of $\mu_{p,q}$ is given by
    \begin{equation}
        \lambda_{+} = \frac{pq-2+\sqrt{(pq-2)^2-4}}{2}.
    \end{equation}
\end{theo}

\begin{rem}
One can give an alternative proof of the above theorem based on the algorithm described in \cite{alonso2023}, and can deduce the following recurrence relation for the degree of the iterates of the mapping (cf. \cref{theo: recurrenceformulasemisimplecase} and \cref{propo:degreenon_semisimple}).
    
    \begin{equation*}
        d_{\rm alg}(\mu_{p,q}^{n+2})+(2-pq)d_{\rm alg}(\mu_{p,q}^{n+1})+d_{\rm alg}(\mu_{p,q}^{n})=0,
    \end{equation*}
    for $pq > 4$.
    This recurrence implies \cref{Theorem Dynamical degree} and the algorithm of \cite{alonso2023} does not use the computation of the pullback operator induced by the map $\mu_{p,q}$ on the Picard group.
\end{rem}

\section{Tropicalization}\label{sec:tropicalization}
We will now present a method developed in \cite{diller_lin_2016} and further discussed in \cite{DR1,DR2} that will play an essential role in our study of $\mu_{p,q}$ when both $p,q < 0$.  It also provides further clarification on what we have already done when $p,q > 0$.   More specifically, 
for a non-semi-simple pair $(p,q)$ there is an additional destabilizing orbit of $\mupq$ (see  \cref{subsec:indeterminacy}) and the methods in this section give us a better understanding of why this orbit exists and why it is not eliminated via a single blow-up but it is eliminated by two blowups.

\begin{defi}
    Denote $\T^k = (\C^{\times})^k$. A toric surface is a complex surface $X$ with an embedding $\T^2 \xhookrightarrow{} X$ as an open dense subset such that the natural action $\T^2\curvearrowright \T^2$ extends holomorphically to an action $\T^2\curvearrowright X$.
\end{defi}

To each toric surface  $X$ there is an associated fan $\Sigma(X) $ which is a partition of $\R^2$ into a union of closed strongly convex two-dimensional cones which intersect by one-dimensional cones generated by integral vectors.  
We denote the sets of two-dimensional cones and their one-dimensional intersections by $\Sigma_2(X)$ and $\Sigma_1(X)$ respectively. We explain the construction of fan on the example of $\CP^2$.

\subsubsection{Fan of $\CP^2$} The natural action  $\T^2 \curvearrowright \CP^2$ is given by
\begin{equation*}
    t \cdot [x] := [t_1x_0: t_2x_1: x_2].
\end{equation*}
where $t= (t_1,t_2)\in \T^2$ and $[x]=[x_0:x_1:x_2] \in \CP^2$.
Observe that, $\T^2 \cong \T^2\cdot[1:1:1] \subset \CP^2$ is an open dense orbit in $\CP^2$. Thus, $\CP^2$ is a toric surface.  
 The sets of cones $\Sigma_1(\CP^2)$ and $\Sigma_2(\CP^2)$ are obtained as follows.  \\

Denote the lines $L_j=\{x_j=0\},\;j= 0,1,2$. For the action $\T^2 \curvearrowright \CP^2$ there is one two-dimensional orbit $\mathcal{O}^{(2)} =\T^2\subset \CP^2$. The one-dimensional orbits are $\mathcal{O}^{(1)}_j = L_j\backslash(L_{j-1}\cup L_{j+1})\cong \T^1$, $j = 0,1,2$. The $0$-dimensional orbits are $\mathcal{O}^{(0)}_j  =L_{j-1}\cap L_{j+1}$. For any orbit $\mathcal{O}$ one corresponds a cone 
\[\sigma(\mathcal{O}) = \left\{(a,b)\in \R^2 | \lim\limits_{z\to0} [z^a\al : z^b\be:1 ]\in \mathcal{O}, \forall (\al,\be) \in \T^2\right\}.\] The dimensions of an orbit and of the corresponding cone sum up to $2$. We denote $\sigma_{j} = \sigma(\mathcal{}\mathcal{O}^{(0)}_j),\; \tau_j = \sigma(\mathcal{O}^{(1)}_j)$ and we have $\{(0,0)\} = \sigma(\mathcal{O}^{(2)})$.

One computes $\tau_j = \R_+u_j$, where $u_0 = (1,0), u_1=(0,1)$ and $u_2=(-1,-1)$ and $\sigma_j\subset \R^2$ is the open and strictly convex cone bounded by the rays $\tau_{j-1}, \tau_{j+1}$, see \cref{fig:1}. \begin{figure}[h]
    \centering
    \begin{tikzpicture}[scale=1.25]
  \draw[line width=1pt,blue,-stealth](0,0)--(1,0) node[anchor=west]{$\boldsymbol{\tau_0}$};
  \draw[line width=1pt,blue,-stealth](0,0)--(0,1) node[anchor=south]{$\boldsymbol{\tau_1}$};
  \draw[line width=1pt,blue,-stealth](0,0)--(-0.707106781,-0.707106781) node[anchor=north east]{$\boldsymbol{\tau_2}$};
 \path node at (0.1,-0.2) [] {$\boldsymbol{0}$}
 node at (0.6,0.6) [] {$\boldsymbol{\sigma_2}$}
     node at (0.2,-0.75) [] {$\boldsymbol{\sigma_1}$}
      node at (-0.75,0.2) [] {$\boldsymbol{\sigma_0}$};
\end{tikzpicture}
\caption{Fan of $\CP^2$}
\label{fig:1}
\end{figure}\\
\subsubsection{Toric maps.}
Recall the affine coordinates $(x,y) = \left(\frac{x_0}{x_2}, \frac{x_1}{x_2}\right)$. Define the meromorphic two-form $\eta = \frac{dx\wedge dy}{xy}$.
\begin{defi}
    A toric map is a rational map $f:\CP^2 \dashrightarrow \CP^2$ such that $f^*\eta = \rho(f)\eta$ for some constant $\rho(f) \in \C^{\times}$.
\end{defi}

\begin{defi}
For a rational toric map $f: \C^2 \dashrightarrow \C^2$ set $(x(t),y(t)) = f(w_1t^a,w_2t^b) $, for some generic $(w_1,w_2)\in\T^2$. Define $$\displaystyle A_{f}(a,b) = \lim\limits_{t\to 0} \left(\left(\frac{\log(|x(t)|)}{\log(|t|)},\frac{\log(|y(t)|)}{\log(|t|)}\right)\right).$$ The map $A_{f}$ is called the tropicalization of the map $f$.
\end{defi}
The map $A_f: \R^2 \to \R^2$ is continuous, piecewise linear and positive homogeneous. If $f$ and $g$ are toric maps then $A_{f\circ g} = A_f \circ A_g$.

In this section we consider the toric surfaces $X$ obtained after finite number of blow-ups  of $\CP^2$ at torus-invariant points. Let $\pi:X\to \CP^2$ denote the blowup map, then $\eta_X:= \pi^*\eta$ is a torus-invariant meromorphic two-form on $X$. The set of poles of $\eta_X \text{ is } X \setminus \T^2$. The torus action preserves the irreducible components of this set and the intersections of these components are the fixed points of the torus action.
The tropicalization helps to trace the behavior of one and zero-dimensional orbits of the torus action under the map $f$.  
\begin{lemma}\label{lemma:Diller_trop}(\cite{diller_lin_2016}, Lemma  8.8) Let $X$ be a smooth compact toric surface and $f: X \dashrightarrow X$ be a toric map. 
    Let $\tau, \tau'$ be the rays $\sigma, \sigma'$ be two-dimensional cones in $\Sigma(X)$. Let $C, C'$ and $P,P'$ be the associated irreducible components and fixed points of the torus action in $X\setminus\T^2$ respectively.
    \begin{enumerate}
        \item $A_f(\tau) = \tau'$ if and only if $f(C) = C'$.
        \item $A_f(\tau) \subset \sigma$ if and only if $f(C) = P$.
        \item $A_f(\sigma) \subset \sigma'$ if and only if $f(P) \setminus \T^2 = P'$.
        \item $\tau \subset A_f(\sigma)$ if and only if $C\subset f(P)$. In particular, $P\in \mathcal{I}(f)$
    \end{enumerate}
\end{lemma}

\begin{rem}\label{REM:FINDING_FIXEDPTS}
Typically, Part (3) of 
Lemma \ref{lemma:Diller_trop} corresponds to $P$ being a fixed point for $f$.  However, there is a subtle issue because 
of the possibility that  $P \in \mathcal{I}(f)$.
However, this can only happen if 
there is exceptional curve $E$ for $f$ with $E \cap \mathbb{T}^2 \neq \emptyset$ and with $P \in E$; see \cite[Lemma 4.7]{DR1}.
\end{rem}

Below, we check that the maps $\mu_{p,q}, \mu_{p,q}^{(1)},\mu_{p,q}^{(2)}$ are toric and compute the tropicalization map for the last two of them.
\begin{equation*}
    \left(\mu_{p,q}^{(1)}\right)^*\left(\frac{dx\wedge dy}{xy}\right) = \frac{d\left(\frac{1+y^q}{x}\right)\wedge dy}{\left(\frac{1+y^q}{x}\right)y}=-\frac{dx\wedge dy}{xy} \text{ , }\;\;\; \left(\mu_{p,q}^{(2)}\right)^*\left(\frac{dx\wedge dy}{xy}\right) = \frac{dx\wedge d\left(\frac{1+x^p}{y}\right)}{\left(\frac{1+x^p}{y}\right)x}=-\frac{dx\wedge dy}{xy}.
\end{equation*}
Hence,
\begin{equation}\label{eqn:2form_pres}
    \mu_{p,q}^*\left(\frac{dx\wedge dy}{xy}\right) = \left(\frac{dx\wedge dy}{xy}\right).
\end{equation}\\
We have $\displaystyle \mu_{p,q}^{(1)}(t^a, t^b) = ( t^{-a}(1+t^{qb}), t^b),\; \mu_{p,q}^{(2)}(t^a, t^b) = (t^a, t^{-b}(1+t^{pa}))$. Then the tropicalizations $A_1, A_2$ of the maps $\mu_{p,q}^{(1)},  \mu_{p,q}^{(2)}$ respectively are given by 
\begin{align}
    A_1\begin{bmatrix}
        a\\
        b
    \end{bmatrix} = 
        \begin{bmatrix}
            \begin{cases}
                -a &\ \text{if}\ \ b\geq 0\\
                -a+qb&\ \text{if}\ \ b\leq 0\\
            \end{cases}\\
            b
        \end{bmatrix},\;\;\;\;
        A_2\begin{bmatrix}
        a\\
        b
    \end{bmatrix} = \begin{bmatrix}
            a\\
            \begin{cases}
                -b &\ \text{if}\ \ a\geq 0\\
                -b+ap&\ \text{if}\ \ a\leq 0\\
            \end{cases}            
        \end{bmatrix}.
\end{align}

Since $\mupq = \mu_{p,q}^{(2)}\circ \mu_{p,q}^{(1)}$ we obtain the tropicalization
$A$ of $\mupq$ to be $A=A_2 \circ A_1$ by functoriality.   It is a homeomorphism as well.
Applying the tropicalization map $A$ to $\Sigma_1(\mathbb{CP}^2)$ we get,

\begin{equation*}
    Au_0 = \begin{bmatrix}
        -1\\
        -p
    \end{bmatrix},\ Au_1 = \begin{bmatrix}
        0\\
        -1
    \end{bmatrix},\ Au_2 = \begin{bmatrix}
        1-q\\
        1+(1-q)p
    \end{bmatrix} .
\end{equation*}

\subsubsection{Semi-simple pairs.} Given $pq-p-q>0$, the action of $A$ to $\Sigma(\CP^2)$ is as shown in $\cref{fig:2}$.

\begin{figure}[h!]
   \centering
    \begin{tikzpicture}[yscale=0.75]
  \draw[line width=1pt,blue,-stealth](0,0)--(1.414213562,0) node[anchor=west]{$\boldsymbol{\tau_0}$};
  \draw[line width=1pt,blue,-stealth](0,0)--(0,1.414213562) node[anchor=south]{$\boldsymbol{\tau_1}$};
  \draw[line width=1pt,blue,-stealth](0,0)--(-1,-1) node[anchor=north east]{$\boldsymbol{\tau_2}$};
  \draw[line width=0.5pt,red,-stealth](0,0)--(0,-1.414213562) node[anchor=north west]
  {$\boldsymbol{A\tau_1}$};
  \draw[line width=0.5pt,red,-stealth](0,0)--(-1,-3) node[anchor=north]
  {$\boldsymbol{A\tau_0}$};
  \draw[line width=0.5pt,red,-stealth](0,0)--(-1,-2) node[anchor=north east]
  {$\boldsymbol{A\tau_2}$};
  \path  node at (0.6,0.6) [] {$\boldsymbol{\sigma_2}$}
     node at (0.4,-0.75) [] {$\boldsymbol{\sigma_1}$}
      node at (-0.75,0.2) [] {$\boldsymbol{\sigma_0}$};
\end{tikzpicture}
\caption{Action of the tropicalization $A$ of $\mu_{p,q}$ on the fan $\Sigma(\mathbb{CP}^2)$ in the semi-simple case.} \label{fig:2}
\end{figure}
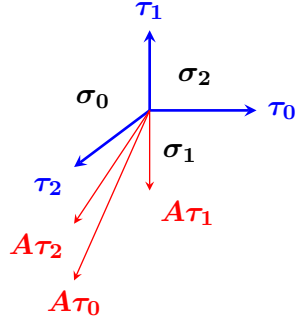

Using  $\cref{lemma:Diller_trop}$, we get the following observations (cf.  $\cref{subsec:indeterminacy}$).
\begin{itemize}
    \item Since $A(\sigma_1) \subset \sigma_1$  we expect that $[0:1:0]$ is a fixed point; see Remark \ref{REM:FINDING_FIXEDPTS}.  In fact it is a fixed point because there are no exceptional
    curves $E$ with $E \cap \mathbb{T}^2 \neq \emptyset$ passing through $P$, as
    shown in Figure \ref{fig:crit_mu23}.  
    \item $\tau_j \subset A(\sigma_0)$ for all $j=0,1,2$, implies $[1:0:0] \in \mathcal{I}(\mupq)$, 
    \item $A\tau_j \subset \sigma_1$ implies $\mupq(\{x_j=0\}) =[0:1:0]$ for all $j=0,1,2$. 
\end{itemize}

\subsubsection{Non-semi-simple pairs.} Given $p>4,\ q=1$, the action of $A$ to $\Sigma(\CP^2)$ is as shown in $\cref{fig:subfig8}$.
\begin{figure}[h!]
 \centering
    \begin{subfigure}[b]{0.2\textwidth}
           \begin{tikzpicture}[yscale=0.75]
  
  \draw[line width=0.5pt,red,-stealth](0,0)--(0,-2) node[anchor= south west]
  {$\boldsymbol{A\tau_1}$};
  \draw[line width=0.5pt,red,-stealth](0,0)--(-0.5546,-2.773) node[anchor= west ]
  {$\boldsymbol{A\tau_0}$};
  \draw[line width=0.5pt,red,-stealth](0,0)--(0,2) node[anchor= east]
  {$\boldsymbol{A\tau_2}$};
  \draw[line width=1pt,blue,-stealth](0,0)--(1,0) node[anchor=west]{$\boldsymbol{\tau_0}$};
  \draw[line width=1pt,blue,-stealth](0,0)--(0,1) node[anchor=east]{$\boldsymbol{\tau_1}$};
  \draw[line width=1pt,blue,-stealth](0,0)--(-1,-1) node[anchor=north east]{$\boldsymbol{\tau_2}$};
  \path  node at (0.6,0.6) [] {$\boldsymbol{\sigma_2}$}
     node at (0.4,-0.75) [] {$\boldsymbol{\sigma_1}$}
      node at (-0.75,0.2) [] {$\boldsymbol{\sigma_0}$};
      \node at (0,-3.984375) [] {$ $};
\end{tikzpicture}
        \caption{Fan of $\CP^2$}
        \label{fig:subfig8}
    \end{subfigure}
    \hfill
    \begin{subfigure}[b]{0.2\textwidth}
            \begin{tikzpicture}[yscale=0.75]
\draw[line width=0.5pt,red,-stealth](0,0)--(0,-2) node[anchor=west]
  {$\boldsymbol{A\tau_1}$};
  \draw[line width=0.5pt,red,-stealth](0,0)--(-0.5546,-2.773) node[anchor=west ]
  {$\boldsymbol{A\tau_0}$};
  \draw[line width=0.5pt,red,-stealth](0,0)--(0,2) node[anchor= east]
  {$\boldsymbol{A\tau_2}$};
  \draw[line width=1pt,blue,-stealth](0,0)--(1,0) node[anchor=west]{$\boldsymbol{\tau_0}$};
  \draw[line width=1pt,blue,-stealth](0,0)--(0,1) node[anchor=east]{$\boldsymbol{\tau_1}$};
  \draw[line width=1pt,blue,-stealth](0,0)--(-1,-1) node[anchor=north east]{$\boldsymbol{\tau_2}$};

  \draw[line width=1pt,blue,-stealth](0,0)--(0,-1) node[anchor=west ]{$\boldsymbol{\tau_3}$};
  \draw[line width=0.5pt,red,-stealth](0,0)--(-1,-4) node[anchor=west]{$\boldsymbol{A\tau_3}$};
  \path  node at (0.6,0.6) [] {$\boldsymbol{\sigma_2}$}
     node at (0.75,-0.75) [] {$\boldsymbol{\sigma_3}$}
     node at (-0.75, -1.75) [] {$\boldsymbol{\sigma_4}$}
      node at (-0.75,0.2) [] {$\boldsymbol{\sigma_0}$};
\end{tikzpicture}
        
        \caption{Blow-up at $\sigma_1$}  
        \label{fig:subfig9}
    \end{subfigure}
    \hfill
    \begin{subfigure}[b]{0.2\textwidth}
        
            \begin{tikzpicture}[yscale=0.75]
            \draw[line width=0.5pt,red,-stealth](0,0)--(0,-2) node[anchor=west]
  {$\boldsymbol{A\tau_1}$};
  \draw[line width=0.5pt,red,-stealth](0,0)--(-0.5546,-2.773) node[anchor=west ]
  {$\boldsymbol{A\tau_0}$};
  \draw[line width=0.5pt,red,-stealth](0,0)--(0,2) node[anchor= east]
  {$\boldsymbol{A\tau_2}$};
  \draw[line width=1pt,blue,-stealth](0,0)--(1,0) node[anchor=west]{$\boldsymbol{\tau_0}$};
  \draw[line width=1pt,blue,-stealth](0,0)--(0,1) node[anchor=east]{$\boldsymbol{\tau_1}$};
  \draw[line width=1pt,blue,-stealth](0,0)--(-1,-1) node[anchor=south east]{$\boldsymbol{\tau_2}$};
  \draw[line width=1pt,blue,-stealth](0,0)--(0,-1) node[anchor=west ]{$\boldsymbol{\tau_3}$};
  \draw[line width=0.5pt,red,-stealth](0,0)--(-1,-4) node[anchor=west]{$\boldsymbol{A\tau_3}$};
 \draw[line width=1pt,blue,-stealth](0,0)--(-1,-2) node[anchor=north east ]{$\boldsymbol{\tau_4}$};
   \draw[line width=0.5pt,red,-stealth](0,0)--(-1,-3) node[anchor=north east]{$\boldsymbol{A\tau_4}$};
  \path  node at (0.6,0.6) [] {$\boldsymbol{\sigma_2}$}
     node at (0.75,-0.75) [] {$\boldsymbol{\sigma_3}$}
     node at (-1.1, -1.5) [] {$\boldsymbol{\sigma_6}$}
     node at (-0.4 ,-3.35) [] {$\boldsymbol{\sigma_5}$}
      node at (-0.75,0.2) [] {$\boldsymbol{\sigma_0}$};
      
\end{tikzpicture}
        
        \caption{Blow-up at $\sigma_4$}
        \label{fig:subfig10}
    \end{subfigure}
\caption{ } 
\label{fig:subfig1.a.4}
\end{figure}
Using $\cref{lemma:Diller_trop}$, we get the following observations (cf.  $\cref{subsec:indeterminacy}$).
\begin{itemize}
    \item $\tau_2 \subset A(\sigma_1)$ implies $[0:1:0]\in\mathcal{I}(\mupq)$,
    \item $A\tau_j \subset \sigma_1$  implies $\mupq(\{x_j=0\})=[0:1:0]$ for $j=0,1$. 
\end{itemize}

This implies that there is a destabilizing orbit terminating at $[0:1:0]$. The fan of the blowup $\tilde{X}$ of $\CP^2$ at $[0:1:0]$ and the action of the tropicalization map on it, are shown in $\cref{fig:subfig9}$. The rational ray corresponding to the exceptional line $E_{1,\infty}$ is $\tau_3 =\R_+(u_3)$ where $u_3=u_0+u_2$. There are two torus-invariant points on the exceptional line $E_{1,\infty}$. We denote the corresponding cones by $\sigma_3, \sigma_4$.
\begin{itemize}
    \item $\tau_2\subset A(\sigma_4)$ implies that the point corresponding to $\sigma_4$ is indeterminate.
    \item $A\tau_3 \subset \sigma_4$ implies that the exceptional line $E_{1,\infty}$ collapses to the point corresponding to $\sigma_4$.
    \item $A\tau_j \not\subset \sigma_3$ implies that there is no one-dimensional torus orbit collapsed to $\sigma_3$. 
\end{itemize}
This implies the existence of a destabilizing orbit. The fan of blowup $\tilde{\tilde{X}}$ of $\tilde{X}$  at the point corresponding to $\sigma_4$ and the action of the tropicalization map on it, are shown in $\cref{fig:subfig10}$. The rational ray corresponding to the exceptional line $E_{2,\infty}$ is $\tau_4 = \R_+(u_4)$ where $u_4 = u_2 +u_3$. There are two torus-invariant points on the exceptional line $E_{2,\infty}$. We denote the corresponding cones by $\sigma_5, \sigma_6$.

\begin{itemize}
    \item Since $A(\sigma_5) \subset \sigma_5$, can can
    follow the method described in Remark \ref{REM:FINDING_FIXEDPTS} to verify that the torus invariant point corresponding to $\sigma_5$ is a fixed point for $\mu_{p,q}$.
    \item $A\tau_j \not\subset \sigma_k$ for all $j$ and $k\neq 5$ implies that there is no one-dimensional torus orbit collapsed to $\sigma_6$. 
\end{itemize}
Thus, if the curves corresponding to one-dimensional orbits of torus action on $\tilde{\tilde{X}}$ are collapsing, then they are only collapsing to point corresponding to $\sigma_5$. The only curves that are collapsing to $[0:1:0]$ under the map $\mupq$ are $\{x_0=0\}$ and $\{x_1=0\}$ (cf.  $\cref{subsec:indeterminacy}$). Hence, after two blow-ups at $[0:1:0]$, the minimal destabilization orbit that arose from this point is eliminated.  This is the way in which the tropicalization shows why
we needed two blow-ups over $[0:1:0]$ when forming the algebraically stable model $X_{p,q}$ in the non-semi-simple case.

\subsubsection{Eigenvalues of the tropicalization}
One can check that the fixed rays of $A$ are in the third quadrant and are generated by the eigenvectors 
\begin{equation*}
    e_+=\begin{bmatrix}
        \frac{-pq + \sqrt{pq(pq-4)}}{2} \\
        -p
    \end{bmatrix}\text{ and } e_-= \begin{bmatrix}
        \frac{-pq - \sqrt{pq(pq-4)}}{2} \\
        -p
    \end{bmatrix}.
\end{equation*}

The respective eigenvalues are
\begin{equation*}
    \lambda_+= \frac{pq-2 + \sqrt{pq(pq-4)}}{2}\ \text{ and }\ \lambda_-= \frac{pq-2 - \sqrt{pq(pq-4)}}{2}.
\end{equation*}

 Observe that the largest eigenvalue coincides with the dynamical degree of the map (See  $\cref{Theorem Dynamical degree}$).  
The map $\mupq$ behaves like a monomial map in the sense that the dynamical degree is equal to the largest eigenvalue $\lambda_+$ of the tropicalization of the map $\mupq$ (See Theorem 9.1 in \cite{DR1}).
    
\begin{rem}
\label{RMK:NO_REG_MODEL}
Note that the ray $\R_{+}(e_-)$ is repelling and has irrational slope. This implies that after any number of blowups at torus-invariant points corresponding to the sector $\boldsymbol{\sigma_0}$, there is always another indeterminacy point. This supports the idea that the map $\mu_{p,q}$ cannot be made regular by a finite number of blowups.
\end{rem}

\section{Study of $\mu_{p,q}$ when $p,q \leq -1$.}\label{sec:negative_pairs}
\subsection{Critical behavior for negative pairs.} Now we consider a pair $(p,q)\in\{(p,q)\in\Z^2_{<0}:{p}{q}>4\}$; these pairs will be called \textit{negative pairs}. In this section we use the notation $\tilde{p} = -p, \tilde{q} = -q$. We study this case separately as we can show that after blowing up $\CP^2$ four times, the map $\mu_{-\tilde{p},-\tilde{q}}$ (for most choices of $p,q$) has a superattracting 3-cycle. This phenomenon is analogous to the existence of superattracting fixed point in the case $p,q\geq 1$ with $pq>4$ (cf. \cref{remark:superattracting_fixed_point}). Some aspects are simpler for the negative pairs, for example, in case $\tilde{p},\tilde{q}\geq 2$ we only need to blow up torus-invariant points to reach algebraic stability. In this section we assume $\tilde{p}\tilde{q} = pq>4$ and $p,q\in\Z_{<0}$. (For the case $p=q=-2$, see Remark \ref{Rem: stability p=2, q=2}.) Denote $\nu_{\tilde{p},\tilde{q}}:=\mu_{-\tilde{p},-\tilde{q}}$. By Equation \eqref{eq: c^2 formula for mutation} we have 
\begin{equation*}
    \nu_{\tilde{p},\tilde{q}}:= (x,y) \mapsto \left(\frac{1+y^{-\tilde{q}}}{x},\frac{x^{-\tilde{p}}+(1+y^{-\tilde{q}})^{-\tilde{p}}}{x^{-\tilde{p}}y}\right)=\left(\frac{y^{\tilde{q}}+1}{xy^{\tilde{q}}},\frac{x^{\tilde{p}}y^{\tilde{p}\tilde{q}}+(1+y^{\tilde{q}})^{\tilde{p}}}{y(y^{\tilde{q}}+1)^{\tilde{p}}}\right).
\end{equation*}
We extend the birational dynamics $\nu_{\tilde{p},{\tilde{q}}}$ to $\CP^2$
\begin{equation}
    \nu_{\tilde{p},{\tilde{q}}}([x_0\colon x_1\colon x_2])=\left[x_2^{\tilde{p}}(x_1^{\tilde{q}}+x_2^{\tilde{q}})^{\tilde{p}+1}\colon x_0x_1^{{\tilde{q}}-1}\left(x_0^{\tilde{p}}x_1^{\tilde{p}{\tilde{q}}}+x_2^{\tilde{p}}(x_1^{\tilde{q}}+x_2^{\tilde{q}})^{\tilde{p}}\right)\colon x_0x_1^{{\tilde{q}}}x_2^{\tilde{p}-1}(x_1^{\tilde{q}}+x_2^{\tilde{q}})^{\tilde{p}}\right]
\end{equation}
The following two lemmas are similar to Lemma \ref{Lemma: Indeterminate set} and Lemma \ref{Lemma1: Critical set} correspondingly. We use the same notation as in Section \ref{subsec:indeterminacy}. We omit the proofs as they are done by direct calculations that are very similar to those in \cref{subsec:indeterminacy}.
\begin{lemma}\label{lemma: indeterminacy negative pairs}
    Let $\tilde{p},\tilde{q} \in \Z_{>0}$.
    Then the indeterminacy locus of $\nu_{{\tilde{p}},{\tilde{q}}}$ is given by 
    $$\mathcal{I}(\nu_{{\tilde{p}},{\tilde{q}}})=\{[0:1:0],[1:0:0]\} \cup T_{\tilde{q}}.$$    
\end{lemma}

\begin{lemma}\label{Lemma: Critical set negative pair}
    Let ${\tilde{p}},{\tilde{q}}\geq 2$, then the critical set of $\nu_{{\tilde{p}},{\tilde{q}}}$ is given by 
    \begin{equation*}
         \displaystyle\mathcal{C}(\nu_{{\tilde{p}},{\tilde{q}}})=\bigcup_{i=0}^2\{x_i=0\}\cup\bigcup_{\zeta^{\tilde{q}}+1=0} K^\textit{v}_\zeta\cup\bigcup_{\omega^{\tilde{p}}+1=0}H^\textit{h}_\omega.
    \end{equation*}
 Moreover, 
 \begin{equation*}
     \nu_{{\tilde{p}},{\tilde{q}}}^{-1}([1\colon 0\colon 0])= \bigcup_{i=0}^1\{x_i=0\}, \;\;\; \nu_{{\tilde{p}},{\tilde{q}}}^{-1}([0\colon 1\colon 0])=\{x_2=0\}\cup\bigcup_{\zeta^{\tilde{q}}+1=0}K^\textit{v}_\zeta, \;\;\;\nu_{{\tilde{p}},{\tilde{q}}}^{-1}([\omega:0:1])=H_{\omega}^h.
 \end{equation*}
 where $H^\textit{h}_\omega:=\{x_0x_1^{\tilde{q}}-\omega^{-1}x_2(x_1^{\tilde{q}}+x_2^{\tilde{q}})=0\},\;\;  K^\textit{v}_\zeta:=\{x_2-\zeta x_1=0\}.$
    \end{lemma}
    \begin{rem}\label{Rem: omega for negative pairs}
        Notice that unlike the case $p>0,q>0$, the points $[\omega:0:1]\in S_{{\tilde{p}}}$ are no longer indeterminate. Nevertheless, they are in destabilizing orbits of the form $H^h_{\omega}\mapsto[\omega:0:1]\mapsto[1:0:0]\in\mathcal{I}(\nu_{{\tilde{p}},{\tilde{q}}})$.
    \end{rem}
\subsubsection{Tropicalization for negative pairs.}\label{SEC:TROP_NEG_PQ}
Like for $\mu_{p,q}$ the map $\nu_{{\tilde{p}},{\tilde{q}}}$ can be written as composition of two toric maps as follows
\begin{equation*}
    \nu_{{\tilde{p}},{\tilde{q}}}(x,y)=\nu_{{\tilde{p}},{\tilde{q}}}^{(2)}(x,y)\circ\nu_{{\tilde{p}},{\tilde{q}}}^{(1)}(x,y)=\left(x,\frac{x^{\tilde{p}}+1}{x^{\tilde{p}}y}\right)\circ\left(\frac{y^{\tilde{q}}+1}{xy^{\tilde{q}}},y\right).
\end{equation*}
 Tropicalization of the maps $\nu_{{\tilde{p}},{\tilde{q}}}^{(1)},\nu_{{\tilde{p}},{\tilde{q}}}^{(2)}$ is given by
 \begin{align}
    A_{\nu_{{\tilde{p}},{\tilde{q}}}^{(1)}}\begin{bmatrix}
        a\\
        b
    \end{bmatrix} = 
        \begin{bmatrix}
            \begin{cases}
                -a -b{\tilde{q}}&\ \text{if}\ \ b\geq 0\\
                -a&\ \text{if}\ \ b\leq 0\\
            \end{cases}\\
            b
        \end{bmatrix},\;\;\;\;
        A_{\nu_{{\tilde{p}},{\tilde{q}}}^{(2)}}\begin{bmatrix}
        a\\
        b
    \end{bmatrix} = \begin{bmatrix}
            a\\
            \begin{cases}
                -b-a{\tilde{p}} &\ \text{if}\ \ a\geq 0\\
                -b&\ \text{if}\ \ a\leq 0\\
            \end{cases}            
        \end{bmatrix}.
\end{align}
Since tropicalization is functorial, the tropicalization map of $\nu_{{\tilde{p}},{\tilde{q}}}$ is given by 
\[
A = A_{\nu_{{\tilde{p}},{\tilde{q}}}^{(2)}}\circ A_{\nu_{{\tilde{p}},{\tilde{q}}}^{(1)}}.
\]

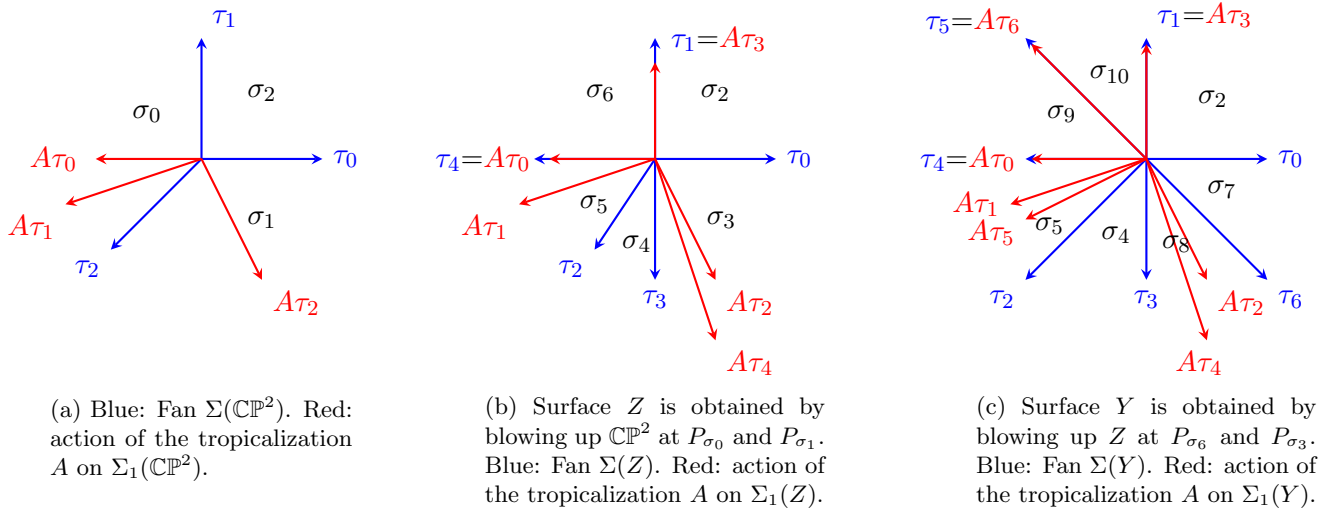
\begin{figure}[h!]
\centering
\begin{tikzpicture}[>=stealth, thick]
  \begin{scope}[xshift=0cm, scale=0.4]
    \draw[blue, ->] (0,0) -- (4,0)   node[right]       {$\tau_0$};
    \draw[blue, ->] (0,0) -- (0,4)   node[above right] {$\tau_1$};
    \draw[blue, ->] (0,0) -- (-3,-3) node[below left]  {$\tau_2$};
 
    \draw[red, ->] (0,0) -- (-3.5,0)    node[left, xshift=-3pt] {$A\tau_0$};
    \draw[red, ->] (0,0) -- (-4.5,-1.5) node[below left]        {$A\tau_1$};
    \draw[red, ->] (0,0) -- (2,-4)      node[below right]       {$A\tau_2$};
 
    \node at (-1.8,  1.5) {$\sigma_0$};
    \node at ( 2.0,  2.2) {$\sigma_2$};
    \node at ( 2,   -2.0) {$\sigma_1$};
  \end{scope}
 
  \begin{scope}[xshift=6cm, scale=0.4]
    \draw[blue, ->] (0,0) -- (4,0)   node[right]      {$\tau_0$};
    \draw[blue, ->] (0,0) -- (0,4);
    \draw[blue, ->] (0,0) -- (0,-4)  node[below]      {$\tau_3$};
    \draw[blue, ->] (0,0) -- (-2,-3) node[below left] {$\tau_2$};
    \draw[blue, ->] (0,0) -- (-4,0);
 
    \draw[red, ->] (0,0) -- (0,3.2)
      node[above right, xshift=2pt]
        {$\textcolor{blue}{\tau_1}\textcolor{black}{=}A\tau_3$};
    \draw[red, ->] (0,0) -- (-3.5,0)
      node[left, xshift=-3pt]
        {$\textcolor{blue}{\tau_4}\textcolor{black}{=}A\tau_0$};
    \draw[red, ->] (0,0) -- (-4.5,-1.5) node[below left]  {$A\tau_1$};
    \draw[red, ->] (0,0) -- (2,-4)      node[below right] {$A\tau_2$};
    \draw[red, ->] (0,0) -- (2,-6)      node[below right] {$A\tau_4$};
 
    \node at (-1.8,  2.2) {$\sigma_6$};
    \node at ( 2.0,  2.2) {$\sigma_2$};
    \node at ( 2.2, -2.0) {$\sigma_3$};
    \node at (-0.6, -2.8) {$\sigma_4$};
    \node at (-2,   -1.5) {$\sigma_5$};
   \end{scope}

  \begin{scope}[xshift=12.5cm, scale=0.4]
    \draw[blue, ->] (0,0) -- (4,0)   node[right]       {$\tau_0$};
    \draw[blue, ->] (0,0) -- (-4,0);
    \draw[blue, ->] (0,0) -- (0,4)
      node[above right] {$\tau_1\textcolor{black}{=}\textcolor{red}{A\tau_3}$};
    \draw[blue, ->] (0,0) -- (0,-4)  node[below]       {$\tau_3$};
    \draw[blue, ->] (0,0) -- (-4,-4) node[below left]  {$\tau_2$};
    \draw[blue, ->] (0,0) -- (4,-4)  node[below right] {$\tau_6$};
    \draw[blue, ->] (0,0) -- (-4,4);
 
    \draw[red, ->] (0,0) -- (0,3.8);
    \draw[red, ->] (0,0) -- (-3.8,0)
      node[left, xshift=-2pt]
        {$\textcolor{blue}{\tau_4}\textcolor{black}{=}A\tau_0$};
    \draw[red, ->] (0,0) -- (-4.5,-1.5) node[left]              {$A\tau_1$};
    \draw[red, ->] (0,0) -- (-4,-2)     node[left, yshift=-5pt] {$A\tau_5$};
    \draw[red, ->] (0,0) -- (-3.8,3.8)
      node[above left]
        {$\textcolor{blue}{\tau_5}\textcolor{black}{=}A\tau_6$};
    \draw[red, ->] (0,0) -- (2,-4)  node[below right]            {$A\tau_2$};
    \draw[red, ->] (0,0) -- (2,-6)  node[below left, xshift=10pt]{$A\tau_4$};
 
    \node at ( 2.2,  2.0) {$\sigma_2$};
    \node at (-1.2,  2.8) {$\sigma_{10}$};
    \node at (-2.8,  1.5) {$\sigma_9$};
    \node at (-3.2, -2.2) {$\sigma_5$};
    \node at (-1.0, -2.5) {$\sigma_4$};
    \node at ( 1,   -2.8) {$\sigma_8$};
    \node at ( 2.5, -1.0) {$\sigma_7$};
  \end{scope}

  \node[below, text width=4cm, align=justify, font=\footnotesize]
    at (0, -3.0)
    {(a) Blue: Fan $\Sigma(\mathbb{CP}^2)$. Red: action of the
     tropicalization $A$ on $\Sigma_1(\mathbb{CP}^2)$.};
 
  \node[below, text width=4.5cm, align=justify, font=\footnotesize]
    at (6cm, -3.0)
    {(b) Surface $Z$ is obtained by blowing up $\mathbb{CP}^2$ at
     $P_{\sigma_0}$ and $P_{\sigma_1}$. Blue: Fan $\Sigma(Z)$.
     Red: action of the tropicalization $A$ on $\Sigma_1(Z)$.};
 
  \node[below, text width=4.5cm, align=justify, font=\footnotesize]
    at (12.5cm, -3.0)
    {(c) Surface $Y$ is obtained by blowing up $Z$ at
     $P_{\sigma_6}$ and $P_{\sigma_3}$. Blue: Fan $\Sigma(Y)$.
     Red: action of the tropicalization $A$ on $\Sigma_1(Y)$.};
 
\end{tikzpicture}
\caption{Action of the tropicalization $A$ of $\nu_{{\tilde{p}},{\tilde{q}}}$ when ${\tilde{p}}, {\tilde{q}} \geq 1$ and
${\tilde{p}}{\tilde{q}} > 4$.}
\label{fig:neg_pq_fan}
\end{figure}

We construct surfaces $Y, Z$  via the following sequence of blowups. 
Let $Z$ be the blowup of $\CP^2$ at $P_{\sigma_{0}}, P_{\sigma_1}$ (see Figure \hyperref[fig:neg_pq_fan]{7a}). Denote the exceptional lines over these points by $\tilde{C}_{\tau_4}$ and $\tilde{C}_{\tau_3}$ respectively. 
Let $Y$ be the blowup of $Z$ at the points $P_{\sigma_{6}}, P_{\sigma_3}$ (see Figure \hyperref[fig:neg_pq_fan]{7b}). Denote the exceptional lines over these points by $\hat{C}_{\tau_5}$ and $\hat{C}_{\tau_6}$ respectively.  If $C$ is a curve on $\CP^2$ we call its lift to $Z$ by $\tilde{C}$ and its lift to $Y$ by $\hat{C}$. Denote the lift of $\tilde{C}_{\tau_{3}}, \tilde{C}_{\tau_4}$ to $Y$ by $\hat{C}_{\tau_{3}}, \hat{C}_{\tau_4}$ respectively.

The following result shows that we obtain algebraic stability for the lift of $\nu_{{\tilde{p}},{\tilde{q}}}$ to $Y$.

\begin{propo}\label{Propo: Negative pairs} 
If ${\tilde{p}},{\tilde{q}}\geq 2$ with ${\tilde{p}}{\tilde{q}}>4$, then the  lift $\nu_{{\tilde{p}},{\tilde{q}}}: Y \dashrightarrow Y$ is algebraically stable. 
\end{propo}
\begin{proof}
We will first use the tropicalization $A$ to guide us in doing the four blow-ups over torus
invariant points indicated
in the construction of $Y$ before the statement of the proposition. This will
eliminate any destabilizing orbit resulting from collapse of a line $C_{\tau_i},\; i = 0,1,2$ by $\nu_{{\tilde{p}},{\tilde{q}}}$.  Even though there were destabilizing
orbits of $\nu_{{\tilde{p}},{\tilde{q}}}$ in $\mathbb{CP}^2$ that corresponded to the collapsing curves $H^h_{\omega}$ and $K^v_\zeta$, it fortuitously turns out that they were eliminated after doing the four blow-ups mentioned
in the previous sentence.

When using Part (3) from Lemma \ref{lemma:Diller_trop} in the following computations, we must always verify that the point $P$ corresponding to the
sector $\sigma$ is not indeterminate.   In each case, we either verify it with explicit calculations in local coordinates or by checking that there is
no exceptional curve $E$ passing through $P$ with $E \cap \mathbb{T}^2 \neq \emptyset$;  See Remark \ref{REM:FINDING_FIXEDPTS}.

 \textbf{Step 1:} By Lemmas \ref{lemma: indeterminacy negative pairs} and \ref{Lemma: Critical set negative pair} the indeterminate points $P_{\sigma_0}$ and $P_{\sigma_1}$ of $\nu_{{\tilde{p}},{\tilde{q}}}: \mathbb{CP}^2 \dashrightarrow \mathbb{CP}^2$ induce two minimal destabilizing orbits. Following the algorithm of \cref{{theo: Diller-Stabilization algorithm}} we need to blow-up at $P_{\sigma_0}$ and $P_{\sigma_1}$ and check if the lift of the map $\nu_{{\tilde{p}},{\tilde{q}}}$ is algebraically stable.

 \cref{lemma:Diller_trop} and the action of the tropicalization $A$ on $\Sigma(Z)$, as shown in \hyperref[fig:neg_pq_fan]{7b}, give us the following information about the action of the lift of $\nu_{{\tilde{p}},{\tilde{q}}}$ on the blown-up space $Z$:

 \begin{itemize}
    \item $P_{\sigma_6},P_{\sigma_4}\in\mathcal{I}(\nu_{{\tilde{p}},{\tilde{q}}})$, due to $\tau_2\subset A(\sigma_6),\; \tau_0\subset A(\sigma_4).$ 
     \item $\nu_{{\tilde{p}},{\tilde{q}}}(P_{\sigma_3})=P_{\sigma_6}$ due to $A(\sigma_3)\subset \sigma_6$.
     \item $\nu_{{\tilde{p}},{\tilde{q}}}(\tilde C_{\tau_4})=\nu_{{\tilde{p}},{\tilde{q}}}\left(\tilde{C}_{\tau_2}\right) =P_{\sigma_3},\;\; \nu_{{\tilde{p}},{\tilde{q}}}\left(\tilde{C}_{\tau_1}\right)=P_{\sigma_5}$ due to $A(\tau_4)\cup A(\tau_2) \subset \sigma_3,\;\;A(\tau_1)\subset \sigma_5$.
     \item $\nu_{{\tilde{p}},{\tilde{q}}}\left(\tilde C_{\tau_3}\right)=\tilde{C}_{\tau_1},\;\nu_{{\tilde{p}},{\tilde{q}}}\left(\tilde{C}_{\tau_0}\right)=\tilde C_{ \tau_4}$,  due to $A(\tau_3)=\tau_1,\;\;A(\tau_0)=\tau_4$.
 \end{itemize}
 Hence, we have a minimal destabilizing orbit $\tilde C_{\tau_4}\mapsto P_{\sigma_3}\mapsto P_{\sigma_6}\in\mathcal{I}(\nu_{{\tilde{p}},{\tilde{q}}})$ and thus we need to blow-up $Z$ at $P_{\sigma_3}$ and $P_{\sigma_6}$ in order to obtain the surface $Y$.
 
 \textbf{Step 2:} 
Analogously to Step $1$ from the action of the tropicalization $A$ on $\Sigma(Y)$ we obtain (cf. Figure \hyperref[fig:neg_pq_fan]{7c})
 \begin{itemize}
     \item We have a 3-cycle $P_{\sigma_{10}}\mapsto P_{\sigma_{5}}\mapsto P_{\sigma_{8}}\mapsto P_{\sigma_{10}}$, with each point in the cycle being a regular point for $\nu_{{\tilde{p}},{\tilde{q}}}$.
     \item $\nu_{{\tilde{p}},{\tilde{q}}}\left(\hat{C}_{\tau_2}\right)=\nu_{{\tilde{p}},{\tilde{q}}}\left(\hat{C}_{\tau_4}\right)=P_{\sigma_{8}}$ and $\nu_{{\tilde{p}},{\tilde{q}}}\left(\hat{C}_{\tau_1}\right)=\nu_{{\tilde{p}},{\tilde{q}}}\left(\hat{C}_{\tau_5}\right)=P_{\sigma_{5}}$.
     \item $\nu_{{\tilde{p}},{\tilde{q}}}\left(\hat{C}_{\tau_0}\right)=\hat{C}_{\tau_4}$, $\nu_{{\tilde{p}},{\tilde{q}}}\left(\hat{C}_{\tau_6}\right)=\hat{C}_{\tau_5}$ and $\nu_{{\tilde{p}},{\tilde{q}}}\left(\hat{C}_{\tau_3}\right)=\hat{C}_{\tau_1}.$
    
 \end{itemize}
In particular, all (proper transforms of) exceptional lines and the proper transforms of the coordinate axes are eventually contracted to a point in the $3$-cycle and thus they are not contracted to a point in a destabilizing orbit.

\textbf{Step 3:} The only curves in $\mathbb{CP}^2$ except coordinate axes that were collapsed
by $\nu_{{\tilde{p}},{\tilde{q}}}: \mathbb{CP}^2 \dashrightarrow \mathbb{CP}^2$ were the ${\tilde{p}}+{\tilde{q}}$ irreducible curves $H^h_{\omega}$ and $K^v_\zeta$.   We verify that their proper transforms $\widehat{H^\textit{h}_\omega}, \widehat{K^v_\zeta} \subset Y$ do
not lead to destabilizing orbits. Indeed, by a direct computation we have 
$$
\widehat{H^\textit{h}_\omega} \mapsto [\omega:0:1]' \subset \widehat{C_{\tau_1}} \mapsto P_{\sigma_5},\;\; \widehat{K^v_\zeta}\mapsto P_{\sigma_8} (\,\,\mbox{See} \,\, \eqref{eq: 46}).
$$
Here $[\omega:0:1]'$ is the lift of the regular point $[\omega:0:1]$ to $Y$.
\end{proof}

    For the case ${\tilde{p}}=1$ and ${\tilde{q}}>4$ (or equivalently ${\tilde{q}}=1$ and ${\tilde{p}}>4$) we do not have algebraic stability for the lift of $\nu_{{\tilde{p}},{\tilde{q}}}$ to $Y$. Indeed, consider blow-up coordinates $(\frac{x_1}{x_0},\frac{x_2}{x_0})=(mz,z)$ for the blow-up at $P_{\sigma_0}$ and the local coordinates $(\frac{x_0}{x_1},\frac{x_2}{x_1})=(es,s)$ for the blow-up at  $P_{\sigma_1}$. Fix the local affine coordinates $(e,s)=(e_1, e_1s_1)$ for the blow-up of $Z$ at $P_{\sigma_3}$ with  and the local affine coordinates $(m,z)=(m_1,m_1z_1)$ for the blowup of $Z$ at $P_{\sigma_6}$. 
We obtain
\begin{equation}\label{eq: 46}
    \nu_{{\tilde{p}},{\tilde{q}}}(m_1,z_1)=(\tilde{e_1},\tilde{s_1})=\left(\frac{z_1(m_1^{\tilde{q}}+1)}{m_1^{{\tilde{q}}-1}},\frac{m_1^{\tilde{q}}z_1^{{\tilde{p}}-1}(m_1^{\tilde{q}}+1)^{{\tilde{p}}-1}}{m_1^{{\tilde{p}}{\tilde{q}}-{\tilde{p}}}+z_1^{{\tilde{p}}}(m_1^{\tilde{q}}+1)^{\tilde{p}}}\right).
\end{equation}
If ${\tilde{p}}=1$ the cancellations of \eqref{eq: 46} imply that $\widehat{K^v_\zeta}$ is eventually collapsed to a point in a destabilizing orbit.

\begin{rem}\label{Rem: stability p=2, q=2}
\cref{Propo: Negative pairs} is done in the context of ${\tilde{p}},{\tilde{q}}\geq 2$ with ${\tilde{p}}{\tilde{q}}>4$. Nevertheless, it applies also to the map $\nu_{2,2}:Y\dashrightarrow Y$ which turns out to be algebraically stable. This is due to the tropicalization of $\nu_{2,2}$ being a degenerate version of the tropicalization appearing in \cref{fig:neg_pq_fan}. More precisely, the reader can check that the only difference between these tropicalizations is that for the version of \hyperref[fig:neg_pq_fan]{7c} for $\nu_{2,2}$ we have $A(\tau_2)=\tau_6$ and $A(\tau_5)=\tau_2.$ The reasoning to show algebraic stability is completely analogous.   
\end{rem} 

\subsection{Existence of superattracting 3-cycle.}
Let $Y$ be the surface result of blowing-up $\CP^2$ four times, as described in Section \ref{SEC:TROP_NEG_PQ}, then the following result follows:
\begin{lemma}\label{Lemma: Superattracting negative pair}
     Let ${\tilde{p}},{\tilde{q}}\geq 2$ and ${\tilde{p}}\tilde{q}>4$  then $\nu_{{\tilde{p}},{\tilde{q}}}=\mu_{p,q}:Y\dashrightarrow Y$ has a superattracting 3-cycle. 
\end{lemma}
\begin{proof} The reader can check on local coordinates that $\nu_{{\tilde{p}},{\tilde{q}}}^{3}$ has a superattracting fixed point at $P_{\sigma_5}$. Another way to see this is the following, notice that by the tropicalization 
\begin{equation*}
    \nu_{{\tilde{p}},\tilde{q}}^3\left(\hat{C}_{\tau_2}\right)=\nu_{{\tilde{p}},\tilde{q}}^3\left(\hat{C}_{\tau_4}\right)=P_{\sigma_5}.
\end{equation*}

As the lines $\hat{C}_{\tau_2}$ and $\hat{C}_{\tau_4}$ are transverse lines passing through $P_{\sigma_5}$, then $D(\nu^3_{{\tilde{p}},{\tilde{q}}})_{P_{\sigma_5}}=0$.
\end{proof}
 Remark \ref{Rem: stability p=2, q=2} implies that $\nu_{2,2}:Y\dashrightarrow Y$ has a periodic point at $P_{\sigma_5}$ but this point is no longer superattracting. Indeed, going to local affine coordinates $(r,s)$ centered at $P_{\sigma_5}$ we have $\nu_{2,2}^3(r,s)=(r^2sQ_1(r,s),sQ_2(r,s))$.
\begin{rem}\label{rem: p=-1, q<-4}
    When ${\tilde{p}}=1, {\tilde{q}}>4$ or ${\tilde{q}}=1, {\tilde{p}}>4$, one can do similar to \cref{Propo: Negative pairs} to see that there is a superattracting periodic point after blowing up torus invariant points in a similar fashion, but in these cases one needs 7 of them to see a superattracting periodic point.
\end{rem}
\subsection{Dynamical degree $\lambda_1(\nu_{{\tilde{p}},{\tilde{q}}})$}\label{Subsc: 8.3 Dynamical degree negative}
In this section, we employ a similar approach to that given in the proofs of \cref{prop:coh_action_nonss,prop:coh_action_ss} to compute the $\lambda_1(\nu_{{\tilde{p}},{\tilde{q}}})$ when ${\tilde{p}}\geq 2, \tilde{q}\geq2$. 

\begin{proof}[Proof of Theorem \ref{thm:C}]
Denote $E_3, E_4, E_5, E_6$ classes of the exceptional curves $\hat{C}_{\tau_3}, \hat{C}_{\tau_4}, \hat{C}_{\tau_5}, \hat{C}_{\tau_6}$ respectively, and H be the class of a generic line in $Y$. The reader can check that the classes
\begin{align*}
    [\pi'\{x_0=0\}] &= H-E_3-2E_6\\
    [\pi'\{x_1=0\}] &= H-E_4-2E_5\\
    [\pi'\{x_2=0\}] &= H-E_3-E_4-E_5-E_6\\
    [\pi'\{H^h_{\omega}=0\}] &= ({\tilde{q}}+1)H-E_3-E_6-({\tilde{q}}+1)E_5-{\tilde{q}}E_4\\
    [\pi'\{K^v_{\zeta}=0\}] &= H-E_4-E_5
\end{align*}
and the action of the induced map is
\begin{align*}
    \nu_{{\tilde{p}},{\tilde{q}}}^*(H-E_3-2E_6)&=0\\
    \nu_{{\tilde{p}},{\tilde{q}}}^*(H-E_4-2E_5) &= {\tilde{p}}(({\tilde{q}}+1)H-E_3-E_6-({\tilde{q}}+1)E_5-{\tilde{q}}E_4) + E_3\\
    \nu_{{\tilde{p}},{\tilde{q}}}^*(E_3) &= ({\tilde{p}}-1){\tilde{q}}(H-E_4-E_5) + ({\tilde{p}}-1)E_4 + ({\tilde{p}}-2)(H-E_3-E_4-E_5-E_6)\\
    \nu_{{\tilde{p}},{\tilde{q}}}^*(E_5) &=E_6 \\
    \nu_{{\tilde{p}},{\tilde{q}}}^*(E_6) &= E_4 + (H-E_3-E_4-E_5-E_6) + {\tilde{q}}(H-E_4-E_5)    
\end{align*}
and $\lambda_1(\nu_{{\tilde{p}},{\tilde{q}}}) = \la_1(\mu_{p,q})$ is the largest real valued root of the polynomial equation
\begin{equation}\label{dynamical_degree}
   C(x)= x^4-{\tilde{p}} {\tilde{q}} x^3-2 x^3+3 x^2-2 x+1  =x^4-pq x^3-2 x^3+3 x^2-2 x+1 =0.
\end{equation}
\end{proof}
\subsection{Numerical experiment to find approximate value of $\lambda_1$}\label{subsec:NUM_APPROX}Let $f: \mathbb{CP}^2\dashrightarrow \mathbb{CP}^2$ be a dominant rational self-map defined over $\mathbb{Q}$. Let $P=[x_0:y_0:z_0]$ be a generic rational point and $f^n(P)$ denote the $n^{th}$ iterate of $P$. The Kawaguchi–Silverman conjecture (\cite{kawaguchiExamplesDynamicalDegree2014}, Conjecture 1), says that the norm of the lift $\widetilde{\|f^n(P)\|}$ behaves approximately as $ \|\widetilde{P}\|^{\lambda_1^n}$, where $\|\cdot\|$ is Euclidean norm and $\widetilde{P}$ represents the lift of the point $P$ in $\mathbb{CP}^2$ to $\C^3\setminus \{0\}$ with relatively prime integer coordinates. Moreover, (\cite{kawaguchiDynamicalArithmeticDegreesof2016}, Theorem 4) gives that the dynamical degree of $f$ is at least the growth rate of the degree of norm of the initial point $\|\widetilde{P}\|$.  This method of approximating dynamical degrees has been used before in the literature; see e.g. \cite[Section 5]{AMV}.

Proceeding from the numerical experiment described above we inferred that for $p=-1$ and $q=-2,-3$ or $-4$, $\lambda_1(\mu_{p,q})$ is greater than $1$. More specifically, when $p=-1$ and $q=-2$, after around 15 iterates, $\lambda_1$ is approximately 3.28 which is consistent with the actual computation. Furthermore, if $p=-1$, $q=-3$, the approximate value for $
\lambda_1$ is in between 4 and 5 and when $p=-1$ and $q=-4$, it is between 5 and 6.

\section{Entropy and further dynamical consequences}\label{sec:entropy_and_futher_dynamical_consequences}
In this section, we will use the results proved in \cite{dillerInvariantMeasureLyapunov2001} and \cite{DillerDynamicsMeromorphicMaps2009b} to show that $\cref{thm:Ergodic_properties}$ holds. 
\begin{defi}
    An algebraically stable birational self-map $f$ of a projective surface $X$ is \textbf{separating} if $$\displaystyle \overline{\bigcup_{n=0}^{\infty}f^{-n}(I(f))}\,\bigcap\,\,\overline{\bigcup_{n=0}^{\infty}f^{n}(I(f^{-1}))} = \emptyset.$$
\end{defi}
\begin{theo}{\cite{dillerInvariantMeasureLyapunov2001}}\label{thm:diller_invariant_measure}
    Let $X$ be a proper-modification of $\CP^2$ and $f:X\dashrightarrow X$ be algebraically stable separating birational self-map with $\lambda_1(f)>1$, then
    \begin{enumerate}
        \item There exists a $f$-invariant measure $\nu$.
        \item $\nu$ is mixing with respect to $f$.
        \item $\|\log^+(Df)\|$ is $\nu-$integrable and hence by Osedelec's theorem, the Lyapunov exponents $\chi_+(\nu,f)$ and $ \chi_-(\nu,f)$ exists.
    \end{enumerate}
\end{theo}
The below theorem is proved in great generality in \cite{DillerDynamicsMeromorphicMaps2009b}, but we are stating it in the context of interest in our paper.
\begin{theo}[Theorem B, \cite{DillerDynamicsMeromorphicMaps2009b}]\label{thm:ddg2009} Let $f:X\dashrightarrow X$ be algebraically stable separating birational map  with $\lambda_1(f)>1$, then the Lyapunov exponents $\chi_+(\nu,f)$ and $\chi_-(\nu,f)$ of $\nu$ under $f$ satisfy
        \begin{align*}
           \chi_+(\nu,f) \geq \log(\lambda_1) / 2 > 0 > - \log(\lambda_1) / 2 \geq \chi_-(\nu,f)
        \end{align*}
and $h_{\nu}(f) = \log{\lambda_1(f)}$, hence $\nu$ is a measure of maximal entropy and $h_{top}(f)= \log{\lambda_1(f)}$.
\end{theo}

\begin{proof}[Proof of \cref{thm:Ergodic_properties}]
    For $p,q>0$ and $pq>4$, By \cref{propo:stable_model}, one can see that the lifts of the map $\mu_{p,q}: X_{p,q} \dashrightarrow X_{p,q}$ is algebraically stable and when $pq-p-q>0$, $$\overline{\bigcup_{n=0}^{\infty}{(\mu_{p,q})}^{n}(I((\mu_{p,q})^{-1}))} = \{ P_{\sigma_{1}}\}  = \{[0:1:0]'\}$$ and $$\overline{\bigcup_{n=0}^{\infty}{(\mu_{p,q})}^{-n}(I(\mu_{p,q}))} = \{ P_{\sigma_{0}}\}  = \{[1:0:0]'\}$$ which implies separating.   Here, $[0:1:0]'$ and $[1:0:0]'$ denote the lifts of $[0:1:0]$ and $[1:0:0]$ from $\mathbb{CP}^2$ to $X_{p,q}$.  Similarly one can check when $p>4,q=1$ and $p=1,q>4$ that $\mu_{p,q}$ is separating on a twice further blown-up copy of $X_{p,q}$.\\~\\
    If $p, q \leq -2$, By \cref{Propo: Negative pairs},  $\mu_{p,q}: Y\to Y$ is algebraically stable and $$\overline{\bigcup_{n=0}^{\infty}{(\mu_{p,q})}^{n}(I((\mu_{p,q})^{-1}))} = \{[\omega:0:1]', P_{\sigma_{5}}, P_{\sigma_{8}}, P_{\sigma_{10}}\}$$ and $$\overline{\bigcup_{n=0}^{\infty}{(\mu_{p,q})}^{-n}(I(\mu_{p,q}))} = \{ [0:\zeta: 1]', P_{\sigma_{9}}, P_{\sigma_{4}},P_{\sigma_{7}}\}$$ where $[\omega:0:1]',[0:\zeta:1]'$ is the lift of $[\omega:0:1],[0:\zeta:1]$ respectively from $\mathbb{CP}^2$ to $Y$ for all $\omega^p+1=0$ and $\zeta^q+1=0$ which implies that the lift of $\mu_{p,q}$ is separating and by \cref{THM:NEGPQ_NO_FIBR}, $\lambda_1(\mu_{p,q})>1$.\\ In all the cases when $(p,q)\in \N^2$ and $pq>4$ or in the cases $p,q\leq -2$,  \cref{thm:Ergodic_properties} follows from \cref{thm:diller_invariant_measure,thm:ddg2009}. By \cite[Corollary 3.5]{DDG2}, any algebraically stable separating birational mapping has finite dynamical energy, so Conclusion $(v)$ follows from By \cite[Theorem 1]{DDG2}.  Note
    any algebraic curve $C$ in $X_{p,q}$ is pluripolar.
\end{proof}

\section{Computer experiment for the case $p = -1, q=-5$.}\label{sec:computer_experiment}
We used the Fractalstream \cite{FRACTALSTREAM} computer software
to make a picture of the initial conditions in the positive real quadrant $Q:= \{(x,y) \in \mathbb{R}^2 \, : \, x,y > 0\}$ whose orbits escape to infinity under $\mu_{-1,-5}$.   They are shown in green in Figure~\ref{FIG:COMPUTER_BASINS} and points whose orbits remain bounded are shown in black. Recall
from Remark \ref{rem: p=-1, q<-4} that when $p=-1$ and $q=-5$ the mapping $\mu_{p,q}$ has a superattattracting cycle on a blown-up copy of $\mathbb{CP}^2$, that is disjoint from $Q$ (the superattracting cycle lies on the poles of the invariant meromorphic form $\eta = \frac{dx \wedge dy}{xy}$).  We expect that the points colored green in Figure \ref{FIG:COMPUTER_BASINS} are in the basin of attraction of this periodic cycle.

It is noteworthy that there seems to be a large open subset of $Q$ consisting of points whose orbits remain bounded (colored black in the figure).  In contrast, when $p, q > 0$ and $pq > 4$, Machachek and Ovenhouse \cite[Theorem 2.5]{machacek2024discrete} proved that every orbit in $Q$ is unbounded.

To explore the dynamics in the black region, we computed five different finite orbits (each of length 200) which are shown in white, yellow, orange, blue, and purple, respectively in the figure.  On compact subsets of $Q$ the invariant two form $\eta$ is comparable to the Euclidean area.  The finite orbits computed here are reminiscent of the KAM phenomenon seen on small perturbations of area preserving integrable systems; see e.g. \cite[Section 4.8]{GH}.  Indeed the orbits in white, yellow, and orange appear to be on invariant KAM curves, the orbit in purple appears to be on a family of resonant islands, and the orbit in blue may be near a Smale horseshoe cased by transverse intersections of stable and unstable manifolds
of saddle-type periodic orbits.   Clearly all of this is speculation and we invite the reader to make a more rigorous study of it.

We have also observed a similar phenomenon for other values of $q < -5$ when $p=-1.$

\begin{figure}[h!]
    \includegraphics[scale=0.45]{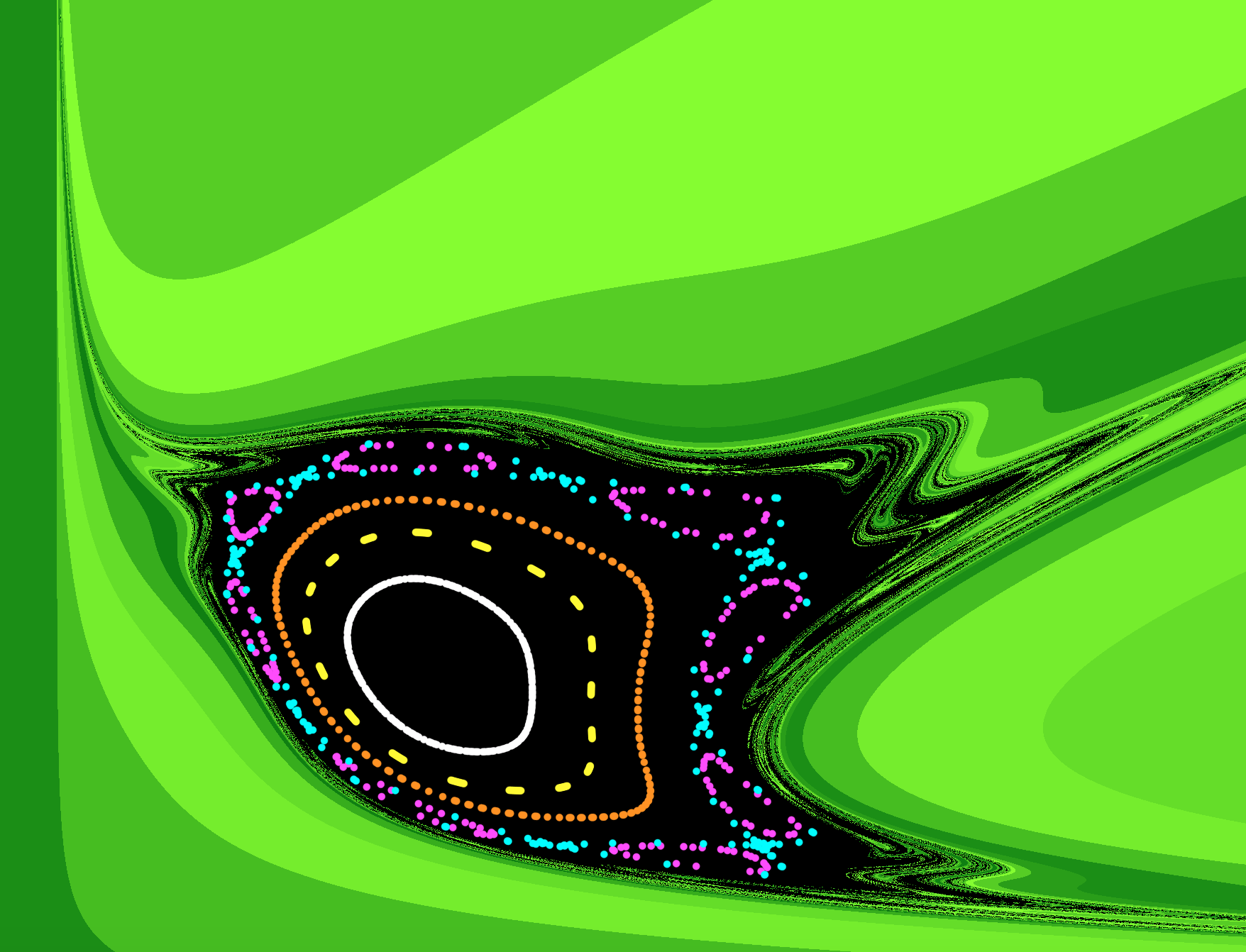}\caption{\label{FIG:COMPUTER_BASINS}Computer plot showing approximately the region $0 < x < 5$ and $0 < y < 5$.   Points in green have orbit escaping to infinity.    Points in black have a bounded orbit.  Five different finite orbits (each of length 200) are shown in white, yellow, orange, light blue, and purple, respectively.}
\end{figure}

\appendix 

\section{Computations and proofs for \cref{sec:picard}}\label{sec:appx_pullback}
\subsection{Computation of pullbacks for \cref{prop:coh_action_ss}}

The preimages of exceptional lines are the critical curves that were collapsing to the corresponding blow-up points. These are illustrated on Figure \ref{fig:crit_mu23}.

Let $Q = [0:\zeta:1]\in T_q$. An open subset of the preimage is $U_{\ze}= \{y-\zeta^{-1} = 0, x\neq 0,y\neq 0\}$, where $(x,y) = (\frac{x_0}{x_2}, \frac{x_1}{x_2})$ are local affine coordinates in the codomain. Recall the local coordinates $(u_0,v_0)$ (see \cref{subsubsec:comp_Tq}). In the coordinates $(u_0,v_0)$ taken near the blow-up point $Q$ the local equation of the exceptional line is $u_0 = \zeta$. Taking the pullback of this equation with respect to $\mu_{p,q}$ we obtain the equation 

$$\frac{x^p + (1+y^q)^p-\zeta x^p y}{x^py} = \frac{(1-\zeta y)\left(x^p + \frac{(1+y^q)^p}{(1-\zeta y)}\right)}{x^py} = 0.$$
This equation contains $y - \zeta^{-1} = 0$ with the multiplicity $1$. 

The line $\{x_1-\zeta^{-1}x_0 = 0\} = \overline{U_{\ze}} \subset \CP^2$ passes through only one blow-up point $[0:\ze^{-1}:1]$ and this point has multiplicity $1$ (as any point on a line). Therefore by \eqref{eqn:int_blowup} the pullback is given by \eqref{eqn:picmuexcz}.

Let $P = [\omega:0:1]\in S_p$. An open subset of the preimage is $U_{\om}=\{x-\omega^{-1}(y^q + 1)=0, x\neq 0, y\neq 0\}$. Recall the local coordinates $(r_0,s_0)$ (see \cref{subsubsec:comp_Sp}). In the coordinates $(r_0,s_0)$ the local equation of the exceptional divisor is $r_0 = \omega$. The pullback of this equation to the affine chart with coordinates $(x,y) = (\frac{x_0}{x_2}, \frac{x_1}{x_2})$ gives the equation 
$$
\frac{x-\omega^{-1}(y^q+1)}{x},
$$
which contains $x-\omega^{-1}(y^q + 1)$ with multiplicity $1$.

The closure $\overline{U_{\om}} \subset\CP^2$ is $\Kh_\omega = \{x_2^{q-1}x_0 - \omega^{-1}(x_1^q + x_2^q) = 0\}$ passes through blow-up points $T_q \cup \{\iota(P)\}$. All these points belong to the affine chart $\{x_2 \neq 0\}$. In local coordinates $(x,y)$ in this chart the curve is given by $x - \omega^{-1}(y^q + 1) = 0$ and clearly is smooth at each its point. Therefore all points in  $T_q \cup \{\iota(P)\}$ have multiplicity $1$. Combining this with the fact that the curve $\Kh_\omega$ has degree $q$ by \eqref{eqn:int_blowup} we obtain \eqref{eqn:picmuexcw}.

Let $L\subset X_{p,q}$ be the proper transform of a generic line $\{\al x_0 + \beta x_1 + \gamma x_2 = 0\} \subset \CP^2$. The preimage on an open subset in coordinates $(x,y)= (\frac{x_0}{x_2}, \frac{x_1}{x_2})$ is given by $\{\al x^{p-1}y(1+y^q) + \beta(x^p+(1+y^q)^p) + \gamma x^py = 0\}$.

We compute the closure of this curve in $X_{p,q}$ as the proper transform of this affine curve in $\CP^2$ which is given by
$$
K_L := \{\al x_0^{p-1}x_1(x_2^q+x_1^q)x_2^{p(q-1)-q} + \beta(x_0^px_2^{p(q-1)}+(x_2^q+x_1^q)^p) + \gamma x_0^px_1x_2^{p(q-1)-1} = 0\}.
$$

This curve passes through all blow-up points $S_p \cup T_q$.

Let $P = [\omega : 0 : 1] \in S_p$. Then in coordinates $(x,y)$ we have
$$\partial_y(\al x^{p-1}y(1+y^q) + \beta(x^p+(1+y^q)^p) + \gamma x^py)|_{P} = -(\omega^{-1}\al+\gamma).$$

Therefore from assumption $\al^p+(-\gamma)^p\neq 0$ (recall that $\al, \be, \ga$ are generic) we get that the curve $K_L$ is smooth at all points of $S_p$, hence each of these points has multiplicity $1$ on $K_L$. 

Let $Q = [0:\zeta:1] \in T_q$. Denote for convenience $\psi:=\mathop{\prod}\limits_{\zeta^{\prime}\in R_{q}\backslash\{\zeta\}}(\zeta-\zeta^{\prime})\neq 0$.

The lower order terms of the expansion of $\al x^{p-1}y(1+y^q) + \beta(x^p+(1+y^q)^p) + \gamma x^py$ at $Q$ is given by 
$$(\be + \ze \ga)x^p + (\al \ze\psi)x^{p-1}(y-\ze) + \be \psi^p (y-\ze)^p.$$

This term is non-zero provided that $\be^q + (-\ga)^q\neq 0$. Therefore the multiplicity of the point $Q$ is $p$. 

Combining these multiplicities with the degree of the curve $K_L$, which is $pq$ we obtain \eqref{eqn:picmuline}. 

\subsection{Proof of \cref{lemma:spectrum_ss}}
\begin{proof}
    For any $P\in S_p\cup T_q$ denote 
    
    $$e_{P,\pm} := \frac{E_{P}\pm E_{\iota(P)}}{2}.$$
    
    Denote $A = \sum_{P\in S_p}E_{P},\;\;B = \sum\limits_{P\in T_q}E_{P}.$
    
    Note that there is a decomposition of $\rV_{p,q}$ into a sum of two $\mupq^*$ invariant subspaces
    $$\rV_{p,q} = \rV_{+} \oplus\rV_{-},$$
    where $\rV_{+} = \left<\{C\}\cup\{e_{P,+}\}_{P\in S_p\cup T_q}\right>$ and $\rV_{-} = \left<\{e_{P,-}\}_{P\in S_p\cup T_q}\right>$.
    
    We have $\mupq^*|_{\rV_{-}} = \mathrm{Id}_{\rV_{-}}$ and $\dim(\rV_{-}) = \left\lfloor{\frac{p}{2}}\right\rfloor  + \left \lfloor{\frac{q}{2}}\right \rfloor$.
    
    It remains to diagonalize the action of $\mupq^*$ on $\rV_{+}$. First, note that the subspace $\left<\{A,B,C\}\right> \subset \rV_{+}$ is invariant with respect to $\mupq^*$. The action on that subspace is given by 
    \begin{equation}
        \mupq^*(A) = p(q C - B) - A,\;\;\;\mupq^*(B) = qC - B,\;\;\;\mupq^*(C) = p\big(q C - B\big) - A.
    \end{equation}
    This action can be diagonalized explicitly. The eigenvalues are $\lambda_{+}, \lambda_{-}, 0$.

    Consider the action on the quotient space $\hat{\rV} = \rV_{+}/\left<\{A,B,C\}\right>$. The space $\hat{\rV}$ is generated by classes of vectors $\{e_{P,+}\}_{P\in S_p\cup T_q}$. The action on such a vector is given by 
    \begin{equation}
        \mupq^*(e_{P}) = \begin{cases}
            q C - B - e_{P},\;\; P \in S_p,\\
            C - e_{P},\quad\qquad P \in T_q.\\
        \end{cases}
    \end{equation}
    Therefore we obtain $[\mupq^*] = -\mathrm{Id}_{\hat{\rV}}$. Note that $\dim\big(\hat{\rV}\big) = \left \lceil{\frac{p}{2}}\right \rceil + \left \lceil{\frac{q}{2}}\right \rceil -2$.
    
    We observe $(\mupq^* + \mathrm{Id}_{\rV})(\mupq^* - \mathrm{Id}_{\rV})(\rV) \subset \left<\{A,B,C\}\right>$, hence 
    \begin{equation}\label{eqn:min_poly}
    (\mupq^* -\lambda_{+} \mathrm{Id}_{\rV})(\mupq^* -\lambda_{-} \mathrm{Id}_{\rV})(\mupq^*)(\mupq^* + \mathrm{Id}_{\rV})(\mupq^* - \mathrm{Id}_{\rV})(\rV) =0.
    \end{equation}
    The assumption $pq-p-q >0$ implies $pq > 4$, which implies $\lambda_{+}>1$ and since $\lambda_{+}\lambda_{-} = 1$, $0 < \lambda_{-} < 1$. Therefore all factors in \eqref{eqn:min_poly} are distinct so the minimal polynomial of $\mupq^*$ splits and then $\mupq^*$ is diagonalizable.
\end{proof}
\subsection{Proof of \cref{prop:coh_action_nonss}}
\begin{proof}
    Let us denote by $X_{p,1}^{(0)} \overset{\pi_0}{\rightarrow} \CP^2$ at points of $S_{p}\cup T_1$. Denote by $X_{p,1}^{(1)} \overset{\pi_1}{\rightarrow} X_{p,1}^{(0)}$ the blowup of $X_{p,1}^{(0)}$ at the point $\pi_0^{-1}([0:1:0])$ and by $X_{p,1} \overset{\pi_2}{\rightarrow} X_{p,1}^{(1)}$ at the point $([0:1:0], [1:0])$ (Cf. Section \ref{sec:stable_model}).

    Images of exceptional lines and proper transforms of the coordinate axes intersect the exceptional lines corresponding to the points in $S_p\cup T_1$ in a finite number of points. Therefore to compute pullbacks of divisors $E_{P}$ for $P \in S_p\cup T_1$ it is sufficient to compute the preimage of the corresponding exceptional lines (counting multiplicity) on the open subset $U = \{x_{0}x_1x_2\neq 0\}$ then compute the closure of this preimage in $\CP^2$ and to compute the class of the  proper transform of this closure. The closures are given by $K^{\textit{h}}_{\om} = \{x_0 -\omega^{-1}(x_1 + x_2) = 0\}$ and $K^{\textit{v}}_{-1} = \{x_1 +x_0 = 0\}$ for the divisors $E_{[\om:0:1]}$ and $E_{[0:-1:1]}$ respectively. Note that both these closures do not pass through the point $[0:1:0]$. The rest of computation repeats the corresponding part of the proof of  \cref{prop:coh_action_ss}. Therefore we obtain \eqref{eqn:picmuomnss} and \eqref{eqn:picmuzenss}.

    In order to compute the pullbacks of the remaining classes we compute the pullbacks for classes of proper transforms of three kinds of lines in $\CP^2$. For all of them it is sufficient to compute the preimage on the open subset $U =\{x_0x_1x_2\neq 0\}$ and then take the class of closure in $X_{p,1}$.

    \textbf{A generic line.}
    Let $L = \{\al x_0 + \be x_1 + \ga x_2 = 0\}$ with $\al, \be, \ga$ generic. The class of the proper transform of $L$ is $C$. 

    The closure of the preimage of $L$ in $\CP^2$ under map $\mu_{p,q}$ is given by
    $$K_L=\{\al x_0^{p-1}x_1(x_1 + x_2) + \be x_2 (x_0^p + (x_1 + x_2)^p) + \ga x_0^px_1 = 0\}.$$

    The points $Q = [0:-1:1]$ and $P = [\om :0:1]\in S_p$ have multiplicities $p$ and $1$ respectively on $K_L$. That implies
    $$[\pi_0^{\prime}(K_L)] = \pi_0^*([K_L]) - \mathop{\sum}\limits_{P \in S_p }E^{(0)}_{P} - pE^{(0)}_{Q}.$$  Here $E^{(0)}_{P}$ are the classes of exceptional lines over blow-up points $P \in S_p\cup T_1$.
    
    Now we compute the class of the proper transform of this curve with respect to two remaining blowups at $[0:1:0]$.

    In the local coordinates $(x,z) = \left(\pi_0^*(\frac{x_0}{x_1}), \pi_0^*(\frac{x_2}{x_1})\right)$ near $\pi_0^{-1}([0:1:0])$ the curve $\pi_0^{\prime}(K_L)$ is given by the equation
        $$
                \al x^{p-1}(1 + z) + \be(x^p + (1+z)^p)z + \ga x^p =0.
        $$
        The point $(x,z) = (0,0)$ has multiplicity $1$. Therefore we get 
        $$
        [\pi_1^{\prime}(\pi_{0}^{\prime}(K_L))] = \pi_{1}^{*}([\pi_{0}^{\prime}(K_L)]) - E_{1,\infty}^{(1)}.
        $$
        Here $E_{1,\infty}^{(1)}$ is the exceptional divisor corresponding to the line $\pi_{1}^{-1}(\pi_{0}^{-1}([0:1:0]))$.

        The proper transform of the curve $\pi_1^{\prime}(\pi_{0}^{\prime}(K_L))$ in the coordinates $(u_1, v_1)$ (see Section \ref{subsec:stable_model_nss}) is given by 
        $$
        \al u_1^{p-2}(1+ u_1v_1 ) + \be(u_1^p + (1 + u_1 v_1)^p) + \ga u_1^{p-1} = 0.
        $$
        The multiplicity of the point $(u_1, v_1) = (0,0)$ is $1$, therefore we get
        $$
        [\pi_2^{\prime}(\pi_1^{\prime}(\pi_{0}^{\prime}(K_L)))] = \pi_{2}^{*}([\pi_1^{\prime}(\pi_{0}^{\prime}(K_L))]) - E_{2,\infty} = (\pi_{1}\circ \pi_{2})^*(\pi_0^*([K_L]) - \mathop{\sum}\limits_{P \in S_p }E^{(0)}_{P} - pE^{(0)}_{Q}) - \pi_{2}^*(E_{1,\infty}^{(1)})  - E_{2,\infty}.
        $$
        
        Note that the point $(u_1,v_1) = (0,0)$ has multiplicity $1$ on the line $\pi_{1}^{-1}(\pi_{0}^{-1}([0:1:0]))$. That implies
        $$\pi_2^*(E_{1,\infty}^{(1)}) = E_{1,\infty} + E_{2,\infty}.$$
        Also since $K_L$ has degree $p+1$ and the exceptional lines corresponding to the blowup $\pi_0$ do not pass through $\pi_{0}^{-1}([0:1:0])$ we have 
        $$(\pi_{1}\circ \pi_{2})^*(\pi_0^*([K_L]) = (p+1)C,\;\; (\pi_{1}\circ \pi_{2})^*(\pi_0^*([E^{(0)}_P]) = E_P, \;\forall P\in S_p\cup T_1.$$
        Hence we conclude
        \begin{equation}\label{eqn:nonsscases_tmp0}
        \mu_{p,1}^*(C) = [\pi_2^{\prime}(\pi_1^{\prime}(\pi_{0}^{\prime}(K_L)))] = (p+1)C - \mathop{\sum}\limits_{P \in S_p }E_{P} - pE_{Q} - E_{1,\infty} - 2E_{2,\infty}. 
        \end{equation}

    \textbf{A generic line passing through $[0:1:0]$.}
    Let $L = \{\al x_0 + \ga x_2 = 0\}$ with $\al, \ga$ generic. 
    The line $L$ does not pass through points in $S_p\cup T_1$ and passes through $[0:1:0]$. The proper transform of $L$ with respect to $\pi_1\circ \pi_0$ does not pass through the point of the last blowup. This implies
    $$[(\pi_2\circ \pi_1\circ \pi_0)^{\prime}(L)] = C - E_{1,\infty} - E_{2,\infty}.$$

        On open subset $U$ the preimage (of the proper transform) of $L$ is $M_{L} = \{\al(x_1 + x_2) + \ga x_0 = 0\}$. Taking the closure in $\CP^2$ we obtain a line passing through one and only one blow-up point $Q \in T_1$. Therefore the class of the proper transform is $C - E_{Q}$. Therefore
        \begin{equation}\label{eqn:nonsscases_tmp1}
            \mu_{p,1}^*(C- E_{1,\infty} - E_{2,\infty}) = C - E_{Q}.
        \end{equation}
    \textbf{Specific line.}

        Let $L = \{x_2 = 0\}$. This line does not pass through points in $S_p\cup T_1$. It passes through the point $[0:1:0]$ and the proper transform with respect to blowup $\pi_1\circ\pi_0$ passes through the point of the last blowup. Therefore we get 
        $$[\pi_{2}^{\prime}((\pi_1\circ \pi_0)^{\prime})(L)] = \pi_2^*([(\pi_1\circ \pi_0)^{\prime})(L)]) - E_{2,\infty} = (\pi_2\circ\pi_1\circ\pi_0)^*(L_z) - \pi_2^*(E_{1,\infty}^{(1)}) - E_{2,\infty} = C - E_{1,\infty} - 2E_{2,\infty}.$$

        On open subset $U$ the preimage of the line $L$ is empty. This implies that the preimage of the proper transform of $L$ is finite. This gives 
        
        \begin{equation}\label{eqn:nonsscases_tmp2}
            \mu_{p,1}^*(C - E_{1,\infty} - 2E_{2,\infty}) = 0.
        \end{equation}

    Linearly combining equations \eqref{eqn:nonsscases_tmp0}, \eqref{eqn:nonsscases_tmp1}, \eqref{eqn:nonsscases_tmp2} we obtain \eqref{eqn:picmuinf1} and \eqref{eqn:picmuinf2}.
    \end{proof}
\subsection{Proof of \cref{lemma:spectrum_nonss}}
    \begin{proof}
The proof is similar to the one of  \cref{lemma:spectrum_ss}. For any $P\in S_p$ denote $$e_{P,\pm} = \frac{E_{P}\pm E_{\iota(P)}}{2}$$
Denote $A = \suml_{P_1\in S_p}E_{P_1}, B = E_Q$, where $Q \in T_1$. Recall that $C$ is the class of the proper transform of a generic line.

Note that there is a decomposition of $\rV_{p,q}$ into a sum of two $\mupo^*$ invariant subspaces
    $$\rV(\rX_{p,q}) = \rV_{+} \oplus\rV_{-},$$
    where $\rV_{+} = \left<\{B,C,E_{1,\infty},E_{2,\infty}\}\cup\{e_{P,+}\}_{P\in S_p}\right>$ and $\rV_{-} = \left<\{e_{P,-}\}_{P\in S_p}\right>$.

    We have $\mupo^*|_{\rV_{-}} = \mathrm{Id}_{\rV_{-}}$ and $\dim(\rV_{-}) = \left\lfloor{\frac{p}{2}}\right\rfloor.$

    It remains to compute the action of $\mupo^*$ on $\rV_{+}$. First, note that the subspace $\rW = \left<\{A,B,C, E_{1,\infty}, E_{2,\infty}\}\right>$ is invariant with respect to $\mupo^*$. The action on this subspace is given by 
    \begin{align*}
        \mupo^*(A)& = pC - pB - A,\;\;\;\mupo^*(B) = C - B,\;\;\;\mupo^*(C) = (p+1)C - p B - A - E_{1,\infty}-2E_{2,\infty},\\
        \mupo^*(E_{1,\infty})&= (p-1)C + (2-p)B - A - E_{1,\infty}-2E_{2,\infty},\;\;\; \mupo^*(E_{2,\infty}) = C - B.
    \end{align*}

    By explicit computation we obtain the characteristic polynomial of the operator $\mupo^*|_{\rW } \in \mathrm{End}_{\C}(\rW )$
    $$\mathcal{P}(\lambda) = \lambda^3(\lambda - \lambda_{+})(\lambda- \lambda_{-}).$$
    Also we have $\dim(\ker(\mupo^*|_{\rW})) = 2$.

    The quotient space $\hat{\rV}=\rV_{+}/\rW$ is generated by the classes of $e_{P, +}$ where $P \in S_p$ and the action on such a vector is given by 
    $$ \mupo^*(e_{P, +}) = -e_{P,+} + C - B,$$
    which implies for the induced operator $[\mupo^*]\in \mathrm{End}_{\C}(\hat\rV)$, $$[\mupo^*] = - \mathrm{Id}_{\hat\rV}.$$

    By additivity of spectrum, this proves that the spectrum of $\mupo^*$ is of the form $\eqref{eqn:specnonss}$. Since the characteristic polynomials of $\mupo^*|_{\rW}$ and $[\mupo^*]$ are coprime there is a $\mupo^*$-invariant complement to $\rW$ in $\rV_{+}$ isomorphic to $\hat{\rV}$ as a $\mupo^*$-space.

    \end{proof}

    \section{Local Indices Method}\label{sec: Local Indices}
 
    In this appendix we provide an alternate proof of \cref{thm:A} using the 
local indices method developed in \cite{alonso2023} and \cite{alonso3D}.  This is an alternative method to the computation in the Picard group that allows to obtain the dynamical degree of a birational map of $\mathbb{CP}^N$. 
 We leave the analogous computations for $p,q < 0$ to the reader.

The idea of the method is to compute $\deg(f^n)$ by means of the local indices which are numbers associated to a polynomial and  a blow-up point in 
$\mathcal{I}(f^{-1})$ (see Definition \ref{def: Localindex}) and the computation of what is called by \cite{alonso2023} the proper pull-back of a polynomial. This is done by deducing recurrence relations for the degrees and the indices. 
This approach can be thought of as an alternative proof of \cref{Theorem Dynamical degree}. Nevertheless, this procedure is mathematically interesting by itself. In particular, it gives
a recurrence formula for the sequence $(\deg(f^n))_{n\in\Z_{>0}}$. One of the results of this Section is a closed formula for $\deg(\mu_{p,q}^n)$ for any non-affine pair $(p,q)$.  
In that perspective \cref{Theorem Dynamical degree} is a corollary of this formula and the following classical result. Let $f$ be a rational map of $\mathbb{P}^N$ then
\begin{equation}\label{Classical dyn. degree}
    \lambda_1(f)=\displaystyle\lim_{n\to\infty}(\deg(f^n))^{1/n}.
\end{equation}

In general, this method requires additional conditions, which are not satisfied for the non-semi-simple pairs but we were able to adapt it so it also works on this case.

\begin{defi}
    A hypersurface $S\subset\mathcal{C}(f)$ is called a degree lowering hypersurface of $f$ if there is some $n\in\Z_{>0}$ such that $f^n(S)\subset \mathcal{I}(f)$. 
\end{defi}

The following result mentioned in \cite[2.3]{alonso2023} illustrates how we can think of algebraical stability in terms of degree lowering curves:

\begin{propo}
    Let $f:\mathbb{CP}^2\dashrightarrow\mathbb{CP}^2$ be a rational map. The following are equivalent:
    \begin{itemize}
        \item[(i)] $f$ is algebraically stable.
        \item[(ii)] For all $n\in\Z_{>0}$, $\deg(f^n)=\left(\deg(f)\right)^n.$
        \item[(iii)] There are no degree lowering curves of $f$.
    \end{itemize}
\end{propo}

Recall that for any rational $f$ and any $n \in \Z_{>0}$, $\deg(f^n)\leq \left(\deg(f)\right)^n$ . The idea behind the local indices method is that the degree lowering components of $\mathcal{C}(f)$, are reducing the degree of the composition $f^n$. Hence, if we can account for how this reduction evolves for each of these components, then we can keep track of the sequence $(\deg(f^n))_{n\in\Z_{>0}}.$ More precisely, define the pull-back by $f:[x_0:x_1:x_2]\mapsto [f_0(x):f_{1}(x):f_{2}(x)]$ of a homogeneous polynomial $P$  as $f^{*}P:=P\circ \hat{f}$,  where $\hat{f}:\C^3\to\C^3$ is defined by $\hat{f}:(x_0,x_1,x_2) \mapsto (f_0(x),f_1(x),f_2(x))$. Then we can define the proper pull-back of $P$ as follows.
\begin{defi}
    The proper pull-back  of a polynomial $P$ with respect to a rational map $f$ is the polynomial $\tilde{P}$ such that 
\begin{equation*}
    f^{*}P=M_1^{\nu_1}M_2^{\nu_2}\dots M_k^{\nu_k}\tilde{P}
\end{equation*}

where $M_1$, \dots $M_k$ are the defining polynomials of the irreducible components of $\mathcal{C}(f)$ which $f$ contracts to points, and $\tilde{P}$ is not divisible by any $M_i$.
\end{defi}

Now, for a degree 1 homogeneous polynomial $P_0$ define inductively the sequence of polynomials $(P_n)_{n\geq 0}$ as follows: Given $P_n$, the polynomial $P_{n+1}$ is the proper pullback of $P_n$ with the associated equation:

\begin{equation}\label{eqn: Pn sequence}
    f^{*}P_n=M_1^{\nu_1(n)}M_2^{\nu_2(n)}\dots M_k^{\nu_k(n)}P_{n+1}
\end{equation}

\begin{propo}\label{Prop: generic Alonso}\cite[5.2]{alonso2023}
   For a generic $P_0$ we have $\deg(f^n)=\deg(P_n)$. 
\end{propo}

\begin{rem}
    The method then consists of computing the precise sequence of exponents $\nu_i(n)$ for all components $\left\{M_i=0\right\}\subset\mathcal{C}(f)$ contracted to points by $f$. To be able to do that we need to deduce equations that relate the exponents $\nu_i(n)$ with numbers associated to the top-level blow-ups (see below) needed to reach algebraic stability, these numbers are called local indices by \cite{alonso2023}and we define them next. 
\end{rem}

 Let $f:\mathbb{P}^2\dashrightarrow\mathbb{P}^2$ be birational with $\{M_i=0\}$ being a component of $\mathcal{C}(f)$ contracted to a point $\tilde{p}$ by the map $f$. Let $\phi=\pi_n\circ\dots\circ\pi_1:X_{\tilde{p}}\to \mathbb{P}^2$ be the projection corresponding to the last blowup over the point $\tilde{p}$ needed to reach algebraic stability (we call this blowup top-level for the point $\tilde{p}$). Without loss of generality assume that $\tilde{p}\in\{x_0\neq 0\}$, notice that $\phi$ can be expressed as a map in local coordinates $\phi:(U_1, U_2)\mapsto [1\colon a_1\colon a_2]$ with exceptional line $\{U_1=0\}$ corresponding to the last simple blow-up $\pi_1$.
\begin{defi}\label{def: Localindex}
   Let $P$ be a homogeneous polynomial. Define the \textbf{local index} of $P$ associated to blow-up $\phi$ as the largest number $\nu_\phi(P)$ such that $U_1^{\nu_\phi(P)}$ divides $P\circ \phi$.
\end{defi}

\begin{defi}
 Let $\{M_i=0\}\mapsto\tilde{p}\in\mathbb{CP}^2$. We call $\tilde{p}$ a singularity. We say that the singularity is resolved if for a top level blow-up $\phi$ the exceptional divisor is not contracted to a point under the lift of $f$ to the blow-up space.
\end{defi}

\begin{rem}
    If we need only $1$ blow-up to resolve the singularity at $\tilde{p}$ then $\nu_\phi(P)$ is equal to the multiplicity of point $\tilde{p}$ on hypersurface $\{P=0\}$.
\end{rem}

We highlight that for the maps $\mu_{p,q}$ all components of $\mathcal{C}(\mu_{p,q})$ are mapped either to indeterminate points or to fixed points (See  \cref{subsec:stable_model_ss,subsec:stable_model_nss}). Also all singularities except one require a single-point blow-up only to reach algebraic stability. The remaining singularity requires two consequent single-point blow-ups. 

In \cite{alonso2023} there is an extensive characterization of the different equations relating the $\nu_i$'s and local indices that can occur depending on the type of singularity. As we highlighted before, the realization of the stable model for $\mu_{p,q}$  is not convoluted so we will not explore cases beyond our setting. In particular, we are only interested in the cases such that the system of
recurrent relations is finite. This happens if for each $\tilde{p}\in \mathcal{I}(f^{-1})$ the following alternative holds (cf. Section $6$ in \cite{alonso2023}),

\begin{enumerate}\label{enum}
    \item The top exceptional divisor $E$ of blowup $\phi$ over the point $\tilde{p}$ is mapped birationally to some irreducible component of $\mathcal{C}(f^{-1})$. In this case there is the following relation between exponents and local indices (see the \textbf{case $(3a)$} in \cite{alonso2023}):
    \begin{equation}\label{eqn:case 3a}
        \nu_\phi(P_{n+1})=s\deg(P_n)-\nu_1(n)\nu_\phi(M_1)-\dots-\nu_k(n)\nu_\phi(M_k).
    \end{equation}

    Here $s\in\Z_{>0}$ is the largest number such that $U_1^s$ divides all the components of $\hat{f}\circ \phi.$ 
    \item One (and then any) curve $\{R=0\}\subset f^{-1}(\tilde{p})$ is not degree lowering. 
    Then $$\nu_{\phi}(P_n)=0.$$ 
\end{enumerate}

We also highlight the following trivial but nevertheless important equation

\begin{equation}\label{eqn: trivial relation}
    \deg(P_{n+1})=d \cdot \deg(P_n)-\nu_1(n)\deg(M_1)-\dots-\nu_k(n)\deg(M_k),
\end{equation}
where $d = \deg(f)$ is the algebraic degree of $f$.

We end this section by summarizing the idea of the next subsections and the end goal for each of them. In  \cref{subsec:local_ind_ss,subsec:local_ind_nonss} we compute the dynamical degree of the maps $\mu_{p,q}$ for semi-simple pairs and non semi-simple pairs respectively using the method of local indices. The key for accomplishing a finite system of recurrence relations is to find proportional relations between the local indices $\nu_{\varphi}(P_n)$, $\nu_{\beta_q}(P_n)$, $\nu_{\beta}(P_n)$ $\nu_{\alpha}(P_n)$, $\nu_{\pi_1}(P_n)$ (See  \cref{subsec:stable_model_ss,subsec:stable_model_nss} for the description of these blow-ups) and the powers of the degree lowering curves $\nu_i(n)$. For instance we show that for both the semi-simple and non-semi-simple case $\nu_{\varphi}(P_n)$ corresponds to the greatest power of the polynomial $K^h=\prod_{\omega\in R_p}\left(K^h_{\omega}\right)$ in the factorization of $\mu_{p,q}^{*}P_n$. Similar relations are found for the other blowups. This resembles the examples in works \cite{alonso2023} and \cite{alonso3D}. For non-semi-simple pairs the computation is more subtle as we reach a stable model but the blowup over $[0\colon1\colon 0]$ does not satisfy any of the conditions of the two alternatives above. In particular, we can not apply Equation \eqref{eqn:case 3a} to this singularity. Nevertheless we find relations that reduce the system of recurrent equations.

\subsection{Semi-simple pairs.}\label{subsec:local_ind_ss}
For the semi-simple pairs, notice that we can disregard critical curves $\{x_0=0\}$, $\{x_1=0\}$ and $\{x_2=0\}$ as they are not degree lowering curves by  \cref{rem: fixed point}. Then by Equation \eqref{eqn: Pn sequence} we have

\begin{equation}\label{eqn: Components-semi-simple}
    \mu_{p,q}^{*}P_n=\prod_{\omega\in R_p} \left(K^h_{\omega}\right)^{\nu_{\omega}(n)}\prod_{\zeta\in R_q} \left(K^v_\zeta\right)^{\nu_\zeta(n)} P_{n+1},
\end{equation}
where we abuse notation of \cref{subsec:indeterminacy} by calling the curves and the equations defining them as $K^h_\omega$ and $K^v_\zeta$.

Let $\omega\in R_p$. Notice that $K^h_{\omega^{-1}}$ is the only degree lowering component of $\mathcal{C}(\mu_{p,q})$ passing through $[\omega\colon 0\colon 1]$ with multiplicity 1. Then by \eqref{eqn:case 3a} we obtain:

\begin{equation*}
    \nu_\varphi(P_{n+1})=\deg(P_n)-\nu_{\omega}(n)\nu_\varphi(K^h_{\omega})=\deg(P_n)-\nu_{\omega}(n).
\end{equation*}
Note that for different choices of $\omega\in R_p$ the blowup is represented by the same map $\varphi$ (see \cref{subsubsec:comp_Sp}). In particular, $\nu_{\omega}(n)$ does not depend on the particular $\omega\in R_p$. For simplicity, we use the notation $\nu_1(n) = \nu_{\om}(n),$ for any $\om \in R_p$. Therefore, we have

\begin{equation}\label{eqn: local indices 1}
    \nu_\varphi(P_{n+1})=\deg(P_n)-\nu_1(n).
\end{equation}

Similarly, let $\zeta \in R_q$. Then $K^v_{\zeta^{-1}}$ intersects the set $T_q$ of blow-up points on the vertical axis by the point $[0\colon\zeta\colon 1]$ with multiplicity 1. Therefore $\nu_{\beta_q}(K^v_{\zeta^{-1}})=1$ and $\nu_{\beta_q}(K^v_{\tilde{\zeta}^{-1}})=0$ for any $\tilde{\ze}\in R_q\backslash\{\zeta\}$. Also, every $K^h_\omega$ passes through $[0\colon\zeta\colon 1]$ with multiplicity 1, hence:
\begin{equation*}
    \nu_{\beta_q}(P_{n+1})=p\deg(P_n)-p\nu_1(n)-\nu_{\zeta}(n).
\end{equation*}

Analogously as before, since the LHS is independent of the choice of $\zeta \in R_q$ then every $\nu_\zeta(n)$ is the same. Hence, we denote $\nu_2=\nu_\zeta$ for any $\zeta\in R_q$. Then we have

\begin{equation}\label{eqn: local indices 2}
    \nu_{\beta_q}(P_{n+1})=p\deg(P_n)-p\nu_1(n)-\nu_{2}(n).
\end{equation}

Let $K^h=\prod_{\omega\in R_p}\left(K^h_{\omega}\right)=x_0^px_2^{p(q-1)} + (x_1^q + x_2^q)^p$ and $K^v=\prod_{\zeta\in R_q} \left(K^v_\zeta\right)=x_1^q+x_2^q$, and notice that with this we can improve equation \eqref{eqn: Components-semi-simple} to get: 

\begin{equation}\label{eqn: components semi-simple 2}
    \mu_{p,q}^{*}P_n=P_{n+1}\left(K^h\right)^{\nu_{1}(n)}\left(K^v\right)^{\nu_{2}(n)}.
\end{equation}
\begin{rem}\label{Rem: local Indices degree=q*local}
    As by \eqref{eqn: trivial relation} we have $\deg(P_{n+1})=pq\deg(P_n)-pq\nu_{1}(n)-q\nu_{2}(n)$, then by \eqref{eqn: local indices 2} we obtain:
    \begin{equation}\label{eq: beta-degree}
        \deg(P_n)=q\nu_{\beta_q}(P_n)\,\,\,\mbox{for all}\,\, n\geq 1.
    \end{equation}
  
\end{rem}

Due to the behavior of the lift of $\mu_{p,q}$ on the critical curves it is expected that $\nu_\varphi(P_n)=\nu_1(n)$ and  $\nu_{\beta_q}(P_n)=\nu_2(n)$. This is true and we give algebraic proofs for the sake of completeness.

\begin{lemma}\label{lemma: reductions local indices 1.}
    For all $n\in\Z_{>0}$ we have:
    \begin{itemize}
        \item[(i)] $\nu_\varphi(P_n)=\nu_1(n)$
        \item[(ii)] $\nu_{\beta_q}(P_n)=\nu_2(n)$
    \end{itemize}
\end{lemma}
\begin{proof}

        \textbf{(i)} Given a homogeneous polynomial $P$ of degree $d$, such that its local index under the blow up $\varphi$ is $\nu_\varphi(P)=a$, then there exists a polynomial $\tilde{P}(r_0,s_0)$ with $\tilde{P}(\omega,s_0)\not\equiv 0$ for all $\omega\in R_p$ such that:
        
\begin{equation*}
    P(x,y,1)=(r_0^p+1)^{a}\tilde{P}(r_0,s_0)=(x^p+1)^{a}\tilde{P}\left(x,\frac{y}{x^p+1}\right)
\end{equation*}
where $x=\frac{x_0}{x_2}$, $y=\frac{x_1}{x_2}$ and we are on the blow-up chart $(r_0,s_0)$ (see \cref{subsubsec:comp_Sp}). In homogeneous coordinates we have:

\begin{equation*} 
P(x_0,x_1,x_2)=x_2^dP\left(\frac{x_0}{x_2},\frac{x_1}{x_2},1\right)=x_2^d\left(\frac{x_0^p}{x_2^p}+1\right)^{a}\tilde{P}\left(\frac{x_0}{x_2},\frac{x_1x_2^{p-1}}{x_0^p+x_2^p}\right).
\end{equation*}

Let $\mu_{p,q}[x_0\colon x_1\colon x_2]=[X_0\colon X_1\colon X_2]$ then $\mu_{p,q}^{*}P=P(X_0,X_1,X_2).$ Thus,

\begin{equation*}
    \mu_{p,q}^{*}P=x_0^{p(d-a)}x_1^dx_2^{(d-a)p (q - 1) - d}\left(K^h\right)^a\cdot\tilde{P}\left(\frac{x_1^q + x_2^q}{x_0x_2^{q - 1}},\frac{x_2}{x_1}\right)
\end{equation*}
This implies that $\mu_{p,q}^{*}P$ is divisible by $\left(K^h\right)^a$. It is not divisible by any greater power of $K^h$ because otherwise this will imply that $\tilde{P}\left(\frac{x_1^q + x_2^q}{x_0x_2^{q - 1}},\frac{x_2}{x_1}\right)|_{K^h_\omega}=\tilde{P}(\omega,\frac{x_2}{x_1})$ is identically zero which is not the case. We proved that $\nu_\varphi(P_n)=\nu_1(n)$ for every $n\in\Z_{>0}$.

\textbf{(ii)} If the local index under the blow up $\beta_q$ is $\nu_{\beta_q}(P)=b$, then there exists a polynomial $\overline{P}(u_0,v_0)$ with $\overline{P}(\zeta,v_0)\not\equiv 0$ for all $\zeta\in R_q$ such that:
        
\begin{equation*}
    P(x,y,1)=(u_0^q+1)^{b}\overline{P}(u_0,v_0)=(y^q+1)^{b}\overline{P}\left(y,\frac{x}{y^q+1}\right)
\end{equation*}
This is working on the blow-up chart $(u_0,v_0)$ (cf. \cref{subsubsec:comp_Tq}). In homogeneous coordinates:

\begin{equation*} 
P(x_0,x_1,x_2)=x_2^dP\left(\frac{x_0}{x_2},\frac{x_1}{x_2},1\right)=x_2^d\left(\frac{x_1^q}{x_2^q}+1\right)^{b}\overline{P}\left(\frac{x_1}{x_2},\frac{x_0}{x_2\left(\frac{x_1^q}{x_2^{q}}+1\right)}\right)
\end{equation*}
In a similar way, one can compute that

\begin{equation*}
   \mu_{p,q}^{*}P= (x_1^q+x_2^q)^b \overline{P}(U,V) Q(x_0,x_1,x_2) 
\end{equation*}
where $Q$ is some polynomial which is not divisible by $x_1^q+x_2^q$ and $U$ and $V$ are rational functions such that  $\overline{P}(U,V)|_{{K^v_\zeta}}=\overline{P}(\zeta,\frac{-x_2}{x_0})$ for $\zeta\in R_q$ which is not identically zero by hypothesis.

Then $\mu_{p,q}^{*}P$ is divisible by $(x_1^q+x_2^q)^b$ but not for any greater power of it, in particular for all $n\in\Z_{>0}$ we have $\nu_{\beta_q}(P_n)=\nu_2(n).$

\end{proof}

Now, with these reductions we can solve the system of recurrent relations to find a recurrence for the sequence $(\deg(P_n))_{n\in\Z_{>0}}$. Moreover, by \cref{Prop: generic Alonso} and the method of characteristic root we can deduce a closed formula for $\deg(\mu_{p,q}^n)$. Let  $\lambda_-<\lambda_+$ be the roots of the quadratic polynomial $t^2+(2-pq)t+1=0$. We recall that $\lambda_{\pm}$ were mentioned in \cref{sec:picard}. The formula for $\deg(\mu_{p,q}^n)$ is interesting by itself, so we include it in the following theorem. 

\begin{theo}\label{theo: recurrenceformulasemisimplecase}
    The following relations holds  for all semi-simple pairs $(p,q)$ and $n\in\Z_{>0}\cup \{0\}$
        \begin{equation}\label{eqn: Recursion degree semi-simple}
     \deg(P_{n+2})+(2-pq)\deg(P_{n+1})+\deg(P_n)=0,   
    \end{equation}
    \begin{equation}\label{eqn: Closed formula degree mu}
    \deg(\mu^n_{p,q})=\frac{\lambda^n_{-}\left(pq-\lambda_{+}\right)-\lambda^n_{+}\left(pq-\lambda_{-}\right)}{\lambda_{-}-\lambda_{+}}.
\end{equation}
   
\end{theo}
\begin{proof}
   By  \cref{lemma: reductions local indices 1.} together with the equations \eqref{eqn: local indices 1}, \eqref{eqn: local indices 2} and \eqref{eq: beta-degree} we have the system of recurrence relations:
   \begin{align}
       \nu_1(n+1)&=q\nu_2(n)-\nu_1(n),\\
       \nu_2(n+1)&=(pq-1)\nu_2(n)-p\nu_1(n).
   \end{align}
  This implies
  \begin{equation}
      \nu_2(n+2)+(2-pq)\nu_2(n+1)+\nu_2(n)=0
  \end{equation}
which by  \eqref{eq: beta-degree} is equivalent to Equation \eqref{eqn: Recursion degree semi-simple}. By equation \eqref{eqn: Recursion degree semi-simple} for every $n\in\Z_{>0}$ we have
  
\begin{equation}
    \deg(P_n)=a\lambda^n_-+b\lambda^n_+
\end{equation}
Now, as $\deg(P_0)=1$, $\deg(P_1)=\deg(\mu_{p,q})=pq$,
 $a$ and $b$ are the solutions of the system of equations $a+b=1$, $a\lambda_{-}+b\lambda_+=pq$. Due to  \cref{Prop: generic Alonso} $\deg(\mu^n_{p,q})=\deg(P_n)$ and the proof is done.

\end{proof}

Our goal for this subsection was to compute the dynamical degree for the semi-simple pairs using the method of local indices. Notice that equation \eqref{eqn: Closed formula degree mu} allow us to prove  \cref{Theorem Dynamical degree} (for $(p,q)$ semi-simple) as a corollary. Now we only need to show an analogous recurrence for the non-semi-simple pairs to complete the proof.

\subsection{Non-semi-simple pairs.}\label{subsec:local_ind_nonss} Notice that for the non-semi-simple pairs  all the components of $\mathcal{C}(\mu_{p,1})$ are degree lowering curves by Lemmas \ref{Lemma: Indeterminate set} and \ref{Lemma1: Critical set}. For the local indices using the same reasoning  as for semi-simple pairs we obtain that local indices corresponding to the points in $S_p$ are all the same. We denote them by $\nu_4(n)$. Let $K^h=(x_0^p + (x_1 + x_2)^p)$ ($K^h$ coincides with the $K^h$ from \cref{subsec:local_ind_ss} under the specification $q=1$) and by \eqref{eqn: Pn sequence} we have the pull-back decomposition

\begin{equation}\label{57}
    \mu_{p,1}^{*}P_n=P_{n+1}\cdot x_0^{\nu_1(n)}x_1^{\nu_2(n)}(x_1+x_2)^{\nu_{3}(n)}(K^h)^{\nu_{4}(n)}.
\end{equation}
Notice that $K^h_{\omega^{-1}}$ intersects $S_p$ only at point $[\omega:0:1]$, as before with multiplicity 1. For the other components of $\mathcal{C}(\mu_{p,1})$ only $\{x_1=0\}$ passes through $[\omega:0:1]$, with $x_1\circ\varphi=s_0(r_0^p+1)$, hence, $\nu_\varphi(x_1)=1$. 
By \eqref{eqn:case 3a} we have
\begin{equation}\label{eqn:local indices varphi non-semi-simple}
        \nu_{\varphi}(P_{n+1})=\deg(P_n)-\nu_2(n)-\nu_4(n).
    \end{equation}
Notice now that the only degree lowering curve not passing through $[0\colon -1\colon 1]$ is $\{x_1=0\}$. Also that $K^h_\omega\circ \beta=u(1-\omega^{-1} v)$ (see Section \ref{Section: u,v coordinates} for these blow-up coordinates), so for all $\omega\in R_p$ we have $\nu_{\beta}(K^h_\omega)=1$, also $(x_1+x_2)\circ\beta=uv$ and $x_0\circ \beta=u$, then $\nu_{\beta}(x_0)=1=\nu_{\beta}(x_1+x_2)$. We conclude by \eqref{eqn:case 3a} that

\begin{equation}\label{eqn:local indices beta non-semi-simple}
        \nu_{\beta}(P_{n+1})=p\deg(P_n)-p\nu_4(n)-\nu_1(n)-\nu_3(n).
    \end{equation}
Besides this we also have the degree equation 

\begin{equation}\label{eqn: trivial non-semisimple}
    \deg(P_{n+1})=(p+1)\deg(P_n)-\nu_1(n)-\nu_2(n)-\nu_3(n)-p\nu_4(n).
\end{equation}

The following reductions are completely analogous to the equations in  \cref{lemma: reductions local indices 1.}. The proofs are completely analogous also. We highlight however, that for the blow-up $\beta$ we work on the second chart for $\beta_{q}$ when $q=1$, meaning the coordinates in Section \ref{Section: u,v coordinates}.

\begin{lemma}\label{lemma:reductions_non-semi-simple} For all $n\in\Z_{>0}$ we have
\begin{itemize}
    \item[(i)] $\nu_\varphi(P_n)=\nu_4(n)$,
    \item[(ii)] $\nu_\beta(P_n)=\nu_3(n)$.
\end{itemize}
    
\end{lemma}
\begin{proof}
 \begin{itemize}
     \item[(i)]  Analogous to item (i) of  \cref{lemma: reductions local indices 1.}, apply $\mu_{p,1}$ to the equation for $P(x_0,x_1,x_2)$ and the same reasoning follows.

 \item[(ii)] Analogous to item (ii) of  \cref{lemma: reductions local indices 1.} using chart $(u,v)$ defined in Section \ref{Section: u,v coordinates}.

 \end{itemize}   
\end{proof}
Notice however that we do not have equations for the local indices related to the two blow-ups at $[0:1:0]$ (cf. \cref{ssec: coordinates for [0:1:0]} for blow-ups $\pi_1$ and $\alpha$). We expect however that $\nu_\alpha(P_{n+1})$ or $\nu_{\pi_1}(P_{n+1})$  were related to $\nu_1(n)$ or $\nu_2(n)$ (similar to \cref{lemma:reductions_non-semi-simple}). Also, notice that we cannot use \eqref{eqn:case 3a} since the exceptional line for this blowup is mapped to the fixed point $u_2=0=v_2$ (cf. \eqref{Equation Other super attracting} for blow-up coordinates). Nevertheless, we find relations on the local indices associated to this blowup. We remark that these properties are more subtle and are closely related to the particular behavior of the map $\mu_{p,1}$. A naive way to deduce these relations is to detect them at first by running iterations of the proper pullbacks of $P_0=c_0x_0+c_1x_1+c_2x_2$ and then generalizing. We summarize and prove our findings in the following lemma.

\begin{lemma}\label{Reductions_nu2_nu1} For all $n\in\Z_{>0}$ and $P_0$ generic we have
    \begin{itemize}
     \item[(iii)]$\nu_2(n)=\nu_{\pi_1}(P_n),$
    \item[(iv)]$\nu_1(n)=p\nu_2(n)$ 
    \item[(v)] $\nu_2(n+1)+\nu_2(n)=\deg(P_n).$
    \end{itemize}
\end{lemma}
\begin{proof}
    
   \textbf{(iii)} Use local blow-up coordinates $(u_1,v_1)=(\frac{x_0}{x_1},\frac{x_2}{x_0})$ (cf. \cref{ssec: coordinates for [0:1:0]}). The idea and computation are analogous to the proof of \cref{lemma: reductions local indices 1.}.

\textbf{(iv)} We show by induction on $n$ the statement of (iv) and the formula
\begin{equation}\label{auxiliar formula}
    P_n\circ\pi_1= u_1^{\nu_2(n)}Q_n(u_1,v_1)=u_1^{\nu_2(n)}\left(C_nv_1^{\nu_2(n)}+u_1R_n(u_1,v_1)\right)
\end{equation}
 where the constant $C_n\neq 0$ and $(0,0)$ is a root of polynomial $R_n(u_1,v_1)$ with multiplicity greater or equal than $\nu_2(n)+1$.

\textbf{The base case.} Assume that $P_0=c_0 x_0+c_1x_1+c_2x_2$ is generic, then

\begin{equation*}
  \mu_{p,1}^{*}P_0=c_0x_0^{p-1}x_1(x_1+x_2)+c_1x_2K^h+c_2x_0^px_1=P_1.  
\end{equation*}
By (iii) we have $\nu_2(1)=\nu_{\pi_1}(P_1)$. We compute $\nu_{\pi_1}(P_1)$ directly:

\begin{equation*}
    P_1\circ\pi_1=u_1\left(c_0u_1^{p-2}(u_1v_1+1)+c_1v_1\left(u_1^p+(u_1v_1+1)^p\right)+c_2u_1^{p-1}\right)=u_1Q_1(u_1,v_1)
\end{equation*}
In particular, $Q_1(u_1,v_1)=c_1v_1+u_1R_1(u_1,v_1)$ with $c_1\neq 0$ and (0,0) is a root of $R_1$ with multiplicity greater than 1. We have $\nu_2(1)=1$ also and notice that by applying \eqref{57} to $P_1$ we get $\nu_1(1)=p$.

\textbf{Induction step.} Suppose that for some $n\geq 1$ we have $\nu_1(n)=p\nu_2(n)$ and $Q_n(u_1,v_1)=C_nv_1^{\nu_2(n)}+u_1R_n(u_1,v_1)$. Now, let $d_n := \deg(P_n(x_0, x_1, x_2))$. Denote the coordinates of the point $\mu_{p,1}([x_0:x_1:x_2])$ in the chart $(u_1,v_1)$ by $U_1:=\frac{x_0^{p - 1}x_1(x_1 + x_2)}{x_2\cdot K^h}$ and $V_1:=\frac{x_0}{x_1+x_2}$.
Notice that by hypothesis on the multiplicity of $(0,0)$ for $R_n$ we have that $x_0^{\nu_2(n)}$ divides $Q_n(U_1,V_1)$ but not any greater power of $x_0$. After homogenization  of formula \eqref{auxiliar formula} and computation of the proper pullback of it by \eqref{57} we have

\begin{equation}\label{eqn: auxiliary 2.0}
   P_{n+1}=\frac{x_2^{d_n-\nu_2(n)}\left(K^h\right)^{d_n-\nu_2(n)-\nu_4(n)}(x_1 + x_2)^{\nu_2(n)-\nu_3(n)}Q_n\left(U_1,V_1\right)}{x_0^{\nu_2(n)}},
\end{equation}
then applying \eqref{57} now to $P_{n+1}$ we obtain

\begin{equation*}
    \mu_{p,1}^{*}P_{n+1}=x_0^{p(d_n-\nu_{2}(n))+\nu_{2}(n)(1-p)}x_1^{d_n-2\nu_{2}(n)}(x_1+x_2)^{p(d_n-\nu_{2}(n)-\nu_4(n))-\nu_3(n)}T_1^{d_n-\nu_{2}(n)-\nu_4(n)}T_2^{\nu_{2}(n)-\nu_3(n)}Q_n\left(\tilde{U}_1,\tilde{V}_1\right)
\end{equation*}
The reader can check that $T_1, T_2$ are polynomials that are not divisible by $x_0$ or $x_1$ and

\begin{align*}
\tilde{U}_1 &= \frac{x_0^{(p-1)^2-p}x_1^{p-2}x_2\cdot (K^h)\cdot T_2}{T_1}, \\
\tilde{V_1} &= \frac{x_0^{p - 1}x_1}{T_2}.
\end{align*}
To simplify notation let $\tilde{U}_1=x_0^{(p-1)^2-p}x_1^{p-2}A_1$, $\tilde{V_1}=x_0^{p - 1}x_1B_1$, then

\begin{equation*}
    Q_n(\tilde{U}_1,\tilde{V}_1)=C_nx_0^{\nu_2(n)(p - 1)}x_1^{\nu_2(n)}B_1^{\nu_2(n)}+x_0^{(p-1)^2-p}x_1^{p-2}A_1R_n(\tilde{U}_1,\tilde{V}_1).
\end{equation*}
By induction hypothesis $R_n$ has multiplicity of at least $\nu_2(n)+1$ at $(0,0)$ which implies that $Q_n(\tilde{U}_1,\tilde{V}_1)$ is divisible by $x_0^{\nu_2(n)(p-1)}x_1^{\nu_2(n)}$ and not any greater power of $x_0$ or $x_1$. In particular, $x_0^{p(d_n-\nu_2(n))}$ and $x_1^{d_n-\nu_2(n)}$  both divide $\mu_{p,1}^{*}P_{n+1}$ and they are the greatest such powers of $x_0$ and $x_1$ respectively to do so. Hence,

\begin{equation*}
\nu_1(n+1)=p(d_n-\nu_2(n))=p\nu_2(n+1).
\end{equation*}
Moreover, notice that by pre-composing equation $\eqref{eqn: auxiliary 2.0}$ with $\pi_1$ we have $Q_{n+1}=C_nv_1^{d-\nu_2(n)}+u_1R_{n+1}(u_1,v_1)$ with the required conditions, so $C_n=c_1\neq 0$ for all $n$ and inductively we have the result.

\textbf{(v)} This follows from the proof of (iv) as $\deg(P_n)=d_n$ and $\nu_2(n+1)=d_n-\nu_2(n)$.
    
\end{proof}
\begin{coro}\label{Reductions of nu3 and nu4 in terms of nu2}
    For all $n\in\Z_{>0}$ we have:
    \begin{itemize}
        \item[(a)]$\nu_4(n+1)+\nu_4(n)=\nu_2(n+1)$
        \item[(b)]$\nu_3(n)=\nu_2(n+1)$
    \end{itemize}
\end{coro}
\begin{proof}

This is a direct algebraic consequence of \cref{lemma:reductions_non-semi-simple,Reductions_nu2_nu1}, and equations \eqref{eqn:local indices varphi non-semi-simple} and \eqref{eqn: trivial non-semisimple}.
\end{proof}
Similarly to \cref{theo: recurrenceformulasemisimplecase} we deduce a recurrence relation for the sequence $(\deg(P_n))_{n\in\Z_{>0}}$ for non-semi-simple pairs and a general formula for the degree of the iterates of $\mu_{p,1}$. 
\begin{theo}\label{propo:degreenon_semisimple}
The following relations holds  for any non-semi-simple pair $(p,1)$ and $n\in\Z_{>0}\cup \{0\}$
        \begin{equation}\label{eqn: Recursion degree non semisimple}
     \deg(P_{n+2})+(2-p)\deg(P_{n+1})+\deg(P_n)=0.   
    \end{equation}
    \begin{equation}\label{eqn: Closed formula degree mu non semisimple}
    \deg(\mu^n_{p,1})=\frac{\lambda^n_{-}\left(p+1-\lambda_{+}\right)-\lambda^n_{+}\left(p+1-\lambda_{-}\right)}{\lambda_{-}-\lambda_{+}}
\end{equation}
 where $\lambda_{-}<\lambda_{+}$ are the roots of polynomial $t^2+(2-p)t+1=0$.
\end{theo}
\begin{proof}
    By  \cref{Prop: generic Alonso} is enough to show the recurrence for $(P_n)_{n\in \Z_{>0}}$ as the same reasoning as in  \cref{theo: recurrenceformulasemisimplecase} follows. For this notice that by  \cref{Reductions_nu2_nu1}, \cref{Reductions of nu3 and nu4 in terms of nu2} and \eqref{eqn: trivial non-semisimple} we have

       \begin{equation*}
 \deg(P_{n+1})= p\nu_4(n+1)  
\end{equation*}

and from \eqref{eqn:local indices varphi non-semi-simple} the recurrence relation follows directly.
 \end{proof}

\begin{rem}
 Notice that \cref{theo: recurrenceformulasemisimplecase}  and \cref{propo:degreenon_semisimple} imply \cref{Theorem Dynamical degree}. This completes the main goal of part \ref{sec: Local Indices} of the appendix.  
\end{rem}

\begin{proof}[Proof of \cref{Theorem Dynamical degree}]

Due to  \cref{theo: recurrenceformulasemisimplecase}  and \cref{propo:degreenon_semisimple} together with the definition of dynamical degree given by equation \eqref{Classical dyn. degree} 

\begin{equation*}
\lambda_1(\mu_{p,q})=\displaystyle\lim_{n\to \infty}\left(\deg(\mu_{p,q}^n)\right)^{1/n}=\lambda_+= \frac{pq-2+\sqrt{(pq-2)^2-4}}{2}.
\end{equation*}

\end{proof}
\bibliographystyle{alpha}
\bibliography{Reference}

\begin{thebibliography}{AdMV06}

\bibitem[AdMV06]{AMV}
J.-Ch. Angl\`es~d'Auriac, J.-M. Maillard, and C.~M. Viallet.
\newblock On the complexity of some birational transformations.
\newblock {\em J. Phys. A}, 39(14):3641--3654, 2006.

\bibitem[ASW23a]{alonso2023}
Jaume Alonso, Yuri~B. Suris, and Kangning Wei.
\newblock Dynamical degrees of birational maps from indices of polynomials with respect to blow-ups {I}. general theory and 2{D} examples, 2023.
\newblock ArXiv preprint: \url{https://arxiv.org/abs/2303.15864}.

\bibitem[ASW23b]{alonso3D}
Jaume Alonso, Yuri~B. Suris, and Kangning Wei.
\newblock Dynamical degrees of birational maps from indices of polynomials with respect to blow-ups {II}. 3{D} examples, 2023.
\newblock ArXiv preprint: \url{https://arxiv.org/abs/2307.09939}.

\bibitem[BD15]{DESERTI_BLANC}
J\'{e}r\'{e}my Blanc and Julie D\'{e}serti.
\newblock Degree growth of birational maps of the plane.
\newblock {\em Ann. Sc. Norm. Super. Pisa Cl. Sci. (5)}, 14(2):507--533, 2015.

\bibitem[BDJ20]{BDJ}
Jason~P. Bell, Jeffrey Diller, and Mattias Jonsson.
\newblock A transcendental dynamical degree.
\newblock {\em Acta Math.}, 225(2):193--225, 2020.

\bibitem[BGM18]{bershtein2018cluster}
Mikhail Bershtein, Pavlo Gavrylenko, and Andrei Marshakov.
\newblock Cluster integrable systems, q-{P}ainlev{\'e} equations and their quantization.
\newblock {\em Journal of High Energy Physics}, 2018(2), 2018.

\bibitem[Bir25]{BIRKETT}
Richard A.~P. Birkett.
\newblock On the stabilisation of rational surface maps.
\newblock {\em Ann. Fac. Sci. Toulouse Math. (6)}, 34(1):47--74, 2025.

\bibitem[CL24]{CHEN202425}
Zhichao Chen and Zixu Li.
\newblock Mutation invariants of cluster algebras of rank 2.
\newblock {\em Journal of Algebra}, 654:25--58, 2024.

\bibitem[DDG10]{DillerDynamicsMeromorphicMaps2009b}
Jeffrey Diller, Romain Dujardin, and Vincent Guedj.
\newblock Dynamics of meromorphic maps with small topological degree {III}: geometric currents and ergodic theory.
\newblock {\em Ann. Sci. \'{E}c. Norm. Sup\'{e}r. (4)}, 43(2):235--278, 2010.

\bibitem[DDG11]{DDG2}
Jeffrey Diller, Romain Dujardin, and Vincent Guedj.
\newblock Dynamics of meromorphic mappings with small topological degree {II}: {E}nergy and invariant measure.
\newblock {\em Comment. Math. Helv.}, 86(2):277--316, 2011.

\bibitem[DF01]{diller_favre_2001}
J.~Diller and C.~Favre.
\newblock Dynamics of bimeromorphic maps of surfaces.
\newblock {\em Amer. J. Math.}, 123:1135--1169, 2001.

\bibitem[DGL23]{DGL}
N.-B. Dang, R.~Grigorchuk, and M.~Lyubich.
\newblock Self-similar groups and holomorphic dynamics: renormalization, integrability, and spectrum.
\newblock {\em Arnold Math. J.}, 9(4):505--597, 2023.

\bibitem[Dil96]{MR1422105}
Jeffrey Diller.
\newblock Dynamics of birational maps of {${\bf P}^2$}.
\newblock {\em Indiana Univ. Math. J.}, 45(3):721--772, 1996.

\bibitem[Dil01]{dillerInvariantMeasureLyapunov2001}
Jeffrey Diller.
\newblock Invariant measure and {L}yapunov exponents for birational maps of {${\bf P}^2$}.
\newblock {\em Comment. Math. Helv.}, 76(4):754--780, 2001.

\bibitem[DL16]{diller_lin_2016}
J.~Diller and Jan-Li Lin.
\newblock Rational surface maps with invariant meromorphic two-forms.
\newblock {\em Math. Ann.}, 364:313–352, 2016.

\bibitem[DN11]{Dinh-Nguyen1}
Tien-Cuong Dinh and Vi\^et-Anh Nguy\^en.
\newblock Comparison of dynamical degrees for semi-conjugate meromorphic maps.
\newblock {\em Comment. Math. Helv.}, 86(4):817--840, 2011.

\bibitem[DNT12]{Dinh-Nguyen2}
Tien-Cuong Dinh, Vi\^et-Anh Nguy\^en, and Tuyen~Trung Truong.
\newblock On the dynamical degrees of meromorphic maps preserving a fibration.
\newblock {\em Commun. Contemp. Math.}, 14(6):1250042, 18, 2012.

\bibitem[DR25]{DR2}
Jeffrey Diller and Roland Roeder.
\newblock Energy, equilibrium measure and entropy for toric surface maps, 2025.
\newblock ArXiv preprint: \url{https://arxiv.org/pdf/2509.04278}.

\bibitem[DR26]{DR1}
Jeffrey Diller and Roland Roeder.
\newblock Equidistribution without stability for toric surface maps.
\newblock {\em Comment. Math. Helv.}, 101(1):115--192, 2026.

\bibitem[DS05]{dinhUpperBoundTopological2006}
Tien-Cuong Dinh and Nessim Sibony.
\newblock Une borne sup\'{e}rieure pour l'entropie topologique d'une application rationnelle.
\newblock {\em Ann. of Math. (2)}, 161(3):1637--1644, 2005.

\bibitem[Duj06]{DUJARDIN_LAMINAR}
Romain Dujardin.
\newblock Laminar currents and birational dynamics.
\newblock {\em Duke Math. J.}, 131(2):219--247, 2006.

\bibitem[Fri91]{FRIEDLAND}
Shmuel Friedland.
\newblock Entropy of polynomial and rational maps.
\newblock {\em Ann. of Math. (2)}, 133(2):359--368, 1991.

\bibitem[FZ02]{FOMIN1}
Sergey Fomin and Andrei Zelevinsky.
\newblock Cluster algebras. {I}. {F}oundations.
\newblock {\em J. Amer. Math. Soc.}, 15(2):497--529, 2002.

\bibitem[FZ03]{FOMIN2}
Sergey Fomin and Andrei Zelevinsky.
\newblock Cluster algebras. {II}. {F}inite type classification.
\newblock {\em Invent. Math.}, 154(1):63--121, 2003.

\bibitem[GH83]{GH}
John Guckenheimer and Philip Holmes.
\newblock {\em Nonlinear oscillations, dynamical systems, and bifurcations of vector fields}, volume~42 of {\em Applied Mathematical Sciences}.
\newblock Springer-Verlag, New York, 1983.

\bibitem[Gue05]{GUEDJ}
Vincent Guedj.
\newblock Entropie topologique des applications m\'{e}romorphes.
\newblock {\em Ergodic Theory Dynam. Systems}, 25(6):1847--1855, 2005.

\bibitem[HK23]{HoneKouloukas}
Andrew N.~W. Hone and Theodoros~E. Kouloukas.
\newblock Deformations of cluster mutations and invariant presymplectic forms.
\newblock {\em J. Algebraic Combin.}, 57(3):763--791, 2023.

\bibitem[HPV97]{Hu1997}
John Hubbard, Peter Papadopol, and Vladimir Veselov.
\newblock A compactification of {H}\'enon mappings in {$\C^2$} as dynamical systems, 1997.
\newblock ArXiv preprint: \url{https://arxiv.org/abs/math/9709227}.

\bibitem[IK21]{IK}
Tsukasa Ishibashi and Shunsuke Kano.
\newblock Algebraic entropy of sign-stable mutation loops.
\newblock {\em Geom. Dedicata}, 214:79--118, 2021.

\bibitem[Kau24]{KAUFMAN}
Dani Kaufman.
\newblock Mutation invariant functions on cluster ensembles.
\newblock {\em J. Pure Appl. Algebra}, 228(2):Paper No. 107495, 25, 2024.

\bibitem[KH95]{Katok_Hasselblatt}
Anatole Katok and Boris Hasselblatt.
\newblock {\em Introduction to the modern theory of dynamical systems}, volume~54 of {\em Encyclopedia of Mathematics and its Applications}.
\newblock Cambridge University Press, Cambridge, 1995.
\newblock With a supplementary chapter by Katok and Leonardo Mendoza.

\bibitem[KR17]{kaschner2017complex}
Scott Kaschner and Roland Roeder.
\newblock Complex perspective for the projective heat map acting on pentagons.
\newblock {\em Conformal Geometry and Dynamics of the American Mathematical Society}, 21(9):247--263, 2017.

\bibitem[KSa]{kawaguchiExamplesDynamicalDegree2014}
Shu Kawaguchi and Joseph~H. Silverman.
\newblock Examples of dynamical degree equals arithmetic degree.
\newblock 63(1).

\bibitem[KSb]{kawaguchiDynamicalArithmeticDegreesof2016}
Shu Kawaguchi and Joseph~H. Silverman.
\newblock On the dynamical and arithmetic degreesof rational self-maps of algebraic varieties.
\newblock 2016(713):21--48.

\bibitem[Lam16]{LAMPE}
Philipp Lampe.
\newblock Diophantine equations via cluster transformations.
\newblock {\em J. Algebra}, 462:320--337, 2016.

\bibitem[MO24]{machacek2024discrete}
John Machacek and Nicholas Ovenhouse.
\newblock Discrete dynamical systems from real valued mutation.
\newblock {\em Experimental Mathematics}, 33(2):261--275, 2024.

\bibitem[Noo]{FRACTALSTREAM}
Matt Noonan.
\newblock Fractalstream computer software.
\newblock \url{https://code.google.com/archive/p/fractalstream/}.

\bibitem[Roe15]{Roeder}
Roland K.~W. Roeder.
\newblock The action on cohomology by compositions of rational maps.
\newblock {\em Math. Res. Lett.}, 22(2):605--632, 2015.

\bibitem[Sha94]{shafarevich1994basic}
Igor~R. Shafarevich.
\newblock {\em Basic algebraic geometry. 1}.
\newblock Springer-Verlag, Berlin, second edition, 1994.
\newblock Varieties in projective space, Translated from the 1988 Russian edition and with notes by Miles Reid.

\end{thebibliography}
 
\end{document}